\documentclass[reqno]{amsart}
\usepackage{amssymb}
\usepackage{amsmath}
\usepackage{amsthm}
\usepackage{mathtools}
\usepackage[svgnames]{xcolor}
\usepackage[pdftex]{hyperref}
\usepackage{enumitem}
\usepackage{subcaption}
\usepackage{placeins}
\usepackage{vmargin}
\setmarginsrb {2cm}{2cm}{2cm}{2cm} {1cm}{1cm}{0.5cm}{0.5cm}

\usepackage{tikz,tkz-tab}
\usetikzlibrary{matrix}
\usepackage{pgfplots}
\pgfplotsset{compat=newest} 
\pgfplotsset{plot coordinates/math parser=false}
\usepackage{float}

\hypersetup{
  pdftitle={},
  pdfauthor={Stefan Le Coz <stefan.lecoz@math.cnrs.fr>},
  pdfsubject={},
  pdfkeywords={}, 
  colorlinks=true,
  linkcolor=DarkBlue  ,          
  citecolor=DarkRed,        
  filecolor=DarkMagenta,      
  urlcolor=DarkGreen,           
}

\newtheorem{theorem}{Theorem}[section]
\newtheorem{main}[theorem]{Main result}

\newtheorem{proposition}[theorem]{Proposition}
\newtheorem{lemma}[theorem]{Lemma}

\theoremstyle{definition}
\newtheorem{assumption}[theorem]{Assumption}

\theoremstyle{remark}
\newtheorem{remark}[theorem]{Remark}

\DeclarePairedDelimiter{\norm}{\lVert}{\rVert}
\DeclarePairedDelimiter{\abs}{\lvert}{\rvert}

\newcommand{\scalar}[2]{\left( #1,#2 \right)}

\newcommand{\dual}[2]{\left\langle #1,#2 \right\rangle}

\newcommand{\eps}{\varepsilon}

\newcommand{\R}{\mathbb{R}}

\renewcommand{\leq}{\leqslant}
\renewcommand{\geq}{\geqslant}

\DeclareMathAlphabet{\mathpzc}{OT1}{pzc}{m}{it}
\renewcommand{\Re}{\mathcal R\!\mathpzc{e}}

\usepackage{color}

\begin{document}

\title[Stationary States on Nonlinear Quantum Graphs]{Gradient Flow Approach to the Calculation of Stationary States on Nonlinear Quantum Graphs}

\author[C.~Besse]{Christophe Besse}
\author[R.~Duboscq]{Romain Duboscq}
\author[S.~Le Coz]{Stefan Le Coz}
\thanks{The work of C. B. is partially supported by ANR-17-CE40-0025.
  The work of S. L. C. is 
  partially supported by ANR-11-LABX-0040-CIMI within the
  program ANR-11-IDEX-0002-02 and  ANR-14-CE25-0009-01}

\address[Christophe Besse]{Institut de Math\'ematiques de Toulouse ; UMR5219,
  \newline\indent
  Universit\'e de Toulouse ; CNRS,
  \newline\indent
  UPS IMT, F-31062 Toulouse Cedex 9, 
  \newline\indent
  France}
\email[Christophe Besse]{Christophe.Besse@math.univ-toulouse.fr}

\address[Romain Duboscq]{Institut de Math\'ematiques de Toulouse ; UMR5219,
  \newline\indent
  Universit\'e de Toulouse ; CNRS,
  \newline\indent
  INSA IMT, F-31077 Toulouse, 
  \newline\indent
  France}
\email[Romain Duboscq]{Romain.Duboscq@math.univ-toulouse.fr}

\address[Stefan Le Coz]{Institut de Math\'ematiques de Toulouse ; UMR5219,
  \newline\indent
  Universit\'e de Toulouse ; CNRS,
  \newline\indent
  UPS IMT, F-31062 Toulouse Cedex 9,
  \newline\indent
  France}
\email[Stefan Le Coz]{stefan.lecoz@math.cnrs.fr}





\date{\today}

\begin{abstract}
We introduce and implement a method to compute stationary states of nonlinear Schr\"odinger equations on metric graphs. Stationary states are obtained as local minimizers of the nonlinear Schr\"odinger energy at fixed mass. Our method is based on a normalized gradient flow for the energy (i.e. a gradient flow projected on a fixed mass sphere) adapted to the context of nonlinear quantum graphs. We first prove that, at the continuous level, the normalized gradient flow is well-posed, mass-preserving, energy diminishing and converges (at least locally) towards  stationary states. We then establish the link between the continuous flow and its discretized version. We conclude by conducting a series of numerical experiments in model situations showing the good performance of the discrete flow to compute stationary states. Further experiments as well as detailed explanation of our numerical algorithm are given in a companion paper.
\end{abstract}

\maketitle


\section{Introduction}
\label{sec:introduction}

Partial differential equations on (metric) graphs have a relatively recent history. Recall that a \emph{metric graph} $\mathcal G$ is a collection of vertices $\mathcal V$ and edges $\mathcal E$ with \emph{lengths} $l_e\in(0,\infty]$ associated to each edge $e\in \mathcal E$. 
One of the earliest
account of a partial differential equation set up on metric graphs is the work of Lumer~\cite{Lu80} in 1980 on ramification spaces. Among the early milestones in the development
of the theory of partial differential equations on graphs, one finds the work of Nicaise~\cite{Ni85} on
propagation of nerves impulses. Since then, the theory has known considerable developments, due in particular
to the natural appearance of graphs in the modeling of various
physical situations. One may refer to the survey book~\cite{DaZu06} for a broad
introduction to the study of partial differential equations on networks, with a special emphasis
on control problems.

Among partial differential equations problems set on metric graphs, one has become increasingly
popular: quantum graphs. By quantum graphs, one usually refers to a
metric graph $\mathcal G=(\mathcal V,\mathcal E)$ equipped with a differential operator $H$ often referred to as
the Hamiltonian. The most popular example of Hamiltonian is $-\Delta$ on the edges with Kirchhoff conditions (conservation of charge and current) at the vertices
(see Section~\ref{sec:preliminaries} for a precise definition), where $\Delta$ is the Laplace operator. The
book of Berkolaiko and Kuchment~\cite{BeKu13} provides an excellent
introduction to the theory of quantum graphs.

Recently, another topic has gained an incredible momentum: nonlinear
quantum graphs. By this terminology, we refer to a metric graph $\mathcal G=(\mathcal V,\mathcal E)$
equipped with a nonlinear evolution equation of Schr\"odinger
type
\[
i\partial_tu-Hu+g(|u|^2)u=0,
\]
where $u=u(t,x) \in \mathbb{C}$ is the unknown wave function, $t$
denoting the time variable and $x$ the position on the edges of $\mathcal{G}$.
Whereas the research on linear quantum graphs is mainly focused
on the spectral properties of the Hamiltonian, one of the main area of
investigation for nonlinear quantum graphs is the existence of ground
states, i.e. minimizers of the Schr\"odinger energy $E$ on fixed
mass $M$, where
\[
E(u)=\frac12\dual{Hu}{u}-\frac12\int_{\mathcal G}G(|u|^2),\quad G'=g,\quad M(u)=\norm{u}_{L^2(\mathcal G)}^2.
  \]
Indeed, ground states are considered to be the building blocks
of the dynamics for the nonlinear Schr\"odinger equation, and being
able to obtain them by a minimization process guarantees in particular
their (orbital) stability.

On the theoretical side, the literature concerning ground states on
quantum graphs is already too vast to be shortly summarized. A
perfect introduction to the topic is furnished by the survey paper of
Noja~\cite{No14} and we
only present a few relevant samples.

Among the model cases for graphs, the simplest ones may be
star-graphs, i.e. graphs with one vertex and a finite number of
semi-infinite edges attached to the vertex  (see Figure~\ref{fig:N_star_graph}). 
\begin{figure}[htbp!]
  \centering
\begin{tikzpicture}[scale=0.065]
  \node at (0,0) {$\bullet$};
  \draw (-40,0) -- (40,0);
  \draw (-20,-34.64) -- (20,34.64);
  \draw (-20,34.64) -- (20,-34.64);
  \node at (-24,-34.64) {$\infty$};
  \node at (24,-34.64) {$\infty$};
  \node at (44,0) {$\infty$};
  \node at (24,34.64) {$\infty$};
  \node at (-24,34.64) {$\infty$};
  \node at (-44,0) {$\infty$};
\end{tikzpicture}    
  \caption{Star-graph with $N=6$ edges}
  \label{fig:N_star_graph}
\end{figure}
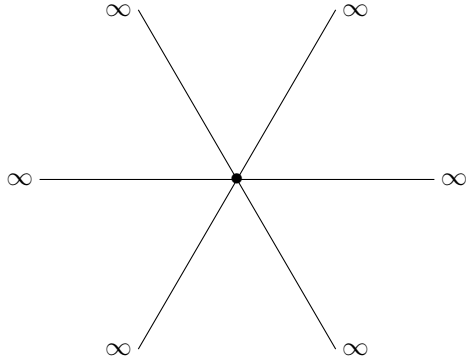
For this type of graphs
with an attractive Dirac type interaction at the vertex,  Adami,
Cacciapuoti, Finco and Noja~\cite{AdCaFiNo14,AdCaFiNo16} established
under a mass condition and for sub-critical nonlinearities the
existence of a (local or global) minimizer of the energy at fixed mass,
with an explicit formula for the minimizer (see Section~\ref{sec:preliminaries} for more
details and explanations). For more general nonlinear quantum graphs,
Adami, Serra and Tilli~\cite{AdSeTi15b,AdSeTi16,AdSeTi17a} have
focused on the case of Kirchhoff-Neumann boundary conditions for
non-compact connected metric graphs with a finite number of edges and
vertices. In particular, they obtained a topological condition
(see Assumption~\ref{ass:H} (H)) under which no ground state exists. On the other hand, in some cases,
metric properties of the graph and the value of the mass constraint
influence the existence or non-existence of the ground state~\cite{AdSeTi16,AdSeTi17a,DOVETTA2020107352,noja2020standing}.

Another particularly interesting study is presented in the work of
Marzuola and Pelinovsky~\cite{MaPe16} for the dumbbell graph. As its
name indicates, the dumbbell graph is made of two circles linked by a
straight edge (see Figure~\ref{fig:dumbell}). It is shown in~\cite{MaPe16} that for small fixed mass, the
minimizer of the energy is a constant. As the mass increases, several
bifurcations for the ground state occur, in particular a symmetric (main part located on the
central edge) and an asymmetric one (main part located on one of the
circles). Numerical experiments (based on Newton's iteration scheme)
complement the theoretical study in~\cite{Go19,MaPe16}.
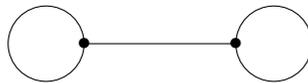
\begin{figure}[htbp!]
  \centering
    \begin{tikzpicture}
      \node at (-1,0) {$\bullet$};
      \node at (1,0) {$\bullet$};
      \draw (-1.5,0) circle (0.5);
      \draw (-1,0) -- (1,0);
      \draw (1.5,0) circle (0.5);
    \end{tikzpicture}
  \caption{Dumbbell graph}
  \label{fig:dumbell}
\end{figure}

Among  the many other interesting recent results on nonlinear quantum graphs, we mention the
flower graphs studied in~\cite{KaMaPeXi20}, graphs with generals
operators and nonlinearities~\cite{Ho19}, periodic graphs~\cite{PeSc17}, etc. 

On the numerical side, however, the literature devoted to nonlinear
quantum graphs is very sparse. 
Finite differences on graphs have been implemented in 
a library developed in Matlab by R. H. Goodman, available in~\cite{GoodmanLib} and which has been used in particular in~\cite{Go19,KaPeGo19}.
The work~\cite{MaPe16} is one of
the rare work containing numerical computation of
nonlinear ground states on graphs.
In our case, we have implemented a finite difference discretization  scheme (see Section \ref{sec:space_disc}) in the framework of the Grafidi library~\cite{Grafidi}, a Python library which we have developed for the numerical simulation on quantum graph and which is presented in the companion paper~\cite{2021_BDL_Arxiv}.

The integrability of the cubic nonlinear Schr\"odinger equation on graphs is analyzed in \cite{SoMaSaSaNa10}, with some numerical simulation and an appendix discussing the discretization at the vertices. The fully discrete (Ablowitz-Ladik type) integrable nonlinear Schr\"odinger is studied in \cite{NaSoMaSa11}. Other model equations on graphs are considered in \cite{SaBaMaKe18,SoBaMaNaUe16}. Extension to transparent vertices conditions is proposed in \cite{YuSaEhMa19a, YuSaAsEhMa20, YuSaEhMa19b}.

Our goal in this paper is to develop numerical tools for the
calculation of local minimizers of the energy at fixed mass $m>0$ in the
setting of generic metric graphs with non necessarily Kirchhoff vertex
boundary conditions.

The numerical method that we have implemented corresponds to a
normalized gradient flow: at each step of time, we evolve in the
direction of the gradient of the energy and renormalize the mass of
the outcome. Such scheme is popular in the physics literature under
the name ``imaginary time method''. One of the earliest mathematical
analysis was performed by Bao and Du~\cite{BaDu04}. More recently, in
the specific case of the nonlinear Schr\"odinger equation on the line $\R$ with
focusing cubic nonlinearity,
Faou and Jezequel~\cite{FaJe18} performed a theoretical analysis of
the various levels of discretization of the method, from the continuous
one to the fully discrete scheme.

At the continuous level, by considering a function $\psi(t,x)$ on $\mathcal{G}$, the normalized gradient flow is given by
  \begin{equation}\tag{CNGF}
    \label{eq:cngf-intro}
    \partial_t\psi=-E'(\psi)+\frac{1}{M(\psi)}\dual{E'(\psi)}{\psi}\psi,
\end{equation}
and we establish in Section~\ref{sec:cngf} the main properties of the flow. This is our first main result, which can be stated in the following informal way.
\begin{main}[see Theorem~\ref{thm:convergence}]
  Under Assumptions~\ref{ass:nonlinearity} and~\ref{ass:coercivity}, 
  the continuous normalized gradient flow is well-posed, mass preserving, energy
  diminishing, and converges locally towards local minimizers. 
\end{main}

Having established the adequate properties of the flow at the continuous level, we turn to the discretization process. As is explained in Section~\ref{sec:discretization}, several time-discretizations are possible, but the so-called Gradient Flow with Discrete Normalization has proven to be very efficient. It consists into the following process to go from $\psi^n$ (an approximation of $\psi(t_n,\cdot)$ at discrete time $t_n$) to $\psi^{n+1}$:
\begin{equation}\tag{GFDN}
  \label{eq:gfdn-intro}
  \left\{
    \begin{aligned} \frac{\varphi^{n+1}-\psi^n}{t_{n+1}-t_n}&=-H\varphi^{n+1}+g(|\psi^n|^2)\varphi^{n+1},\\
      \psi^{n+1}&=\sqrt{m}\frac{ \varphi^{n+1}}{\norm{\varphi^{n+1}}_{L^2}}.
  \end{aligned}
\right.
  \end{equation}
The space discretization can be performed using second order finite differences inside the edges. The values at the vertices are obtained by approximating by finite differences the boundary conditions at the vertices. 
  
In our second main result, we establish the link between the continuous normalized gradient flow and its space-time discretization. 
\begin{main}[see Section~\ref{sec:discretization}]
The Gradient Flow with Discrete Normalization~\eqref{eq:gfdn-intro} is a time-discretization of the continuous normalized gradient flow~\eqref{eq:cngf-intro}. Its space discretization can be obtained by finite differences with a special treatment at the vertices.
\end{main}

Finally, we illustrate by numerical experiments the efficiency of our
technique. We use as test case the $2$-star graph with $\delta$ and $\delta'$ boundary conditions at the vertex connecting the two edges. This test case has been extensively studied from a theoretical point of view (see~\cite{FuJe08,FuOhOz08,LeFuFiKsSi08} for earlier works and~\cite{AdBoRu20} and the references therein for more recent achievements). A sneak peek of the results presented in Section~\ref{sec:experiments} is offered in Figure~\ref{fig:delta_2Edges} where the almost perfect agreement between the theoretical solution and the computed one is shown in the case of a $2$-star graph with attractive $\delta$ condition at the vertex. We also consider other possible types of graphs. Further numerical experiments as well as a detailed presentation of our numerical algorithm are given in the  companion paper~\cite{2021_BDL_Arxiv}.
\begin{figure}[htbp!]
  \centering \includegraphics[width=.38\textwidth]{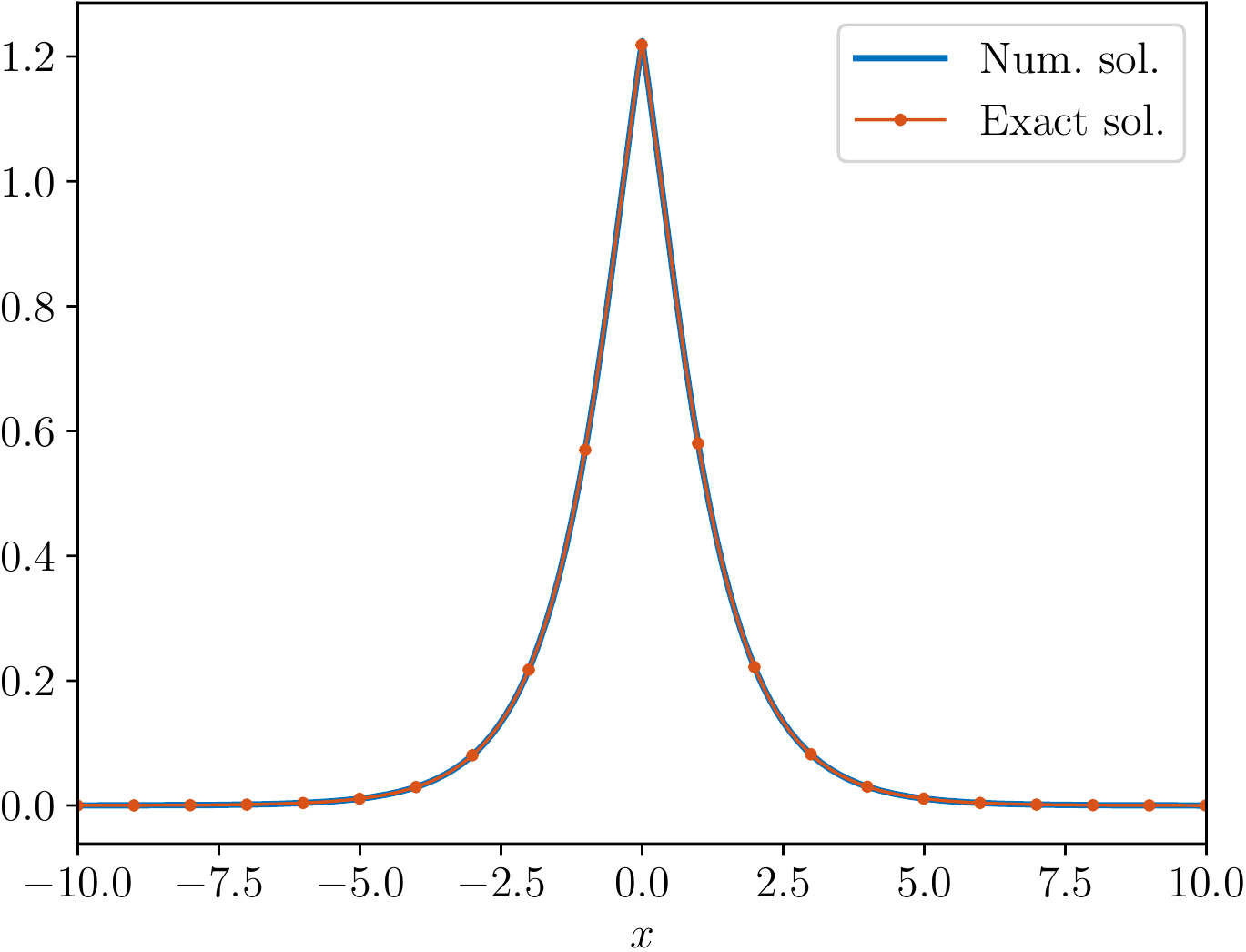}   \caption{Comparison of numerical solution to ground state for $\delta$
    interaction.}
  \label{fig:delta_2Edges}
\end{figure}

Our main achievements in the numerical experiments are summarized in
the following statement.

\begin{main}[see Section~\ref{sec:experiments}]
    The observed convergence of the discretized flow is of order $2$ in
  space. 
  In the test case of a nonlinear Schr\"odinger equation on a star graph
  with two edges and attractive $\delta$ or $\delta'$ interactions at the vertex, the discretized
  flow converges towards the explicitly known ground state. Applicability of the method to generic graphs is illustrated on the sign-post graph and the tower of bubbles graph.  
\end{main}

The rest of this paper is organized in the following way. In Section~\ref{sec:preliminaries}, we present in details the setting in which
we work and give theoretical preliminaries. In Section~\ref{sec:cngf},
we prove that the continuous normalized gradient flow is well-posed, energy
diminishing and converges locally towards a stationary state. In Section~\ref{sec:discretization}, we present the space-time discretization
process of the continuous flow. Finally, numerical experiments in a
test case and in more elaborate settings are presented in Section~\ref{sec:experiments}.

\newcommand{\EnergyOD}{E_{RW}}
\newcommand{\EnergyNorm}{E_{W}}

\section{Preliminaries}
\label{sec:preliminaries}
We start with a few preliminaries to give the precise setting in which we would like to work. 

\subsection{Linear quantum graphs}
Let $\mathcal G$ be a metric graph, i.e. a collection of edges $\mathcal E $ and vertices $\mathcal V$. We assume that $\mathcal G$ 
connected. Two vertices might be connected by several edges and one edge can link a vertex to itself. Each of the edges $e\in\mathcal E $ will be identified with a segment $I_e=[0,l_e]$ if $l_e\in(0,\infty)$ or $I_e=[0,\infty)$ if $l_e=\infty$, where $l_e$ is the (finite or infinite) \emph{length} of the edge. 

A (complex valued) function $\psi:\mathcal G\to\mathbb C$ is a collection of one dimensional maps defined for each edge $e\in\mathcal E $:  
\[
\psi_e:I_e\to\mathbb C.
\]
We define $L^p(\mathcal G)$ and $H^k(\mathcal G)$ by
\[
L^p(\mathcal G)=\bigoplus_{e\in\mathcal E }L^p(I_e),\quad H^k(\mathcal G)=\bigoplus_{e\in\mathcal E }H^k(I_e).
\]
The corresponding norms will be given by
\[
\norm{\psi}_{L^p}^p=\sum_{e\in\mathcal E }\norm{\psi_e}_{L^p(I_e)}^p,\quad \norm{\psi}_{H^k}^2=\sum_{e\in\mathcal E }\norm{\psi_e}_{H^k(I_e)}^2.
  \]
  The scalar product on $L^2(\mathcal G)$ will be given by
  \[
\scalar{\phi}{\psi}_{L^2}=\sum_{e\in\mathcal E }\Re \int_{I_e}\phi_e\bar\psi_edx.
\]
To denote the duality product between $H^1(\mathcal G)$ and its dual we will use the angle brackets:
\[
\dual{\cdot}{\cdot}=\dual{\cdot}{\cdot}_{H^{-1},H^1}.
  \]
  Note that it is common to include in the definition of $H^1(\mathcal G)$ a continuity condition at the vertices. In order to consider more general situations, we do not make this restriction here and we will later instead introduce the space $H_D^1(\mathcal G)$, which corresponds to the Dirichlet part of the compatibility conditions at the vertices (see~\eqref{eq:Dirichlet-Sobolev}). 

Given $u\in H^2(\mathcal G)$ and a vertex $v\in\mathcal V$ of degree $d_v$, define $u(v)\in\R^{d_v}$ as the column vector
\[
u(v)=(u_e(v))_{e\sim v}
\]
where $e\sim v$ denotes the edges incident to the vertex $v$ and $u_e(v)$ is the corresponding limit value of $u_e$. The boundary conditions at the vertex $v$ will be described by 
\[
A_vu(v)+B_vu'(v)=0,
\]
where $A_v$ and $B_v$ are $d_v\times d_v$ matrices and $u'(v)$ is formed with the derivatives along the edges in the outgoing directions. Consider for example the classical \emph{Kirchhoff-Neumann boundary conditions} at the vertex $v$:  we require the \emph{conservation of charge}, i.e.
for all $e$ and $e'$ incident to the same vertex $v$
\[
u_e(v)=u_{e'}(v),
\]
and the \emph{conservation of current}, i.e.
\[
\sum_{e\sim v}u_e'(v)=0.
\]
These conditions are expressed in terms of $A_v$ and $B_v$ by
\begin{equation}
  \label{eq:kirchhoff-neumann}
A_v=
\begin{pmatrix}
  1&-1&&&(0)\\
&1&-1\\
&&\ddots&\ddots\\
&&&1&-1\\
(0)&&&&0
\end{pmatrix},
\quad
B_v=
\begin{pmatrix}
0&\dots&0\\
\vdots&&\vdots\\
0&\dots&0\\
1&\dots&1
\end{pmatrix}.
\end{equation}

For the sake of conciseness, we use the notation 
\[
u(\mathcal V)=(u(v))_{v\in\mathcal V},
\]
for the column vector of all values at the end of the edges and the corresponding boundary conditions  matrices are given by
  \begin{equation*}
A_{\mathcal V}=
\begin{pmatrix}
  A_{v_1}&&(0)\\
&\ddots\\
(0)&&A_{v_V}
\end{pmatrix},
\quad
B_{\mathcal V}=
\begin{pmatrix}
  B_{v_1}&&(0)\\
&\ddots\\
(0)&&B_{v_V}
\end{pmatrix}.
\end{equation*}
The boundary conditions considered are local at the vertices, we refrain here from taking into account more general boundary conditions.

We now define on the graph a second order  unbounded operator $H$ by
\[
    H:D(H)\subset L^2(\mathcal G)\to L^2(\mathcal G)
\]
where the domain of $H$ is given by
\[
D(H)=\{ u\in H^2(\mathcal G): A_{\mathcal V}u(\mathcal V)+B_{\mathcal V}u'(\mathcal V)=0 \} 
\]
and the action of $H$ on $u\in D(H)$ is given by
\[
(Hu)_e=-\partial_{xx}u_e
\]
for every edge $e\in\mathcal E $. We restrict ourselves to self-adjoint operators, which is known to be equivalent for $H$ (see e.g.~\cite[Theorem 1.4.4]{BeKu13}) to request  that at each vertex $v$ the $d_v\times 2d_v$ matrix $(A_v |B_v)$ has maximal rank and the matrix $A_vB_v^*$ is symmetric. In that case, for each vertex $v$ there exist three orthogonal and mutually orthogonal operators $P_{D,v}$ (Dirichlet part), $P_{N,v}$ (Neumann part) and $P_{R,v}=Id-P_{D,v}-P_{N,v}$ (Robin part), acting on $\mathbb C^{d_v}$ and an invertible self-adjoint operator $\Lambda_v$ acting on the subspace $P_{R,v}\mathbb C^{d_v}$ such that the boundary values of $u\in D(H)$ at the vertex $v$ verify
\[
P_{D,v}u(v)=P_{N,v}u'(v)=P_{R,v}u'(v)-\Lambda_v P_{R,v}u(v)=0.
\]
Using this expression of the boundary conditions, we can express (see e.g.~\cite[Theorem 1.4.11]{BeKu13}) the quadratic form corresponding to $H$, which we denote by $Q$ and is given by
\begin{equation}
  \label{eq:Q}
Q(u)=\frac12\norm{u'}_{L^2}^2+\frac12\sum_{v\in\mathcal V}\scalar{\Lambda_v P_{R,v}u}{P_{R,v}u}_{\mathbb C^{d_v}}.
\end{equation}
  The domain of $Q$ is given by all functions $u\in H^1(\mathcal G)$ such that at each vertex $P_{D,v}u=0$. We denote it by
    \begin{equation}
H^1_D(\mathcal G)=\{ u\in H^1(\mathcal G) : \forall v\in \mathcal V, \; P_{D,v}u=0\}.\label{eq:Dirichlet-Sobolev}
\end{equation}

We now consider two examples of boundary conditions: Kirchhoff-Neumann and $\delta$-type. We already recalled what 
the classical Kirchhoff-Neumann boundary conditions~\eqref{eq:kirchhoff-neumann} are. In terms of the projection operator,  the Dirichlet part $P_{D,v}$ in the Kirchhoff-Neumann case is simply the projection on the kernel of $B_v$, given by
\[
 P_{D,v}= \frac{1}{d_v}\begin{pmatrix}
    d_v-1&-1&\cdots&\cdots&-1\\
    -1&d_v-1&&&\vdots\\
    \vdots&&\ddots&&\vdots\\
    \vdots&&&d_v-1&-1\\
    -1&\cdots&\cdots&-1&d_v-1
  \end{pmatrix}.
\]
The Neumann part is given by $I-P_{D,v}$, precisely
\[
P_{N,v}=\frac{1}{d_v}
\begin{pmatrix}
  1&\cdots&1\\
  \vdots&&\vdots\\
  1&\cdots&1
\end{pmatrix},
\]
and there is no Robin part.

We consider now a vertex with a $\delta$-type condition of strength $\alpha_v\in\mathbb R$ at the vertex $v$, which is defined for $u\in H^2(\mathcal G)$ as follows:
\[
u\text{ is continuous at }v,\quad \sum_{e\sim v}u_e'(v)=\alpha_v u(v).
\]
This vertex condition is analogous to the jump condition appearing in the domain of the operator for the celebrated Schr\"odinger operator with Dirac potential (see e.g. the reference book~\cite{AlGeHoHo88} and Section~\ref{sec:delta-ground-states}). In terms of $A_v$ and $B_v$ matrices, the condition takes the form
\begin{equation*}
A_v=
\begin{pmatrix}
  1&-1&&&(0)\\
0&1&-1\\
\vdots&\ddots&\ddots&\ddots\\
0&&\ddots&1&-1\\
-\alpha_v&0&\cdots&0&0
\end{pmatrix},
\quad
B_v=
\begin{pmatrix}
0&\dots&0\\
\vdots&&\vdots\\
0&\dots&0\\
1&\dots&1
\end{pmatrix}.
\end{equation*}
When $\alpha_v=0$, we recover the classical Kirchoff-Neumann boundary conditions. When $\alpha_v\neq0$, the Dirichlet, Neumann and Robin projectors are given as follows. The Dirichlet projector $P_{D,v}$ is (as when $\alpha_v=0$) the projection on the kernel of $B_v$. There is no Neumann part and the Robin part is given by $I-P_{D,v}$ (which was the Neumann part for $\alpha=0$). The operator $\Lambda_v=B_v^{-1}A_v$ on the range of $P_{R,v}$ is the multiplication by $\frac{\alpha_v}{d_v}$. Assuming that we have $\delta$-type conditions on the whole graph, the domain $H^1_D(\mathcal G)$ of the quadratic form $Q$ associated with $H$ is the space of functions of $H^1(\mathcal G)$ continuous at each vertex, and we thus may write $u(v)$ for the unique scalar value of $u\in H^1_D(\mathcal G)$ at each vertex. 
The quadratic form associated with $H$ then becomes 
\[
Q(u)=\frac{1}{2}\norm{u'}_{L^2}^2+\frac12\sum_{v\in\mathcal V}\alpha_{v}|u(v)|^2.
\]

\subsection{Nonlinear quantum graphs}
Having established the necessary preliminaries on linear quantum graphs in the previous section, we now turn to nonlinear quantum graphs. Given a quantum graph $(\mathcal G,H)$, 
we consider the nonlinear Schr\"odinger equation on the graph
$\mathcal G$ given by 
\begin{equation}
  \label{eq:nls}
  i\partial_tu-Hu+f(u)=0,
\end{equation}
where $u=u(t,\cdot) \in L^2(\mathcal{G})$ is the unknown wave function, $t$ the time variable, and $f$ is a nonlinearity satisfying the following requirements.

\begin{assumption}
  \label{ass:nonlinearity}
  The nonlinearity $f:\mathbb C\to\mathbb C$ verifies the following assumptions.
  \begin{itemize}
    \item  Gauge invariance: there exists $g:[0,\infty)\to\R$ such that
$f(z)=g(|z|^2)z$ for any $z\in\mathbb C$. 
  \item $g\in\mathcal{C}^0([0,+\infty),\R)\cap\mathcal{C}^1((0,+\infty),\R)$, $g(0)=0$ and $\lim_{s\to 0} sg'(s)=0$.
  \item There exist $C>0$ and $1< p<\infty$ such that $|s^2g'(s^2)|\leq Cs^{p-1}$ for $s\geq1$.
  \end{itemize}
\end{assumption}

Typical examples for $f$ are power type or double power type
nonlinearities
\[
f(u)=\pm|u|^{p-1}u,\quad f(u)=|u|^{p-1}u-|u|^{q-1}u,
\]
where $1<p,q<\infty$. We will use the real form of the anti-derivative of $f$, which is given for every $z\in\mathbb C$ by
\[
F(z)=\int_0^{|z|}f(s)ds.
\]
Using the antiderivative $G$ of $g$, we may also express $F$ (as we
did in the Introduction) as
\[
F(z)=\frac12G(|z|^2).
  \]
Observe that $f$ is a function defined on $\mathbb C$. Its differential $df$ at $z\in\mathbb C$ might be expressed for $h\in\mathbb C$ by
\[
df(z)h=2g'(|z|^2)z\Re(z\bar h)+g(|z|^2)h.
\]
The functions on which $f$ will be evaluated in the next sections will mostly be real-valued and for simplicity we will use the following notation when the argument of $f$ is real: for $s\in\mathbb R$ we define
\[
f'(s)=2g'(s^2)s^2+g(s^2). 
  \]
Formally,~\eqref{eq:nls} is a Hamiltonian system in the form
\[
i\partial_tu=E'(u),
\]
where the Hamiltonian, or the \emph{energy}, $E$ is a conserved quantity defined for any $u\in H^1_D( \mathcal G)$ by
\[
E(u)=Q(u)-\int_{\mathcal G}F(u)dx.
\]
It is a $\mathcal C^2$ functional on $H^1_D(\mathcal G)$ and its derivative is given by
\begin{equation}
E'(u)=Hu-f(u),\label{eq:derivE}
\end{equation}
with the slight abuse of notation that $H$ here denotes the corresponding operator from $H^1(\mathcal G)$ to its dual.

  From Noether's theorem, the gauge symmetry of~\eqref{eq:nls} yields another conserved quantity (see e.g.~\cite{DeGeRo15}), the \emph{mass}, given by
  \[
M(u)=\norm{u}_{L^2(\mathcal G)}^2. 
    \]

We are interested in this paper in the standing waves solutions for the nonlinear Schr\"odinger equation set on the graph. By definition, a \emph{standing wave} is a solution $u$ of~\eqref{eq:nls} given for all $e\in\mathcal E $ by
\[
  u_e(t,\cdot)=e^{i\omega t}\phi_e(\cdot),
\]
where $\omega\in\R$ and the profile $\phi\in H^1(\mathcal G)$ is independent of time. We refer to the profile $\phi$ as \emph{stationary state}. Substituting into~\eqref{eq:nls} leads to the equation of the profile $\phi$, given by
\begin{equation}
  \label{eq:snls}
H\phi+\omega \phi-f(\phi)=0.
\end{equation}
Therefore, $\phi$ is a critical point of the \emph{action} functional
\[
E+\frac\omega2 M.
\]
Observe that there is a natural smoothing for $\phi$: since, with our assumptions, $(\omega \phi-f(\phi))\in L^2(\mathcal G)$, we have $\phi\in D(H)$. 

Strategies abound to find critical points of the action. One particularly interesting strategy is to minimize the energy on fixed mass, as the obtained minimizer will be (following the method established by Cazenave and Lions~\cite{CaLi82}) the profile of an orbitally stable  standing wave of~\eqref{eq:nls} (provided minimizing sequences are compact, which is usually a key step of the proof). More precisely, given $m>0$, we will be looking for $\phi\in H_D^1(\mathcal G)$ such that
  \begin{equation}
M(\phi)=m,\quad E(\phi)=\min\{ E(\psi):\psi\in H_D^1(\mathcal G),\,M(\psi)=m \}.\label{eq:minimization}
\end{equation}
The theoretical existence of minimizers for the problem~\eqref{eq:minimization} has attracted a lot of attention in the past decade and we will not attempt to give an exhaustive overview of the existing literature. Some examples have already been shortly mentioned in Section~\ref{sec:introduction}. In what follows, we give a few more details on the case of star graphs with two or more edges, and on the topological assumption preventing the existence of ground states. 

\subsubsection{Star graphs with two or more edges}
\label{sec:delta-ground-states}
One of the simplest nontrivial graph is given by two semi-infinite half-lines connected at a vertex, with $\delta$ type condition on the vertex. In this case, the operator $H$ is equivalent to the second order derivative on $\R$ with point interaction at $0$. 
In this setting, existence and stability of standing waves for a focusing power-type nonlinearity was treated by Fukuizumi and co.~\cite{FuJe08,FuOhOz08,LeFuFiKsSi08}, using techniques based on Grillakis-Shatah-Strauss  stability theory (see~\cite{GrShSt87,GrShSt90} for the original papers and~\cite{DeGeRo15,DeRo19} for recent developments).

Various generalizations have been obtained, e.g. for a generic point interaction~\cite{AdNo09,AdNo13,AdNoVi13} ($\delta$ or $\delta'$ boundary conditions) or in the case of non-vanishing boundary conditions at infinity~\cite{IaLeRo17}. In particular, the following results have been obtained in~\cite{AdNoVi13}.
\begin{proposition}
  Assume that $\mathcal G$ is formed by two semi-infinite edges $\{ e_1,e_2\}$ connected at the vertex $v$. Let $H:D(H)\subset L^2(\mathcal G)\to L^2(\mathcal G)$ be the operator $-\partial_{xx}$ with 
one of the following conditions to be satisfied at the vertex.
  \begin{itemize}
  \item Attractive $\delta$ conditions:
    \[
      \varphi_{e_1}(v)=\varphi_{e_2}(v),\quad \varphi_{e_1}'(v)+\varphi_{e_2}'(v)=\alpha\varphi(v),\;\alpha<0.
    \]
    \item Attractive $\delta'$ conditions:
    \[
\varphi_{e_1}(v)-\varphi_{e_2}(v)=\beta\varphi_{e_2}'(v),\;\beta<0,\quad       \varphi_{e_1}'(v)+\varphi_{e_2}'(v)=0.
\]
\item Dipole conditions: 
  \[
\varphi_{e_1}(v)+\tau\varphi_{e_2}(v)=0,\quad     \varphi_{e_1}'(v)+\tau\varphi_{e_2}'(v)=0,\quad \tau\in\R.
    \]
  \end{itemize}
  Define for $\varphi\in H^1_D(\mathcal G)$ the energy
  \[
E(\varphi)=Q(\varphi)-\frac{1}{p+1}\norm{\varphi}_{L^{p+1}}^{p+1},
\]
where $1<p<5$. Then for any $m>0$ there exists up to phase shift and translation a unique minimizer to
\[
\min\{E(\varphi):\varphi\in H^1_D(\mathcal G),\,M(\varphi)=m\}.
\]
\end{proposition}
A detailed review of these results as well as announcement of new
results can be found in~\cite{AdBoRu20}. The minimizer is in fact
explicitly known, and we use its explicit form in Section
\ref{sec:2edges} to compare the outcome of our numerical experiences
with the theoretical ground states.

\subsubsection{General non-compact graphs with Kirchhoff condition}
The existence of ground states with prescribed mass for the focusing nonlinear
Schr\"odinger equation on non-compact graphs $\mathcal{G}$ equipped with
Kirchhoff boundary conditions is linked to the topology of the graph. Actually, a
topological hypothesis, usually referred to as Assumption {(H)} can prevent a graph from having ground states for every
value of the mass (see~\cite{AdSeTi17b} for a review). For the sake of clarity,
we recall that a \textsl{trail} in a graph is a path made of adjacent edges, in
which every edge is run through exactly once. In a trail, vertices can be run
through more than once. The Assumption {(H)} has many formulations
(see \cite{AdSeTi17b}) but we give here only the following one.

\begin{assumption}[Assumption (H)]
  \label{ass:H}
  Every $x\in \mathcal{G}$ lies in a trail that contains two half-lines.
\end{assumption}

Under Assumption \ref{ass:H} {(H)}, no global minimizer exists, unless $\mathcal G$
is (up to symmetries) isomorphic to $\R$ (note that this assumption does not prevent the existence of local minimizers).
Let us consider for example a general
$N$-edges star-graph $\mathcal{G}$ (see Figure~\ref{fig:N_star_graph}). 
The $N$ star-graph with $N>2$ 
verifies Assumption~\ref{ass:H} {(H)}, so there are no ground states in this case
without adding more constraints. Another example satisfying Assumption~\ref{ass:H} {(H)} is the triple bridge $\mathcal{B}_3$ (represented in Figure~\ref{fig:three_bridge}).
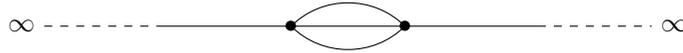
\begin{figure}[htbp!]
  \centering
  \begin{tikzpicture}
    \node[left] at (-4,0) {$\infty$};
    \node[right] at (4,0) {$\infty$};
    \draw[dashed] (-4,0) to (-2.5,0);
    \draw (-2.5,0) to (2.5,0);
    \draw[dashed] (2.5,0) to (4,0);
    \draw (-.75,0) to [bend right=45] (.75,0);
    \draw (-.75,0) to [bend left=45] (.75,0);
    \node at (-.75,0) {$\bullet$};
    \node at (.75,0) {$\bullet$};
  \end{tikzpicture}
  \caption{The $3$-bridge $\mathcal{B}_3$}
  \label{fig:three_bridge}
\end{figure}
When we are searching to obtain ground
states, we  consider graphs violating Assumption~\ref{ass:H} {(H)},
for example the
signpost graph or a line with a tower of bubbles (Figure~\ref{fig:signpost_bubbles}).

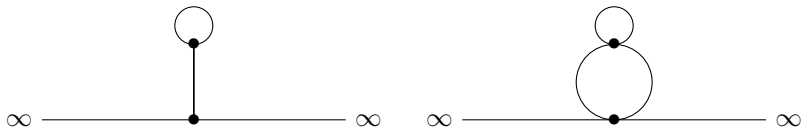
\begin{figure}[htbp!]
  \centering
  \begin{tabular}{cc}
    \begin{tikzpicture}
      \node[left] at (-2,0) {$\infty$};
      \node[right] at (2,0) {$\infty$};
      \draw (-2,0) -- (0,0) -- (0,1) -- (0,0) -- (2,0);
      \node at (0,0) {$\bullet$};
      \node at (0,1) {$\bullet$};
      \draw (0,1.25) circle (0.25);
    \end{tikzpicture}&
    \begin{tikzpicture}
      \node[left] at (-2,0) {$\infty$};
      \node[right] at (2,0) {$\infty$};
      \draw (-2,0) -- (0,0)  -- (2,0);
      \node at (0,0) {$\bullet$};
      \draw (0,0.5) circle (0.5);
      \node at (0,1) {$\bullet$};
      \draw (0,1.25) circle (0.25);
    \end{tikzpicture}\\
  \end{tabular}
  \caption{Line with a signpost graph (left) and with a tower of bubbles (right)}
  \label{fig:signpost_bubbles}
\end{figure}

\section{Continuous normalized gradient flow}
\label{sec:cngf}

We want here to show that, when the standing wave profile $\phi$ is a strict local minimizer for the energy on fixed mass, the corresponding continuous normalized gradient flow (i.e. the gradient flow of the energy projected on the mass constraint) converges towards $\phi$.

The continuous normalized gradient flow is defined by
\begin{equation}
  \label{eq:CNGF}
    \partial_t\psi=-E'(\psi)+\dual{E'(\psi)}{\frac{\psi}{\norm{\psi}_{L^2}}}\frac{\psi}{\norm{\psi}_{L^2}},\quad     \psi(t=0)=\psi_0,
\end{equation}
where $\psi=\psi(t,\cdot)$.
It is the projection of the usual gradient flow
\[
\partial_t\psi=-E'(\psi)
\]
on the $L^2$ sphere
\[
  \mathcal S_{\psi_0}=\{ u\in H^1_D(\mathcal G) :\norm{u}_{L^2}=\norm{\psi_0}_{L^2}\}.
\]  

Let $\phi\in H^1_D(\mathcal G)$ be a standing wave profile solution of~\eqref{eq:snls}.  We define the linearized action operator $L_+$ around $\phi$ by
\begin{equation}
  \label{eq:L_+}
  \begin{aligned}
    L_+:D(H)\subset L^2(\mathcal G)&\to L^2(\mathcal G),\\
    u&\mapsto Hu+\omega u-f'(\phi)u.
  \end{aligned}
\end{equation}
We will assume that the bound state $\phi$ is a strict local minimizer of the energy on fixed $L^2$-norm, which translates for $L_+$ into the following assumption.
  \begin{assumption}
    \label{ass:coercivity}
    There exists $\kappa>0$ such that for any $\varphi\in D(H)$ verifying
    \[
\scalar{\varphi}{\phi}_{L^2}=0,
\]
we have
\[
\scalar{L_+\varphi}{\varphi}_{L^2}\geq \kappa\norm{\varphi}_{H^1}^2.
  \]
  \end{assumption}

Since the pioneering work of Weinstein~\cite{We85}, this assumption is
well known to hold (if one removes translations and phase shifts) in
the classical case of Schr\"odinger equations on $\R^d$ with
subcritical power-nonlinearities ($f(\varphi)=|\varphi|^{p-1}\varphi$,
$1<p<1+4/d$). It is has also been established in many different cases,
for example in~\cite{IaLeRo17,LeFuFiKsSi08} in the case of the $2$
branches star graph with $\delta$ conditions on the vertex (which is
equivalent to the line with a Dirac potential) or in~\cite{GuLeTs17}
in the case of a $1$-loop graph with Kirchhoff conditions at the vertex
(which is equivalent to an interval with periodic boundary
conditions). Observe that a local minimizer is not necessarily a global
minimizer (see e.g.~\cite{PiSoVe20}).
  
Our main result in this section is the following.

\begin{theorem}
  \label{thm:convergence}
  Let the nonlinearity $f$ and the bound state $\phi$ be such that Assumption~\ref{ass:nonlinearity} and Assumption~\ref{ass:coercivity} hold. 
  Then for every  $0<\mu<\kappa$ (where $\kappa$ is the coercivity constant of Assumption~\ref{ass:coercivity}) there exist $\eps>0$ and $C>0$ such that for every $\psi_0\in H^1_D(\mathcal G)$ such that
  \[
\norm{\psi_0-\phi}_{H^1}<\eps
\]
the unique solution $\psi\in\mathcal C([0,T),H^1_D(\mathcal G))$ of~\eqref{eq:CNGF} is global (i.e. $T=\infty$) and converges to $\phi$ exponentially fast: for every $t\in[0,\infty)$ we  have
\[
\norm{\psi(t)-\phi}_{H^1}<Ce^{-\mu t}\norm{\psi_0-\phi}_{H^1}.
  \]
\end{theorem}

\begin{remark}
In Theorem~\ref{thm:convergence}, we may choose any $\mu>0$ such that $\mu<\kappa$, where $\kappa$ is the coercivity constant in Assumption~\ref{ass:coercivity}. Hence the convergence rate towards the profile $\phi$ depends on the steepness of the energy well around $\phi$.
\end{remark}
The proof of the theorem is divided into three parts. This is the subject of the next three subsections. 

\subsection{Local well-posedness of the continuous normalized gradient flow}

Before proving Theorem~\ref{thm:convergence}, we establish the following local well-posedness result for the continuous normalized gradient flow~\eqref{eq:CNGF}.


  \begin{proposition}
    \label{prop:lwp}
Assume that the nonlinearity $f$ verifies Assumption~\ref{ass:nonlinearity}.
    Then, 
    for any $\psi_0\in H^1_D(\mathcal G)$, there exists a unique maximal solution
    \[
      \psi\in\mathcal C([0,T^{max}), H^1_D(\mathcal G))\cap \mathcal C((0,T^{max}),D(H))\cap \mathcal C^1((0,T^{max}),L^2(\mathcal G))
    \]
    of the continuous normalized gradient flow~\eqref{eq:CNGF} with $T^{max}\in(0,+\infty]$. Moreover, the mass of the solution is preserved and its energy is diminishing, i.e. for all $t\in(0,T^{\max})$ we have
    \[
\norm{\psi(t)}_{L^2}=\norm{\psi_0}_{L^2}, \quad \partial_tE(\psi(t)) =-\norm{\partial_t\psi}_{L^2}^2\leq 0.
\]
    \end{proposition}

  \begin{proof}[Proof of Proposition~\ref{prop:lwp}]
Let $\psi_0\in H^1_D(\mathcal G)$.   We first show the second part of the statement: preservation of the mass. Let $\psi$ be a solution of~\eqref{eq:CNGF} as in the first part of the statement of Proposition~\ref{prop:lwp}.
We have 
\begin{multline*}
  \frac12\partial_t\norm{\psi}_{L^2}^2=\scalar{\partial_t\psi}{\psi}_{L^2}
  =\dual{-E'(\psi)+\dual{E'(\psi)}{\frac{\psi}{\norm{\psi}_{L^2}}}\frac{\psi}{\norm{\psi}_{L^2}}}{\psi}
  \\=
\dual{-E'(\psi)}{\psi} +\frac{1}{\norm{\psi}_{L^2}^2}\dual{E'(\psi)}{\psi}\scalar{\psi}{\psi}_{L^2}
  =0.
  \end{multline*}
  The mass is therefore preserved for~\eqref{eq:CNGF}. Set
  \[
\alpha=\norm{\psi_0}_{L^2}.
\]

We now prove the first part of the statement (existence and uniqueness of a solution).
We first consider the intermediate problem
    \begin{equation}
      \label{eq:5}
      \partial_t\psi=-E'(\psi)+\frac{1}{\alpha^2}\dual{E'(\psi)}{\psi}\psi.
    \end{equation}
The intermediate problem~\eqref{eq:5} can be written more explicitly (using the expression \eqref{eq:derivE} of $E'(\psi)$) as
    \[
      \partial_t\psi=-H\psi+{\tilde f}(\psi),\quad {\tilde f}(\psi)=f(\psi)+\frac{1}{\alpha^2}\left(2Q(\psi)-\int_{\mathcal G}f(\psi)\psi dx\right)\psi,
    \]
    where $Q$ is the quadratic form associated with $H$ and was defined in~\eqref{eq:Q}. 
    Recall that the operator $H:D(H)\subset L^2(\mathcal G)\to L^2(\mathcal G)$ is self-adjoint. Moreover, there exists $\lambda>0$ such that $H\geq -\lambda$ (this might be seen from the expression of $Q$ given in~\eqref{eq:Q} and the injection of $H^1(\mathcal G)$ into $L^\infty(\mathcal G)$).

Since $f$ verifies Assumption~\ref{ass:nonlinearity}, the nonlinearity ${\tilde f}:H^1_D(\mathcal G)\to H^1_D(\mathcal G)$ is continuous, and, as a function ${\tilde f}:H^1_D(\mathcal G)\to L^2(\mathcal G)$, it is Lipschitz continuous on bounded sets. 
Indeed, for any $z_1,z_2\in \mathbb C$ we have 
\[
|f(z_1)-f(z_2)|\lesssim (1+|z_1|^{p-1}+|z_2|^{p-1})|z_1-z_2|.
\]
Therefore,
for any $M>0$  and for any $\psi_1,\psi_2\in H^1_D(\mathcal G)$ such that $\norm{\psi_1}_{H^1}+\norm{\psi_2}_{H^1}<M$, we have
\[
   \norm{f(\psi_1)-f(\psi_2)}_{L^2}\leq C(M)  \norm{\psi_1-\psi_2}_{H^1},
 \]
 and a similar estimate holds for ${\tilde f}$.

    The existence of the desired solution then follows from classical
    results in the theory of semilinear parabolic problems (see e.g.
    \cite{Lu95,QuSo19}).
    More precisely, there exists a unique
        \[
      \psi\in\mathcal C([0,T^{max}), H^1_D(\mathcal G))\cap \mathcal C((0,T^{max}),D(H))\cap \mathcal C^1((0,T^{max}),L^2(\mathcal G))
    \]
    solution of~\eqref{eq:5} with $\psi(0)=\psi_0$.

    Given $\psi$, we now go back to the continuous normalized gradient flow~\eqref{eq:CNGF} by proving that $t\to\norm{\psi}_{L^2}$ is constant along the evolution in time. We have by direct calculations on~\eqref{eq:5}
\begin{equation*}
  \frac12\partial_t\norm{\psi}_{L^2}^2=\scalar{\partial_t\psi}{\psi}_{L^2}
  =\dual{-E'(\psi)}{\psi}+\frac{1}{\alpha^2}\dual{E'(\psi)}{\psi}\norm{\psi}_{L^2}^2
  =
\frac{1}{\alpha^2}\dual{E'(\psi)}{\psi} \left(\norm{\psi}_{L^2}^2-\alpha^2\right).
\end{equation*}
This is a first order linear ordinary differential equation in $\norm{\psi}_{L^2}^2$ which may be solved explicitly: 
\[
\norm{\psi(t)}_{L^2}^2=\alpha^2+\left(\norm{\psi(0)}_{L^2}^2-\alpha^2\right)\exp\left(\frac{2}{\alpha^2}\int_0^t \dual{E'(\psi(s))}{\psi(s)}ds\right ).
\]
Since $\norm{\psi_0}_{L^2}=\alpha$, this indeed gives 
\[
\norm{\psi(t)}_{L^2}^2=\alpha^2
\]
for any $t\in [0,T^{\max})$. Therefore $\psi$ is also a solution of~\eqref{eq:CNGF}. Uniqueness of such a solution is a direct consequence of the uniqueness for~\eqref{eq:5} and the preservation of the mass. 

  Finally, we establish the energy diminishing property. Using~\eqref{eq:CNGF} to replace $E'(\psi)$, we have 
  \begin{align*} \partial_tE(\psi(t))=\scalar{E'(\psi)}{\partial_t\psi}_{L^2}&=-\scalar{\partial_t\psi}{\partial_t\psi}_{L^2}-\frac{1}{\norm{\psi}_{L^2}^2}\dual{E'(\psi)}{\psi}\scalar{\psi}{\partial_t\psi}_{L^2}
  \\ &=-\norm{\partial_t\psi}_{L^2}^2\leq 0,
  \end{align*}
  where we have used the conservation of the mass in the form $\scalar{\psi}{\partial_t\psi}_{L^2}=0$ to obtain the last equality. This concludes the proof. 
\end{proof}

Having established local well posedness of the continuous normalized gradient flow~\eqref{eq:CNGF}, we turn our attention to the evolution for initial data in the vicinity of the bound state $\phi$.

\subsection{The normal part of the continuous normalized gradient flow}



Given the bound state $\phi\in H^1_D(\mathcal G)$, 
we define a Hilbert subspace $W$ of $H^1_D(\mathcal G)$ by
  \[
    W=\{w\in H^1_D(\mathcal G):\scalar{w}{\phi}_{L^2}=0\}.
  \]
We define the coordinates-to-data map $\chi:\R\times W\to H^1_D(\mathcal G)$ by
  \[
\chi(r,w)=(1+r)\phi+w.
\]
The map $\chi$ is smooth and has bounded derivatives. Its inverse is the data-to-coordinates map $\chi^{-1}:H^1_D(\mathcal G)\to \R\times W$ which is explicitly given by
    \begin{equation}
      \label{eq: data-to-coordinates}
\chi^{-1}(\psi)=(r(\psi),w(\psi))=\left(\frac{\scalar{\psi}{\phi}_{L^2}}{\|\phi\|^2_{L^2}}-1,\psi-\frac{\scalar{\psi}{\phi}_{L^2}}{\|\phi\|^2_{L^2}}\phi\right).
\end{equation}
As $\chi$, the map $\chi^{-1}$ is smooth and has bounded derivatives.

The second step of the proof of Theorem~\ref{thm:convergence} is to decompose the continuous normalized gradient flow~\eqref{eq:CNGF} by projecting it on $W$ and $\phi$, as is done in the following proposition.

\begin{proposition}
\label{prop:CNGF-decomposition}
  Let $T>0$ and $\psi\in\mathcal C((0,T),D(H))\cap\mathcal C^1((0,T),L^2(\mathcal G))$ be a solution of~\eqref{eq:CNGF} such that $\norm{\psi}_{L^2}=\norm{\phi}_{L^2}$ and decompose $\psi$ using the data-to-coordinates map
$\chi^{-1}(\psi(t))=(r(t),w(t))\in \R\times D(H)$ given  by~\eqref{eq: data-to-coordinates}:
\[
\psi(t)=(1+r(t))\phi+w(t).
\]
Then we have
\[
\partial_tw=-L_+w+o(w),\;{ \textrm{in } W',\quad\textrm{and}}\quad r=O(\norm{w}_{L^2}^2),
\]
where $W'$ is the dual of $W$ and $L_+$ was defined in~\eqref{eq:L_+}.
\end{proposition}

The proof of Proposition~\ref{prop:CNGF-decomposition} is divided into three steps. 
In the first step we will  consider the orthogonal decomposition of the flow along $\phi$ and $W$. In the second step we will  project this orthogonal decomposition on the $L^2$-sphere. The third and last step will make the link between the projected normalized energy derivative and the linearized action operator $L_+$.

\subsubsection{Step 1: Orthogonal Decomposition}

We first consider the orthogonal decomposition of the energy.

Consider the functional $\EnergyOD :\R\times W\to \R$ defined for $(r,w)\in\R\times W$ by
\[
\EnergyOD (r,w)=(E\circ \chi)(r,w)=E(\chi(r,w)).
\]
\begin{lemma}[Orthogonal decomposition of the energy]
  \label{lem:orth-dec-energy}
The functional $\EnergyOD $ is differentiable and we have
the following estimates
  \begin{align*}
     D_r\EnergyOD (r,w)&=-\omega\norm{\phi}_{L^2}^2-\scalar{w}{f(\phi)-f'(\phi)\phi}_{L^2}+O(r)+o(\norm{w}_{L^2}),\\
    D_w \EnergyOD (r,w)
      &=Hw -f'(\phi)w+O(r)+o(w).
  \end{align*}
\end{lemma}
\begin{remark}
$D_w \EnergyOD (r,w)$ is an operator acting on $W\subset H^1_D(\mathcal{G})$. For any element $h\in W$, we will note indifferently $D_w \EnergyOD (r,w)h$ or $\langle D_w \EnergyOD (r,w),h\rangle$ the image of $h$ by $D_w \EnergyOD (r,w)$.
\end{remark}
\begin{proof}
Since $E$  and $\chi$ are differentiable, the functional $\EnergyOD $ is also differentiable and we have
  \begin{align}
    D_r\EnergyOD (r,w)&=E'(\chi(r,w))\circ D_r\chi(r,w)=\dual{E'(\chi(r,w))}{\phi},
                                                         \label{eq:der-tilde-1}\\
    D_w \EnergyOD (r,w)&=E'(\chi(r,w))\circ D_w\chi(r,w)=\dual{E'(\chi(r,w))}{\textrm{Id}_W(\cdot)}.
                                  \label{eq:der-tilde-2}
  \end{align}
We now recall that $\phi\in H^1_D(\mathcal G)$ satisfies~\eqref{eq:snls}.
Given $r\in\R$ and $w\in W$, using~\eqref{eq:derivE}, we
have
\begin{multline*}
  E'(\chi(r,w))=H\chi(r,w)-f(\chi(r,w))
  \\=(1+r)H\phi+Hw-f(\phi)-f'(\phi)(r\phi+w)+o(r\phi+w)
  \\=(1+r)(f(\phi)-\omega\phi)-f(\phi)+Hw -f'(\phi)w  -rf'(\phi)\phi+o(r\phi+w)
  \\=-\omega \phi +Hw -f'(\phi)w +O(r)+o(w).
\end{multline*}
We have used here the fact that $f\in\mathcal C^1$ for the Taylor expansion, and that $\phi$ is bounded. 
We may now use this estimate in~\eqref{eq:der-tilde-1} to obtain
\[
  D_r\EnergyOD (r,w)=\dual{E'(\chi(r,w))}{\phi}=-\omega\norm{\phi}_{L^2}^2+\dual{Hw-f'(\phi)w}{\phi}+O(r)+o(\norm{w}_{L^2}).
  \]
  The operator $H-f'(\phi)$ is self-adjoint and $H\phi-f'(\phi)\phi=-\omega\phi+f(\phi)-f'(\phi)\phi$.
Using $w\in W$ (i.e. $\scalar{w}{\phi}_{L^2}=0$) we obtain
\[
 D_r\EnergyOD (r,w)=-\omega\norm{\phi}_{L^2}^2-\scalar{w}{f(\phi)-f'(\phi)\phi}_{L^2}+O(r)+o(\norm{w}_{L^2}),
\]
which proves the first part of the statement.

From~\eqref{eq:der-tilde-2}, for $h\in W$ we get
    \begin{multline*}
      D_w \EnergyOD (r,w)h=\dual{E'(\chi(r,w))}{h}\\
      = -\omega\dual{\phi}{h}+\dual{Hw -f'(\phi)w}{h} +\dual{O(r)+o(w)}{h}\\
      =\dual{Hw -f'(\phi)w}{h} +\dual{O(r)+o(w)}{h},
    \end{multline*}
    where to get the last line we have used that $h\in W$ and thus
    $\dual{\phi}{h}=\scalar{\phi}{h}_{L^2}=0$. 
This proves the second part of the statement. 
\end{proof}

\subsubsection{Step 2: Projection on the $L^2$-sphere}


We now make the link between the orthogonal decomposition and the mass normalization constraint.

We denote the $L^2$ sphere of radius $\norm{\phi}_{L^2}$ by
\[
\mathcal S_{\phi}=\{v\in H^1_D(\mathcal G):\norm{v}_{L^2}=\norm{\phi}_{L^2}\}.
\]
Consider the open subset of $W$ given by
  \[
  \mathcal O_W  =\{w\in W:\norm{w}_{L^2}<\norm{\phi}_{L^2}\}.
  \]
We define the functional $r_W:  \mathcal O_W\to\R$ for any $w\in  \mathcal O_W$ by the implicit relation
\[
\norm{\chi(r_W(w),w)}_{L^2}=\norm{\phi}_{L^2}. 
\]
  The functional $r_W$ can be made explicit by a direct calculation on the above equality and is given for $w\in \mathcal O_W$ by
\[
r_W(w)=-1+\sqrt{1-\left(\frac{\norm{w}_{L^2}}{\norm{\phi}_{L^2}}\right)^2}.
  \]
 In particular, $r_W$ is well defined and smooth. Moreover, we have in the open set $  \mathcal O_W $ the estimate
 \begin{equation}
\abs{r_W(w)}\leq \left(\frac{\norm{w}_{L^2}}{\norm{\phi}_{L^2}}\right)^2.\label{eq:estim_rW}
\end{equation}
Thus, we have a local parametrization of $\mathcal S_{\phi} $ around $\phi$ given by 
\[
  \begin{aligned}
    \mathcal O_W&\mapsto \mathcal S_{\phi},\\
    w&\to\chi(r_W(w),w).
  \end{aligned}
  \]
Introduce the functional $\EnergyNorm :   \mathcal O_W\to\R$ defined by
  \[
\EnergyNorm (w)=\EnergyOD (r_W(w),w)=E(\chi(r_W(w),w))=E((1+r_W(w))\phi+w).
\]
This functional can be used to describe the dynamics of the projected part of the normalized flow, as is done in the following lemma.

\begin{lemma}[Gradient flow in local variables]
  \label{lem:gradient-flow-local}
Let $w$ be as in Proposition~\ref{prop:CNGF-decomposition}. Then
  $w$ is a solution of
\begin{equation}
  \label{eq:projected-CNGF}
    \partial_tw
=-D_w\EnergyNorm (w)+  \frac{\dual{D_w\EnergyNorm (w)}{w}}{\|\phi\|_{L^2}^2}w,\; \textrm{in }W'.
\end{equation}
\end{lemma}

\begin{proof}
 Observe first that $r_W$ and $\EnergyNorm $ are differentiable on $\mathcal O_W$. Their differentials are given, for $w\in   \mathcal O_W$ and $h\in W$
 such that $w+h\in   \mathcal O_W$, by
  \begin{align*}
    \dual{D_w r_W(w)}{h}&=-\frac{\scalar{w}{h}_{L^2}}{\norm{\phi}_{L^2}^2}\left( 1-\left(\frac{\norm{w}_{L^2}}{\norm{\phi}_{L^2}}\right)^2 \right)^{-\frac12}=-\frac{\scalar{w}{h}_{L^2}}{\norm{\phi}_{L^2}^2}\frac{1}{1+r_W(w)},
\end{align*}
and
\begin{align*}
\dual{D_w\EnergyNorm (w)}{h}&=\dual{E'(\chi(r_W(w),w))}{h}+\dual{D_w r_W(w)}{h}\dual{E'(\chi(r_W(w),w))}{\phi}.
  \end{align*}
Using the derivatives of $\EnergyOD $ (given in~\eqref{eq:der-tilde-1}-\eqref{eq:der-tilde-2}) and $r$,  we might express $D_w\EnergyNorm (w)$ in the following way:
  \begin{equation}
    \label{eq:10}
D_w\EnergyNorm (w)=(D_w\EnergyOD )(r_W(w),w)-\frac{(D_r\EnergyOD )(r_W(w),w) }{\norm{\phi}_{L^2}^2}\frac{1}{1+r_W(w)}w.
\end{equation}
Recall that $\psi\in\mathcal C([0,T],H^1_D(\mathcal G))\cap\mathcal C^1((0,T),L^2(\mathcal G))$ is a solution of the continuous normalized gradient flow~\eqref{eq:CNGF} such that $\norm{\psi}_{L^2}=\norm{\phi}_{L^2}$ and that $\psi$ is decomposed using the data-to-coordinates map
$\chi^{-1}(\psi(t))=(r(t),w(t))\in \R\times W$ given  by~\eqref{eq: data-to-coordinates} in the following way:
\[
\psi(t)=(1+r(t))\phi+w(t).
\]
Since $r(t)=\scalar{\psi(t)}{\phi}_{L^2}-1= r_W(w(t))$,
the function $r$ of $t$  is $\mathcal C^1$. The regularity of $w$ in
$t$ is the same as the regularity of $\psi$. 
We have
\[
\norm{\psi(t)}_{L^2}^2=(1+r(t))^2\norm{\phi}_{L^2}^2+\norm{w(t)}_{L^2}^2,
\]
which, by conservation of the $L^2$-norm for $\psi$ implies that for all $t\in[0,T]$ we have
\[
-1\leq r(t)\leq0\quad \text{and}\quad \norm{w(t)}_{L^2}^2\leq \norm{\psi}_{L^2}^2.
\]
We want to convert the continuous normalized gradient flow~\eqref{eq:CNGF} in $\psi$ into a closed equation for $w$ ($r$ can be directly deduced from $w$ by preservation of the $L^2$ norm). Observe first that
\[
\partial_t\psi=\phi\partial_tr+\partial_tw.
\]
To obtain the evolution equation for $w$, we take $h\in W$ and compute:
\[
\dual{\partial_tw}{h}=\dual{\partial_t\psi-\phi\partial_tr}{h}=\dual{\partial_t\psi}{h}.
\]
Since $\psi$ is a solution of the normalized gradient flow~\eqref{eq:CNGF}, we get
\[
\dual{\partial_t\psi}{h}=\dual{-E'(\psi)}{h}+\frac{\dual{E'(\psi)}{\psi}}{\norm{\psi}_{L^2}^2}\dual{\psi}{h}.
\]
Since $h\in W$, we have
\[
\dual{-E'(\psi)}{h}=-D_w\EnergyOD (r,w)h.
\]
We also have
\begin{align*}
  \dual{E'(\psi)}{\psi}&= \displaystyle (1+r) \dual{E'(\psi)}{\phi}+\dual{E'(\psi)}{w} \\
                       &=\displaystyle (1+r)D_r\EnergyOD (r,w)+\dual{D_w \EnergyOD (r,w)}{w}.    
\end{align*}
Using $\dual{\psi}{h}=\dual{w}{h}$, we get the following equation:
{\color{black}\begin{align*}
    \dual{\partial_tw}{h}=&-D_w\EnergyOD(r,w)h+\\
    & \displaystyle \frac{1}{\norm{\phi}_{L^2}^2}\Big((1+r)D_r\EnergyOD(r,w)+\dual{D_w \EnergyOD(r,w)}{w}
    \Big)\dual{w}{h}.  
\end{align*}
}%
{\color{black}Since the previous equation holds for any $h\in W$, it can be rewritten as}
  \[
\partial_tw=-D_w\EnergyOD(r,w)+\frac{1}{\norm{\phi}_{L^2}^2}\Big((1+r)D_r\EnergyOD(r,w)+\dual{D_w \EnergyOD(r,w)}{w}
    \Big)w{\color{black},\quad \textrm{in }W'}.
    \]
    By conservation of the $L^2$-norm in the normalized gradient flow~\eqref{eq:CNGF}, $r$ might be inferred from $w$ and we have  for $w$ the following closed equation
    \begin{multline}
        \label{eq:9}
\partial_tw=-(D_w\EnergyOD)(r_W(w),w)\\+\frac{1}{\norm{\phi}_{L^2}^2}\Big((1+r_W(w))(D_r\EnergyOD)(r_W(w),w)+\dual{(D_w \EnergyOD)(r_W(w),w)}{w}
    \Big)w.  
  \end{multline}
Using~\eqref{eq:10} to replace $(D_w\EnergyOD)(r_W(w),w)$ in~\eqref{eq:9}, we obtain
{\color{black}
\begin{equation*}
\begin{array}{rcl}
  \partial_tw&=&\displaystyle -D_w\EnergyNorm (w)-\frac{(D_r\EnergyOD)(r_W(w),w) }{\norm{\phi}_{L^2}^2}\frac{1}{1+r_W(w)}w\\
              && \displaystyle +\frac{1}{\norm{\phi}_{L^2}^2}(1+r_W(w))(D_r\EnergyOD)(r_W(w),w)w\\
              && \displaystyle +\frac{1}{\norm{\phi}_{L^2}^2}
                 \left\langle D_w\EnergyNorm (w) + \frac{(D_r\EnergyOD)(r_W(w),w)}{\norm{\phi}_{L^2}^2}\frac{1}{1+r_W(w)}w,w\right\rangle \, w\\
             &=&\displaystyle -D_w\EnergyNorm (w) + \frac{1}{\norm{\phi}_{L^2}^2} \left\langle D_w\EnergyNorm (w),w\right\rangle \, w\\
             &&\displaystyle + \frac{(D_r\EnergyOD)(r_W(w),w)}{\norm{\phi}_{L^2}^2} \left[
                -\frac{1}{1+r_W(w)}w+(1+r_W(w))w\right .\\
  &&\displaystyle \left. \qquad \qquad\qquad\qquad \qquad\qquad\qquad \qquad\qquad +\frac{1}{1+r_W(w)} \frac{\|w\|_{L^2}^2}{\|\phi\|_{L^2}^2}w
                \right]\\
  &=&\displaystyle -D_w\EnergyNorm (w)+  \frac{\dual{D_w\EnergyNorm (w)}{w}}{\|\phi\|_{L^2}^2}w,  
\end{array}
\end{equation*}
}%
where to get the last line we have used the expression of $\norm{w}_{L^2}^2$ in terms of $r_W(w)$, i.e.
\[
\norm{w}_{L^2}^2=\norm{\phi}_{L^2}^2\left(1-(1+r_W(w))^2\right).
\]
This concludes the proof. 
\end{proof}

\subsubsection{Step 3: Link with the linearized action}

To conclude the proof of Proposition~\ref{prop:CNGF-decomposition}, it remains to make the link between $ D_w\EnergyNorm (w)$ and $L_+$.

\begin{lemma}
  \label{lem:L_+}
  The differential $ D_w\EnergyNorm (w)$ can be approximated in the following way:
  \[
    D_w\EnergyNorm (w)=L_+w+o(w),
  \]
  where $L_+$ has been defined in~\eqref{eq:L_+}.
\end{lemma}
  \begin{proof}
      We already have obtained  in~\eqref{eq:10} the identity
  \[
D_w\EnergyNorm (w)=(D_w\EnergyOD)({\color{black}r_W}(w),w)-\frac{(D_r\EnergyOD)({\color{black}r_W}(w),w) }{\norm{\phi}_{L^2}^2}\frac{1}{1+{\color{black}r_W}(w)}w.
\]
From the estimates on  $(D_w\EnergyOD)({\color{black}r_W}(w),w) $ and $(D_r\EnergyOD)({\color{black}r_W}(w),w)$ given in Lemma~\ref{lem:orth-dec-energy}, we have
\[
  -\frac{(D_r\EnergyOD)({\color{black}r_W}(w),w) }{\norm{\phi}_{L^2}^2}\frac{1}{1+{\color{black}r_W}(w)}w=
\omega w
+O(r)+{\color{black}o(\|w\|_{L^2})},
\]
where we have used the classical power series expansion
  \(
\frac{1}{1+r}=1-r+\cdots
\).
Finally, we obtain
    \begin{equation*}
D_w\EnergyNorm (w)=Hw +\omega w-f'(\phi)w+O(r)+o(w)=L_++o(w)
\end{equation*}
(since {\color{black}from~\eqref{eq:estim_rW}}, we have $r=O(\norm{w}_{L^2}^2)$). This concludes the proof.
\end{proof}

Proposition~\ref{prop:CNGF-decomposition} is a direct consequence of Lemmas~\ref{lem:orth-dec-energy},~\ref{lem:gradient-flow-local}, \ref{lem:L_+} and bound~\eqref{eq:estim_rW}.

\subsection{Convergence of the  normal part of the continuous normalized gradient flow}

 The second step of the proof of Theorem~\ref{thm:convergence} is to prove convergence to $0$ of the projected part $w$ of the solution $\psi$ of the continuous normalized gradient flow~\eqref{eq:CNGF}, provided $\norm{w_0}_{H^1}$ is small (i.e. $\psi_0$ is close enough to the bound state $\phi$). This is the object of the following proposition. 

  \begin{proposition}[Convergence of the flow]
    \label{prop:convergence}
  For every  $0<\mu<\kappa$ (where $\kappa$ is the coercivity constant of Assumption~\ref{ass:coercivity}) there exist $\eps>0$ and $C>0$ (independent of $\varepsilon$) such that for any $w_0\in W$ verifying $\norm{w_0}_{H^1}<\eps$ the associated solution $w$ of~\eqref{eq:projected-CNGF} is global and for all $t\in[0,\infty)$ verifies
    \[
      \norm{w(t)}_{H^1}\leq C \norm{w_0}_{H^1} 
      e^{-\mu t}.
    \]
  \end{proposition}

Since $\abs{r(w)}\leq C \norm{w}_{L^2}^2$ ({\color{black}see}~\eqref{eq:estim_rW}), 
  Theorem~\ref{thm:convergence} is a direct consequence of Propositions~\ref{prop:CNGF-decomposition}
and~\ref{prop:convergence}.

\begin{proof}[Proof of Proposition~\ref{prop:convergence}]
Denote
\begin{equation}
  \label{eq:13}
  R(w)=D_w\EnergyNorm (w)-\frac{\dual{D_w\EnergyNorm(w)}{w}}{\color{black}\|\phi\|_{L^2}^2}w-L_+w=o(w).
\end{equation}
We remark that, for any $t\in(0,T)$,
\[
\scalar{\phi}{w(t)}_{L^2} = 0\quad\Rightarrow\quad \scalar{\phi}{\partial_t w(t)}_{L^2} = 0.
\]
Thus, {\color{black} by denoting $P_W$ the orthogonal projector on $W$,} since $L_+$ is self-adjoint and ${\color{black}\partial_t w \in W}$ verifies~\eqref{eq:projected-CNGF}, we have
      \begin{align}
\frac{d}{d t}\dual{L_+w(t)}{w(t)} &{\color{black}= 2 \dual{L_+w(t)}{P_W\partial_t w(t)} = 2 \dual{P_W L_+w(t)}{\partial_t w(t)}}\nonumber
\\ &=-2\dual{L_+w(t)}{L_+w(t)}+2\dual{R(w(t))}{w(t)},\label{eq:14}
\end{align}
where $R$ is given by~\eqref{eq:13}.  By the coercivity estimate in Assumption~\ref{ass:coercivity} and Cauchy-Schwartz inequality, we have
\[
\kappa\norm{w}_{H^1}^2\leq \dual{L_+w}{w}\leq \norm{L_+w}_{L^2}\norm{w}_{L^2},
\]
which implies in particular that
\[
\kappa\norm{w}_{H^1}\leq  \norm{L_+w}_{L^2}.
  \]
  Coming back to~\eqref{eq:14}, we get
  \begin{align}
    \frac{d}{d t}\dual{L_+w(t)}{w(t)}&\leq -2\kappa^2\norm{w(t)}_{L^2}^2+2 \norm{R(w(t))}_{L^2}\norm{w(t)}_{L^2} \nonumber\\
    &\leq -2\kappa^2\norm{w(t)}_{L^2}^2+ o(\norm{w(t)}_{L^2}^2).     \label{eq:1}
\end{align}
Assume that $\norm{w_0}_{H^1}<\eps$ where $\eps>0$ is  chosen such that
\[
-2\kappa^2\norm{w_0}_{L^2}^2+ o(\norm{w_0}_{L^2}^2)<-2\kappa\mu\norm{w_0}_{L^2}^2,
\]
(recall that $0<\mu<\kappa)$).
Since $w$ is continuous, there exists $T_0>0$ such that for any $t\in[0,T_0]$, we have
  \begin{equation}
    \label{eq:2}
-2\kappa^2\norm{w(t)}_{L^2}^2+ o(\norm{w(t)}_{L^2}^2)<-2\kappa\mu\norm{w(t)}_{L^2}^2.
\end{equation}
For $t\in[0,T_0]$ we integrate~\eqref{eq:1} in time from $0$ to $t$  and use~\eqref{eq:2} to obtain
\[
\dual{L_+w(t)}{w(t)}-\dual{L_+w_0}{w_0}\leq -2\kappa\mu\int_0^t\norm{w(s)}_{H^1}^2ds.
\]
Defining the constant $C_0=\dual{L_+w_0}{w_0}/\kappa$ and using again the coercivity estimate of Assumption~\ref{ass:coercivity}, we get
\[
  \norm{w(t)}_{H^1}^2\leq C_0-2\mu\int_0^t\norm{w(s)}_{H^1}^2ds,
\]
which by Gronwall inequality gives
\[
\norm{w(t)}_{H^1}^2\leq C_0e^{-2\mu t}
\]
and therefore
\[
\norm{w(t)}_{H^1}\leq \sqrt{C_0}e^{-\mu t}.
\]
Note that there exists a constant $C>0$ independent of $w_0$ such that
\[
   C_0 = \dual{L_+w_0}{w_0}/\kappa \leq C \norm{w_0}_{H^1}^2,
\]
thanks to \eqref{eq:Q} and Sobolev embeddings.
This concludes the proof. 
  \end{proof}

\section{Space-time discretization of the normalized gradient flow}
\label{sec:discretization}

\subsection{Time discretization}
The discretization scheme of the continuous normalized gradient flow~\eqref{eq:CNGF} must provide a numerical method to obtain a minimizer of~\eqref{eq:minimization}. We first consider the semi-discretization in time. The
time step $\delta t >0$ is chosen to be fixed and the discrete times $t_n$ are
defined as $t_n=n \delta t$, $n\geq 0$. The semi-discrete approximation of any
unknown function $\psi(\cdot,x)$, $x \in \mathcal{G}$, at time $t_n$ is denoted by
$\psi^n(x)$. In order to present the numerical schemes, we recall that the nonlinearity verifies Assumption~\ref{ass:nonlinearity} and that it is of the form
\begin{equation*}
f(\psi) = g(|\psi|^2)\psi,
\end{equation*}
where $g$ is continuous.
We also introduce the variable $\mu_m$, usually referred to as the chemical potential,
\begin{equation*}
  \mu_{m}(\psi) = \frac{1}{m}\left \langle
    E'(\psi),\psi \right \rangle,\qquad m=\|\psi\|_{L^2}^2.
\end{equation*}
Dropping the dependence on $(t,x)\in[0,+\infty)\times\mathcal{G}$ and using~\eqref{eq:derivE}, the continuous normalized gradient flow~\eqref{eq:CNGF} can therefore be rewritten as
  \begin{equation}
  \label{eq:CNGF2}
      \partial_t\psi=-(H - g(|\psi|^2))\psi+\mu_m(\psi) \psi,\quad
      \psi(t=0)=\psi_0,
    \end{equation}
 where $\|\psi_0\|_{L^2}^2=m$.

Several numerical methods can be considered for discretizing~\eqref{eq:CNGF2}. For
example, if the nonlinearity is $f(\psi)=|\psi|^2\psi$, a standard Crank-Nicolson scheme
would consist in
\[
  \frac{\psi^{n+1}-\psi^{n}}{\delta t} = \left (-H + 
  \frac{|\psi^{n+1}|^2 + |\psi^n|^2}{2} + \mu_{m}^{n+\frac12}\right )  \psi^{n+\frac12},
\]
where the intermediate values at $t_{n+\frac12}$ are given by
  \begin{gather*}
    \psi^{n+\frac12}=\frac{\psi^{n+1}+\psi^{n}}{2},
\quad
    \mu_{m}^{n+\frac12} = \frac{D^{n+\frac12}}{\|\psi^{n+\frac12}\|_{L^2}^2},
\\
D^{n+\frac12} = 2 Q (\psi^{n+\frac12}) - \frac12 \left( \norm{ |\psi^{n+1}| \, \psi^{n+\frac12}}_{L^2}^2 +  \| |\psi^{n}|\, \psi^{n+\frac12} \|_{L^2}^2\right).
  \end{gather*}
This method can be proved to be energy diminishing.
However, in the above discretization, we need to solve a fully
nonlinear system at every time step, which is  time- and resource-consuming in practical
computation. 

Bao and Du introduced in~\cite{BaDu04} a more efficient solution: the
Gradient Flow with Discrete Normalization (GFDN) method, which
consists into one step of classical gradient flow followed by a
mass normalization step. By setting $\psi^0=\psi_0$, it is given by
\begin{equation}
  \label{eq:GFDN}
  \left \{
    \begin{aligned}
      \displaystyle \frac{\varphi^{n+1}-\psi^n}{\delta t}& = -H
      \varphi^{n+1}+g(|\psi^n|^2)\varphi^{n+1},
      \\
      \displaystyle
      \psi^{n+1}&=\sqrt{m}\frac{\varphi^{n+1}}{\|\varphi^{n+1}\|_{L^2}}.
    \end{aligned}
  \right .    
\end{equation}
It is not clear at first sight that~\eqref{eq:GFDN} is indeed a
discretization of~\eqref{eq:CNGF2}, but we have the following
result. 

\begin{proposition}
  \label{prop:discretization}
The GFDN method~\eqref{eq:GFDN} is a time-discretization of the continuous normalized gradient flow~\eqref{eq:CNGF2}.
\end{proposition}
Some arguments are given in~\cite{bao2007,BaDu04} for $m=1$, we provide here a proof with additional details and extend the result for any $m>0$.

\begin{proof}[Proof of Proposition~\ref{prop:discretization}]
The starting point is to apply a first order splitting, also known as
Lie splitting, to~\eqref{eq:CNGF2}. Assuming that the approximation
$\psi^n$ of $\psi$ at time $t_n$,  of mass $\|\psi^n\|_{L^2}^2=m$,
is known,  the steps of the splitting scheme are as follows.
\begin{itemize}
\item[Step 1:] Solve
  \begin{equation}
    \left \{  \begin{aligned}
                \partial_tv&= -H v + g(|v|^2)v, & t_n\leq t &\leq t_{n+1},\\
                v(t_n)&=\psi^n.
              \end{aligned}
            \right .\label{eq:split_step1}
          \end{equation}
\item[Step 2:] Solve
          \begin{equation}
            \left \{  \begin{aligned}
                        \partial_tw&= \mu_{m}(w)w, & t_n\leq t \leq t_{n+1},\\
                        w(t_n)&=v(t_{n+1}).
                      \end{aligned}
                    \right .\label{eq:split_step2}
                  \end{equation}
                \end{itemize}
After the two steps,  we simply define $\psi^{n+1}=w(t_{n+1})$. 

Step 1 requires to solve a nonlinear parabolic type partial
differential equation. Following~\cite{BaDu04}, we approximate~\eqref{eq:split_step1} by a semi-implicit time discretization:
\begin{equation}
  \label{eq:split_step1_disc}
  \frac{v^{n+1}-\psi^n}{\delta t} = -H v^{n+1}+g(|\psi^n|^2)v^{n+1}.
\end{equation}
So, we have $\varphi^{n+1}=v^{n+1}$.
The interest of having a semi-implicit scheme stems from its stability property.

The equation involved in Step 2 is an ordinary differential equation. In~\cite{BaDu04}, its solution is approximated by
\begin{equation}
  \label{eq:sol_split_step2}
  w^{n+1}=\sqrt{m}\frac{\varphi^{n+1}}{\|\varphi^{n+1}\|_{L^2}}.
\end{equation}

The coupling of~\eqref{eq:split_step1_disc} and~\eqref{eq:sol_split_step2} leads
to the GFDN method. It is not totally obvious that~\eqref{eq:sol_split_step2} is actually an
approximation of the solution $w(t_{n+1})$ to~\eqref{eq:split_step2}.
The normalization part~\eqref{eq:sol_split_step2} is actually
equivalent to solving the ordinary differential equation
\begin{equation*}
\left \{  \begin{aligned}
    \partial_t\rho&=\nu_{n,m}(\delta t) \rho,& t_n<t<t_{n+1},\\
    \rho(t_n)&=\varphi^{n+1},
          \end{aligned}
          \right .
\end{equation*}
where
\begin{equation*}
  \nu_{n,m}(\delta t)=\frac{\ln(m)-\ln\left(
    \|\varphi^{n+1}\|_{L^2}^2\right)}{2\delta t} .
\end{equation*}
We define the piecewise function
\[
  \tilde{\mu}_{m}(t,\delta t)= \sum_{n = 0}^{+\infty} \nu_{n,m}(\delta t)\mathbf{1}_{[t_n,t_{n+1})}(t).
\]
With this definition, the gradient flow with discrete normalization method~\eqref{eq:GFDN} is an approximation of
\begin{equation}
  \label{eq:gfdnm2}
  \left \{
    \begin{aligned}
      \partial_t\varphi&= -H\varphi + g(|\varphi|^2)\varphi, &\varphi(t_n)&=\psi(t_n), &t_n<t<t_{n+1},\\
      \partial_t\rho&=\tilde{\mu}_{m}(t,\delta t) \rho,& \rho(t_n)&=\varphi(t_{n+1}), & t_n<t<t_{n+1},
    \end{aligned}
  \right .
\end{equation}
with $\|\psi(t_n)\|_{L^2}^2=m$.
Actually, the system~\eqref{eq:gfdnm2} has to be read as the Lie splitting
approximation of 
\begin{equation}
  \label{eq:CNGFM2}
  \partial_t\Upsilon=-H \Upsilon + g(|\Upsilon|^2)\Upsilon + \tilde{\mu}_{m}(t,\delta t) \Upsilon,\quad \Upsilon(t=0)=\psi_0,
\end{equation}
and $\|\psi_0\|^2_{L^2}=m$. Thus, it remains
to make the link between~\eqref{eq:CNGFM2} and~\eqref{eq:CNGF2} by 
determining the limit of $\tilde{\mu}_{m}(t,\delta t)$ when $\delta t$ goes to $0$. Let us define $t^*=t_n$ that remains constant when $\delta t \to 0$ and $n \to \infty$. For $t^*\leq t < t^*+\delta t$, we have
    \begin{align*}
    \tilde{\mu}_{m}(t,\delta t) &= 
    -\frac12  \frac{\mathrm{ln}
    \|\varphi(t^*+\delta t)\|_{L^2}^2-\mathrm{ln}(m)}{\delta t}\\[2mm]
&= 
-\frac{1}{2} \frac{\mathrm{ln}
\|\varphi(t^*+\delta t)\|_{L^2}^2 - \mathrm{ln}
    \|\varphi(t^*)\|_{L^2}^2}{\delta t}.
      \\
&=-\frac12 \frac{\displaystyle \frac{d}{ds}\left(\|\varphi(t^*+s)\|_{L^2}^2\right)_{|s=0 }}{\|\varphi(t^*)\|_{L^2}^2} + {O}(\delta t).      
  \end{align*}
Since $s\mapsto \varphi(t^*+s)$ is solution to $\partial_s\varphi=
-E'(\varphi)$, for $0<s<\delta t$ we have 
  \begin{align*}
    \displaystyle \frac{d}{ds}\|\varphi(t^*+s)\|_{L^2}^2 
    & \displaystyle = -2 \left \langle
      E'(\varphi(t^*+s)),\varphi(t^*+s) \right \rangle.
  \end{align*}
Consequently, for $t^*\leq t <
t^*+\delta t$, we have
  \begin{align*}
\tilde{\mu}_{m}(t,\delta t)&  = 
  \frac{\left \langle
      E'(\varphi(t^*)),\varphi(t^*) \right \rangle}{\|\varphi(t^*)\|_{L^2}^2}+O(\delta t) \\
                           & = \frac{\left \langle
      E'(\psi(t^*)),\psi(t^*) \right \rangle}{m} + O(\delta t)\\ &=\mu_m(\psi)+O(\delta t).
  \end{align*}
Thus, we conclude that~\eqref{eq:CNGFM2} is an approximation of~\eqref{eq:CNGF2}. This finishes the proof. 
\end{proof}

The complete Gradient Flow with Discrete Normalization algorithm is therefore
\begin{equation}
  \label{eq:algoGFDN}
  \begin{array}{l}
    \psi^0=\psi_0, \text{ such that } \|\psi^0\|_{L^2}^2=m,\\
    n=0.\\
    \text{Repeat}\\
    \qquad \left |
    \begin{array}{l}
      \displaystyle \text{Solve } \left(\textrm{Id}+\delta t \left(H-g\left(|\psi^n|^2\right)\right)\right) \varphi^{n+1}=\psi^n,\\
      \displaystyle
      \psi^{n+1}=\sqrt{m}\frac{\varphi^{n+1}}{\|\varphi^{n+1}\|_{L^2}},\\
      n=n+1,
    \end{array} \right .\\
  \text{until } \|\psi^{n+1}-\psi^n\|_{L^2}<\varepsilon.
  \end{array}
\end{equation}
where $\textrm{Id}$ is the identity map and $\varepsilon$ is a tolerance value.

\begin{remark}
  Since we use a splitting scheme to discretize the continuous normalized gradient flow~\eqref{eq:CNGF}, it is no longer guaranteed that~\eqref{eq:algoGFDN} is energy diminishing. We observe in numerical experiments that it is actually almost the case.
\end{remark}


\begin{remark}
  It is possible to modify the algorithm~\eqref{eq:algoGFDN} to deal with mass one unknowns. Indeed, let us consider
  \[
    \tilde{\psi}=\frac{\psi}{\|\psi\|_{L^2}}=\frac{\psi}{\sqrt{m}}.
  \]
Then, the algorithm~\eqref{eq:algoGFDN} becomes
\begin{equation*}
  \begin{array}{l}
    \tilde{\psi}^0=\psi_0/\sqrt{m}, \text{ such that } \|\psi^0\|_{L^2}^2=m,\\
    n=0.\\
    \text{Repeat }\\[3mm]
    \qquad \left |
    \begin{array}{l}
      \displaystyle \text{Solve } \left(\textrm{Id}+\delta t \left(H-g\left(m|\tilde{\psi}^n|^2\right)\right)\right) \varphi^{n+1}=\tilde{\psi}^n,\\[1mm]
      \displaystyle
      \tilde{\psi}^{n+1}={\varphi^{n+1}}/{\|\varphi^{n+1}\|_{L^2}},\\
      n=n+1,
    \end{array} \right .\\
    \\
    \text{until } \|\tilde{\psi}^{n+1}-\tilde{\psi}^n\|_{L^2}<\varepsilon,\\[2mm]
    \psi^{N+1}=\sqrt{m}\tilde{\psi}^{N+1}.
  \end{array}
\end{equation*}
\end{remark}

\subsection{Space discretization}
\label{sec:space_disc}

We obtained the discretization in time of the normalized gradient flow in the previous section. To complete the discretization of the flow, we now proceed to the space discretization of the operator $H$. We
recall that $H$ is defined as a Laplace operator on each edge $e\in\mathcal{E}$
with boundary conditions given for each vertex $v$ by
\begin{equation*}
A_v \psi(v) + B_v\psi'(v) = 0,
\end{equation*}
where $\psi(v) = (\psi_e(v))_{e\sim v}$, $\psi'(v) =
(\partial_x\psi_e(v))_{r\sim v}$ are vectors, with $\partial_x \psi_e(v)$ the
outgoing derivative on $e$ at $v$, and $(A_v,B_v)$ are matrices (see Section~\ref{sec:preliminaries}). 

For each edge $e\in\mathcal{E}$, we consider $N_e\in\mathbb{N}^*$ the number of
interior points and $\{x_{e,k}\}_{1\leq k\leq N_e}$ a uniform discretization of
the interval $I_e = [0,l_e]$, \textit{i.e.} 
\begin{equation*}
x_{e,0} : = 0 < x_{e,1}< \ldots<x_{e,N_e}< x_{e,N_e+1} : = l_e,
\end{equation*}
with $x_{e,k+1} - x_{e,k} = l_e/(N_e+1) : = \delta x_e$ for $0\leq k\leq
N_e$ (see Figure \ref{fig:disc_edge}). We denote $v_1$ the vertex at $x_{e,0}$, $v_2$ the one at $x_{e,N_e+1}$
and, for any $\psi\in H^1_D(\mathcal G)$, for all $e\in\mathcal{E}$
and $1\leq k\leq N_e$, 
\begin{equation*}
\psi_{e,k} : = \psi_e (x_{e,k}),
\end{equation*}
as well as $\psi_{e,v} : =  \psi_e(y_v)$ for $v\in\{v_1,v_2\}$, where $y_{v_1} =
x_{e,0}$ and $y_{v_2} = x_{e,N_e+1}$. 
\begin{figure}[H]
  \centering
  \begin{tikzpicture}
  \draw (-3,0) -- (3,0);
  \node[label=above:$v_1$] at (-3,0) {$\times$};
  \node[label=above:$v_2$] at (3,0) {$\times$};
  \node[below] at (-3,-0.1) {$x_{e,0}$};
  \node[below] at (-2.25,-0.1) {$x_{e,1}$};
  \node[below] at (-1.5,-0.1) {$x_{e,2}$};
  \node[below] at (3,-0.1) {$x_{e,N_e+1}$};
  \foreach \x in {-2.25,-1.5,...,2.25} {
    \node at (\x,0) {$\bullet$};
  };
\end{tikzpicture}  
  \caption{Discretization mesh of an edge $e\in \mathcal{E}$}
  \label{fig:disc_edge}
\end{figure}
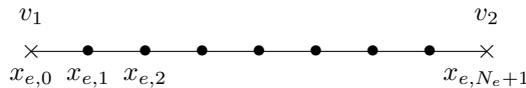

We now assume that $N_e\geq 3$. For any $2\leq k\leq N_e-1$, the
second order approximation of the Laplace operator by finite differences on $e$ is given by
\begin{equation*}
\Delta \psi(x_{e,k}) \approx \frac{\psi_{e,k-1} - 2\psi_{e,k} + \psi_{e,k+1}}{{\delta x_e}^2}.
\end{equation*}
For the case $k = 1$ and $k = N_e$, the approximation requires $\psi_{e,v_1}$
and $\psi_{e,v_2}$ and we have to use the boundary conditions in order to
evaluate them. We use second order finite differences to approximate them as
well. For $-2\leq j\leq 0$, we denote 
\begin{align*}
\psi_{e,v_1,j} = \psi_{e}(x_{e,|j|})\quad\textrm{and}\quad\psi_{e,v_2,j} = \psi_{e}(x_{e,N_e+j}).
\end{align*}
We have the approximation of the outgoing derivative from $e$ at $v\in\{v_1,v_2\}$
\begin{equation*}
\psi_e'(x_{e,v}) \approx \frac{3 \psi_{e,v,0} - 4 \psi_{e,v,-1} + \psi_{e,v,-2}}{2\delta x_e}.
\end{equation*}
Assuming that $\delta x = \delta x_e$ for every edge $e\in \mathcal{E}$ to simplify the presentation, this leads to the approximation of the boundary conditions
\begin{equation*}
A_v \psi_{v,0} + B_v\left(\frac{3\psi_{v,0} - 4 \psi_{v,-1} + \psi_{v,-2}}{2\delta x}\right)= 0,
\end{equation*}
where $\psi_{v,j} = (\psi_{e,v,j})_{e\sim v}$. Assuming that $2\delta x A_v + 3 B_v$ is invertible, this is equivalent to
\begin{equation}\label{eq:edgevalue}
\psi_{v,0} = \left(2\delta x A_v + 3B_v \right)^{-1} B_v\left(4 \psi_{v,-1} - \psi_{v,-2}\right).
\end{equation}
Thus, we can explicitly express the value of $\psi_{e,v_1}$
(resp. $\psi_{e,v_2}$) : it depends linearly on the vectors $\psi_{v_1,-1}$ and
$\psi_{v_1,-2}$ (resp. $\psi_{v_2,-1}$ and $\psi_{v_2,-2}$). It is then possible
to deduce an approximation of the Laplace operator at $x_{e,1}$ and
$x_{e,N_e}$. That is, there exists $(\alpha_{e,v})_{e\sim v}\subset \mathbb{R}$,
for $v\in\{v_1,v_2\}$, such that 
\begin{equation*}
\Delta \psi(x_{e,1}) \approx \frac{\displaystyle \psi_{e,2} - 2\psi_{e,1} + \sum_{e\sim v_1}\alpha_{e,v_1}(4 \psi_{e,v_1,-1} - \psi_{e,v_1,-2})}{{\delta x}^2},
\end{equation*}
and
\begin{equation*}
\Delta \psi(x_{e,N_e}) \approx \frac{\displaystyle \psi_{e,N_e-1} - 2\psi_{e,N_e} + \sum_{e\sim v_2}\alpha_{e,v_2}(4 \psi_{e,v_2,-1} - \psi_{e,v_2,-2})}{{\delta x}^2}.
\end{equation*}
Since $(\psi_{e,v,j})_{-2\leq j\leq 0,v\in\{v_1,v_2\}}$ are interior mesh points
from the other edges, we limit our discretization to the interior mesh
points of the graph. The approximated values of $\psi$ at each
vertex are computed using~\eqref{eq:edgevalue}. We denote $\boldsymbol{\psi} =
(\psi_{e,k})_{1\leq k\leq N_e, e\in\mathcal{E}}$ the vector in $\mathbb{R}^{N_T}$,
with $N_T = \sum_{e\in\mathcal{E}} N_e$, representing the values of $\psi$ at each
interior mesh point of each edge of $\mathcal{G}$. We introduce the matrix
$[\boldsymbol{H}]\in\mathbb{R}^{N_T\times N_T}$ corresponding to the
discretization of $H$ on the interior of each edge of the graph, which yields the
approximation 
\begin{equation*}
H\psi \approx [\boldsymbol{H}]\boldsymbol{\psi}.
\end{equation*}

\subsection{Space-time discretization}

Finally, we obtain the Backward Euler Finite Difference (BEFD) scheme
approximating~\eqref{eq:GFDN}. Let $\boldsymbol{\psi}^0\in\mathbb{R}^{N_T}$. We
compute the sequence $(\boldsymbol{\psi}^n)_{n\geq 0}\subset\mathbb{R}^{N_T}$ given
by 
\begin{equation}
  \label{eq:befd}
  \left\{
    \begin{aligned}
{\boldsymbol{\varphi}}^{n+1}  &= \boldsymbol{\psi}^n - \delta t\left([\boldsymbol{H}]{\boldsymbol{\varphi}}^{n+1} - [g(|\boldsymbol{\psi}^n|^{2})]{\boldsymbol{\varphi}}^{n+1}\right),\\
     \displaystyle \boldsymbol{\psi}^{n+1} &= \sqrt{m}\frac{{\boldsymbol{\varphi}}^{n+1}}{\|{\boldsymbol{\varphi}}^{n+1}\|_{\ell^2}},
\end{aligned}\right.
\end{equation}
where $[g(|\boldsymbol{\psi}^n|^{2})]\in\mathbb{R}^{N_T\times N_T}$ is a diagonal
matrix whose diagonal is the vector $g(|\boldsymbol{\psi}^n|^2)$ and
$\|{\boldsymbol{\varphi}}^{n+1}\|_{\ell^2}$ is the usual $\ell^2$-norm on the graph
$\mathcal{G}$ of ${\boldsymbol{\varphi}}^{n+1}$. This scheme has been studied on
rectangular domains (with an additional potential operator) and Dirichlet
boundary conditions~\cite{BaDu04}  and is known to be unconditionally
stable. Since it is implicit, the computation of ${\boldsymbol{\varphi}}^{n+1}$
involves the inversion of a linear system whose matrix is 
\begin{equation*}
[\boldsymbol{M}_n] = [\mathbf{Id}] + \delta t \left([\boldsymbol{H}] - [g(|\boldsymbol{\psi}^n|^{2})] \right),
\end{equation*}
where $[\mathbf{Id}]$ is the identity matrix, and right-hand-side is $\boldsymbol{\psi}^n$. Using the matrix $[\mathbf{H}]$,
we may also compute the energy. For instance, in the case where $g(z)
=z$, by using the standard $\ell^2$ inner product on the graph, we
obtain
\begin{equation}
  \label{eq:energy_num}
  E(\boldsymbol{\psi}^n) = \frac12 \scalar{ [\mathbf{H}]
  \boldsymbol{\psi}^n}{\boldsymbol{\psi}^n}_{\ell^2} - \frac{1}{4}
 \scalar{ (\boldsymbol{\psi}^n)^2}{(\boldsymbol{\psi}^n)^2 }_{\ell^2}.
\end{equation}

To illustrate our methodology, we give below an example of a star-graph with $3$
edges (see Figure~\ref{fig:f1f2} (\subref{fig:f1})). The operator $H$ is given with Dirichlet boundary conditions for the
exterior vertices and Kirchoff-Neumann conditions for the central vertex. We can see on Figure~\ref{fig:f1f2} (\subref{fig:f2})
the positions of the non zero coefficients of the corresponding
matrix $[\boldsymbol{H}]$ when the discretization is such that $N_e = 10$, for
each $e\in\mathcal{E}$. The coefficients accounting for the Kirchhoff boundary
condition are the ones not belonging to the tridiagonal component of
the matrix.
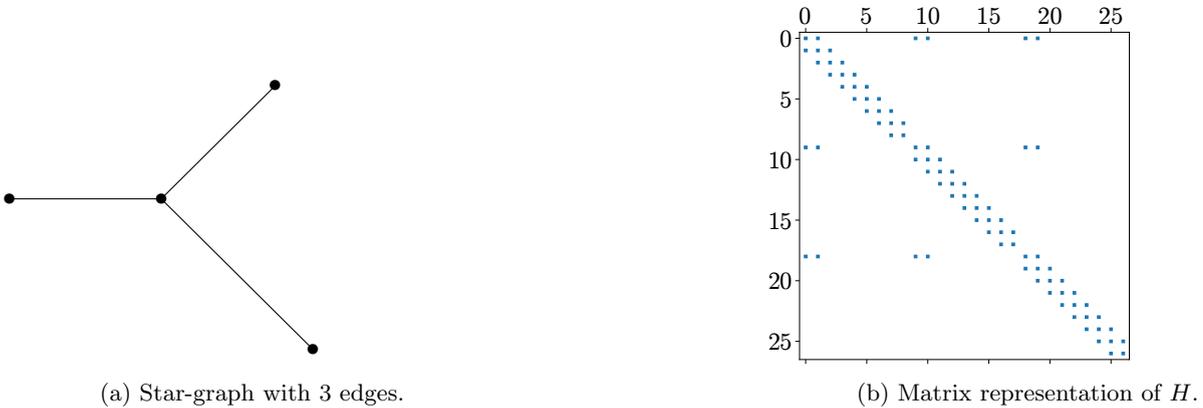
\begin{figure}[!tbp]
  \begin{subfigure}[b]{0.4\textwidth}
  \begin{tikzpicture}
    \coordinate (A) at (-2,0);
    \coordinate (O) at (0,0);
    \coordinate (B) at (2,-2);
    \coordinate (C) at (1.5,1.5);


    \node  at (A) {$\bullet$};
    \node at (O) {$\bullet$};
    \node at (B)  {$\bullet$};
    \node at (C) {$\bullet$};

    \draw (A) -- (O);
    \draw (B) -- (O);
    \draw (C) -- (O);
  \end{tikzpicture}  
  \caption{Star-graph with $3$ edges.}
  \label{fig:f1}
  \end{subfigure}
  \hfill
  \begin{subfigure}[b]{0.4\textwidth}
    \includegraphics[width=0.7\textwidth]{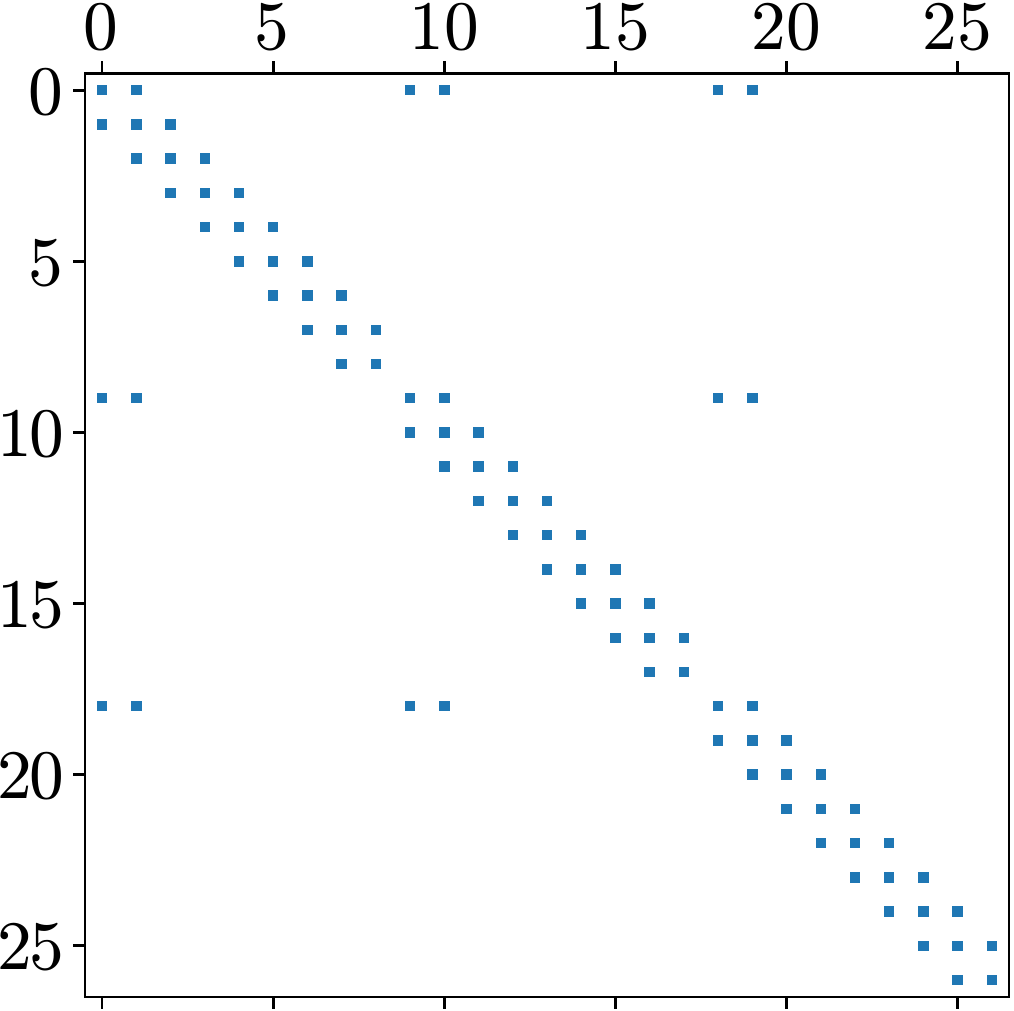}
    \caption{Matrix representation of $H$.}
    \label{fig:f2}
  \end{subfigure}
  \caption{An example for a star-graph.}
\label{fig:f1f2}
\end{figure}

\begin{remark}
  We have implemented this space discretization in the framework of the Grafidi library~\cite{Grafidi}, a Python library which we have developed for the numerical simulation on quantum graph and which is presented in~\cite{2021_BDL_Arxiv}. Note that finite differences on graphs have also been implemented in 
a library developed in Matlab by R. H. Goodman, available in~\cite{GoodmanLib} and which has been used in particular in~\cite{Go19,KaPeGo19}.
\end{remark}

\section{Numerical experiments}
\label{sec:experiments}

We present here various numerical computations of ground states using the Backward
Euler Finite Difference scheme~\eqref{eq:befd}. Even though
the (BEFD) method was built for a general nonlinearity, for simplicity
we focus in this section
on the computations of the ground states of the focusing cubic nonlinear
Schr\"odinger (NLS) equation on
a graph $\mathcal{G}$, that reads
\begin{equation}
  \label{eq:nls2}
  i\partial_t\psi = H \psi - |\psi|^2 \psi.
\end{equation}
Explicit exact solutions are available for (NLS) on various
graphs, in particular star graphs. We use the two-edges star graph in Section~\ref{sec:2edges} to validate 
our implementation of the (BEFD) method and to show its  efficiency to compute ground states. We
present in Section~\ref{sec:non-compact} some numerical results for non compact graphs for which no explicit
solutions are available. More examples are presented in a
 companion paper~\cite{2021_BDL_Arxiv}.

\subsection{Two-edges star-graph}
\label{sec:2edges}
The two-edges star-graph is one of the simplest graph. We
identify the graph $\mathcal{G}=\mathcal{G}_2$ as the collection of two-half lines connected to a
central vertex $A$. Each edge is referred to with index $i=1,2$ (see
Figure~\ref{fig:2_star_graph}). The coordinate of vertex $A$ is therefore both
$x_1=0$ and $x_2=0$. The unknown $\psi$ of~\eqref{eq:nls2} can be
thought as the collection
\[\psi=(\psi_{1},\psi_{2})^T,\]
each function $\psi_{i}$ living on the edge $i=1,2$.
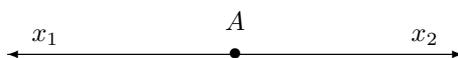
\begin{figure}[H]
  \centering
\begin{tikzpicture}
  \draw[-] (-3,0) -- (3,0);
  \node[label=above:$A$] at (0,0) {$\bullet$};
  \draw[<-,>=latex] (-3,0) -- (0,0);
  \draw[->,>=latex] (0,0) -- (3,0);
  \node[above] at (-2.5,0) {$x_1$};
  \node[above] at (2.5,0) {$x_2$};
\end{tikzpicture}    
  \caption{Two star-graph}
  \label{fig:2_star_graph}
\end{figure}
\subsubsection{Kirchhoff condition}
The ground state of the cubic nonlinear Schr\"odinger equation~\eqref{eq:nls2}
on the real line is known to be the soliton. To compute it on a two-edges
star-graph, we identify the real line $\mathbb{R}$ to the graph $\mathcal{G}_2$ with Kirchhoff condition at the vertex located at $x=0$ (see~\cite{AdCaFiNo12}). The Kirchhoff condition on $\mathbb{R}$ is
\begin{equation*}
  \psi(0^-)=\psi(0^+),\qquad \psi'(0^-)=\psi'(0^+),
\end{equation*}
with $\psi'$ denoting the usual forward derivative, whereas on $\mathcal{G}_2$,
it is
\begin{equation*}
  \psi_1(0)=\psi_2(0),\qquad \psi_1'(0)+\psi_2'(0)=0.
\end{equation*}
The energy is
\[
  E_{\text{NLS}}(\psi)=\frac12 \|\psi'
  \|_{L^2(\mathbb{R})}^2-\frac{1}{4}\|\psi\|_{L^4(\mathbb{R})}^4, \quad
  \psi \in H^1(\mathbb{R}),
\]
or similarly
\[
  E_{\text{NLS}}(\psi)=\sum_{i=1}^2\left(\frac12 \|\psi_i'
  \|_{L^2(\mathbb{R}_{x_i}^+)}^2-\frac{1}{4}\|\psi_i\|_{L^4(\mathbb{R}_{x_i}^+)}^4\right), \quad
  \psi \in H^1(\mathcal{G}_2).
\]
The minimum of the functional $E_{\text{NLS}}$ among functions of
$H^1(\mathbb{R})$ with squared $L^2$-norm equal to $m>0$ is given (up
to phase and translation) by
\begin{equation}
  \label{eq:soliton}
  \phi_m(x)=\frac{m}{2\sqrt{2}}\frac{1}{\text{cosh}({m}x/4)},
\end{equation}
and
\begin{equation*}
  E_{\text{NLS}}(\phi_m)=-\frac{m^3}{96}.
\end{equation*}
In order to simulate the two semi-infinite edges originated from the central
vertex, we consider two finite edges of length $40$. The graph is presented on
Figure~\ref{fig:graph_2Edges} (left).
\begin{figure}[htbp!]
  \centering
  \begin{tabular}{cc}
    \begin{tabular}{c}
  \includegraphics[width=.42\textwidth]{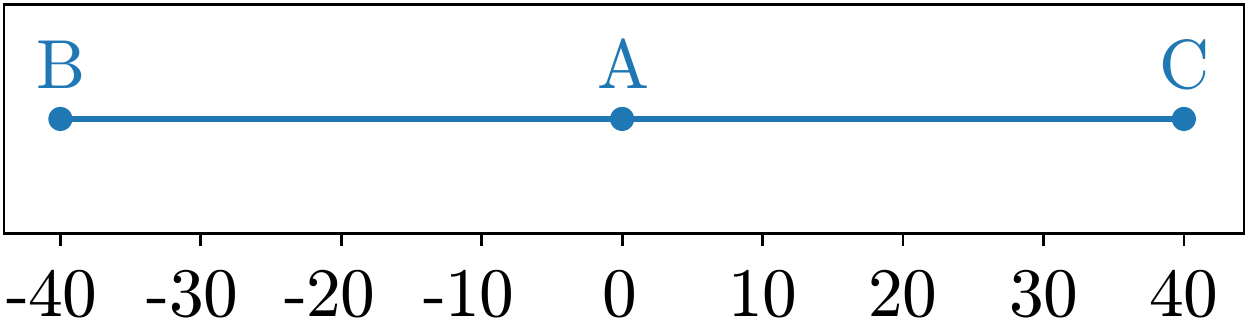}
    \end{tabular}
    &
      \begin{tabular}{c}
      \includegraphics[width=.38\textwidth]{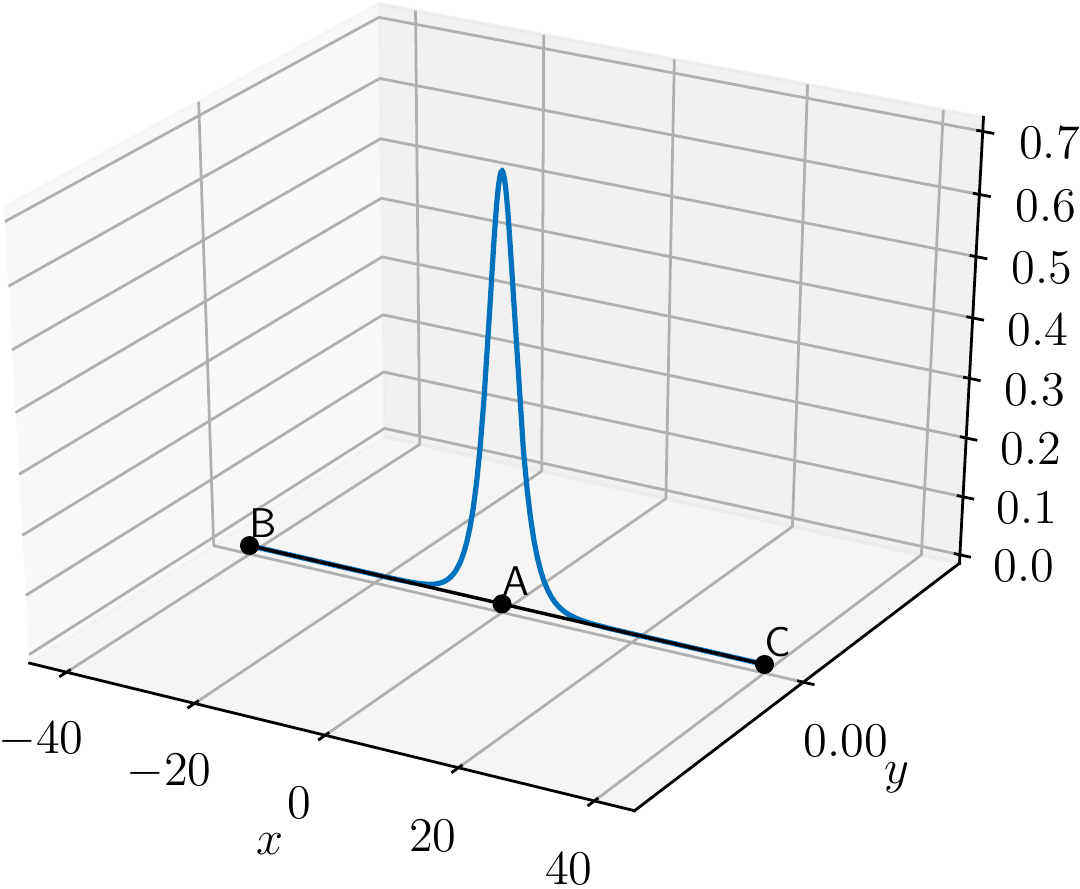}
      \end{tabular}
    \\
  \end{tabular}
  \caption{Two-edges graph (left) and the exact solution $\phi_m(x)$ for
    $m=2$ (right)}
  \label{fig:graph_2Edges}
\end{figure}
\FloatBarrier
We discretize each edge with $N_e=4000$ nodes and set homogeneous Dirichlet boundary
conditions at the external vertices ($B$ and $C$ on
Figure~\ref{fig:graph_2Edges}). The time step is $\delta t=10^{-2}$. The
mass is $m=2$. The exact solution is plotted on
Figure~\ref{fig:graph_2Edges} (right). 
The initial datum is chosen as a Gaussian of mass $m/2$ on each edge, 
namely
\[
  \psi_0(x)=\sqrt{\frac{10 m}{\sqrt{5\pi}}} e^{-10x^2}.
\]
We plot on Figure~\ref{fig:soliton} both the exact solution $\phi_m$~\eqref{eq:soliton} and the numerical one $\phi_{m,\text{num}}$
obtained after $3000$ iterations (left), as well as the error
$|\phi_m-\phi_{m,\text{num}}|$ (right). The error is plotted for a fixed $\delta x$. We discuss the variation of the error with respect to $\delta x$ in Figure \ref{fig:conv_curves} and observe that the scheme is of order $2$.
\begin{figure}[htbp!]
  \centering
  \begin{tabular}{cc}
    \includegraphics[width=.38\textwidth]{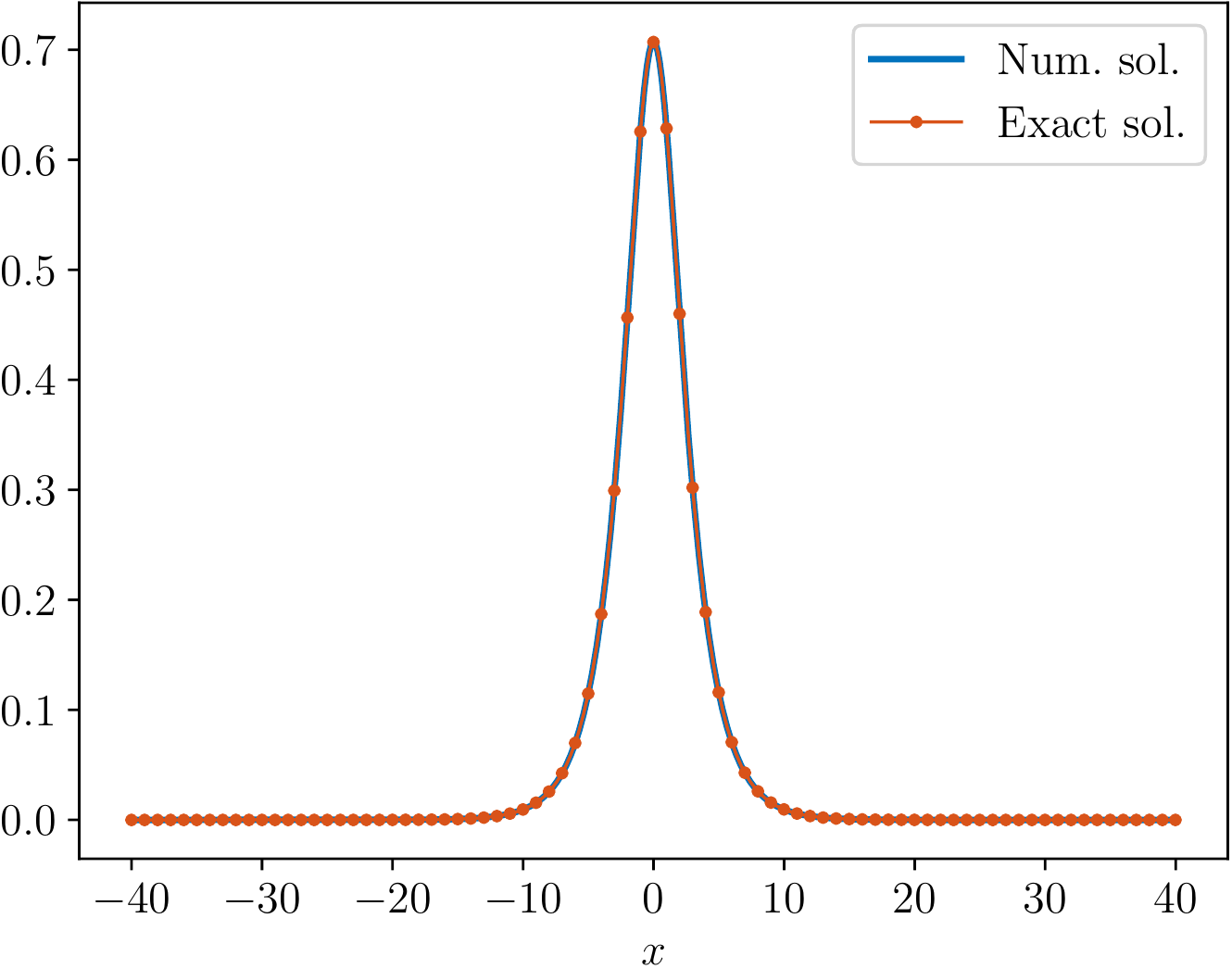} &
    \includegraphics[width=.38\textwidth]{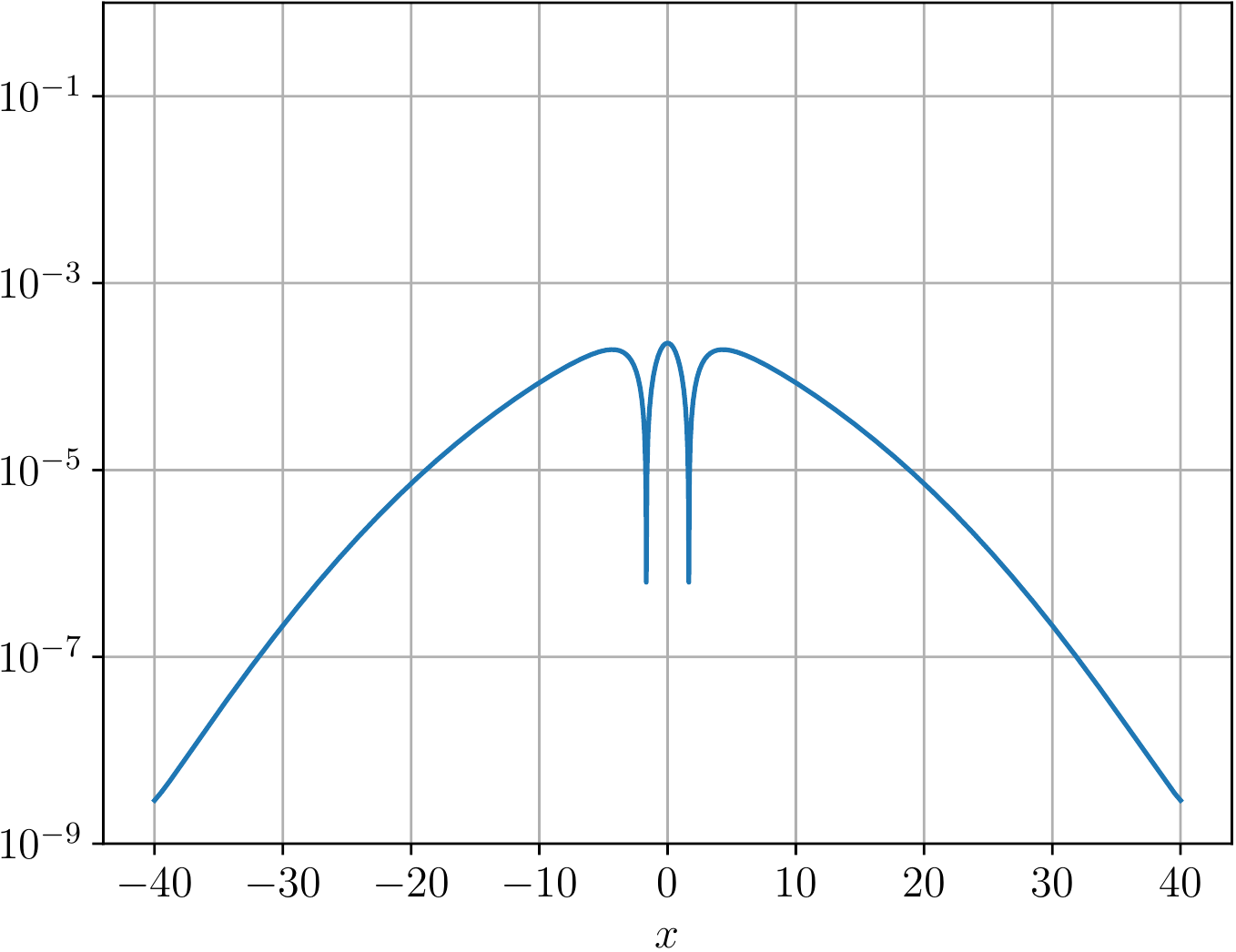}\\
    $\phi_{m}$ and $\phi_{m,\text{num}}$ & $|\phi_m-\phi_{m,\text{num}}|$
  \end{tabular}
  \caption{Comparison between $\phi_{m}$ and $\phi_{m,\text{num}}$}
  \label{fig:soliton}
\end{figure}
We obtain a very close numerical solution. Since the initial data is symmetric and centered on $0$, our solution is also symmetric and centered on $0$.

The (BEFD) method allows to compute the exact
energy and show that the scheme is energy diminishing. 
We plot in Figure~\ref{fig:soliton_energy} the evolution of the numerical energy
using~\eqref{eq:energy_num}
and the comparison with the exact energy. The scheme is clearly energy
diminishing and we obtain a very good agreement with the exact energy.
\begin{figure}[htbp!]
  \centering
  \includegraphics[width=.38\textwidth]{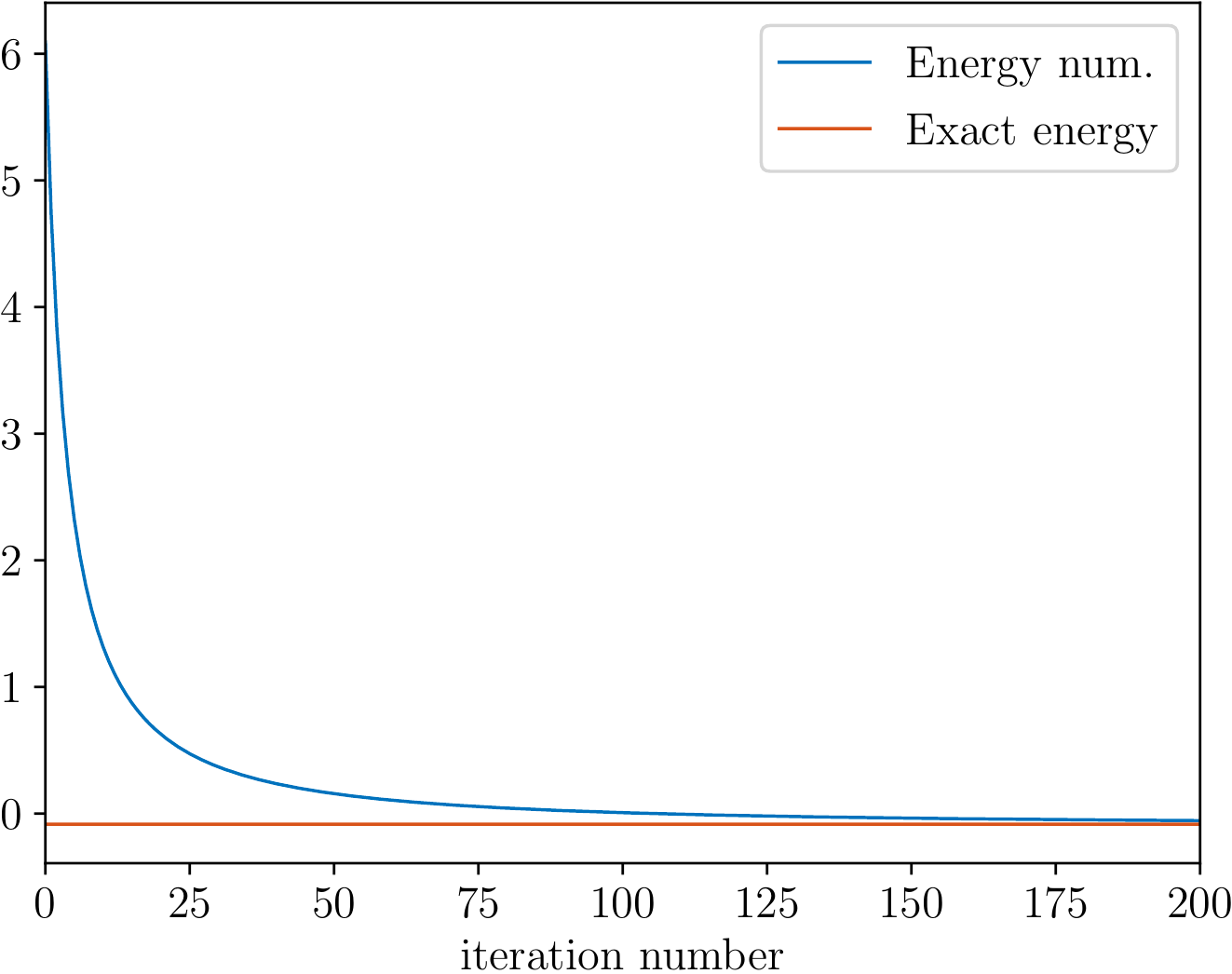}
  \caption{Evolution of the energy when computing the ground state of~\eqref{eq:nls2} for
    $x\in \mathbb{R}$ compared to $E_{\text{NLS}}$.}
  \label{fig:soliton_energy}
\end{figure}

\subsubsection{$\delta$-condition}
We consider now a $\delta$-condition at the central vertex $A$ of the graph
$\mathcal{G}_2$.
The unknown
$\psi_\delta=(\psi_{\delta,1},\psi_{\delta,2})^T$ is the collection
of $\psi_{\delta,i}$ living on each edge $i=1,2$.
Recall that the boundary conditions at $A$ are
\begin{equation*}
  \psi_{\delta,1}(0)=\psi_{\delta,2}(0),\qquad \psi_{\delta,1}'(0)+\psi_{\delta,2}'(0)=\alpha \psi_{\delta,1}(0).
\end{equation*}
The parameter $\alpha$ is interpreted as the strength of the $\delta$
potential and we focus on the attractive case ($\alpha<0$). The mass and energy are
\[
  M(\psi_\delta)=\sum_{i=1}^2\int_{\mathbb{R}_{x_i}^+}
  |\psi_{\delta,i}(x_i)|^2\, dx_i,
\]
and
\[
  E_{\delta}(\psi_\delta)=\sum_{i=1}^2 \left(\int_{\mathbb{R}_{x_i}^+}
    \frac{|\psi_{\delta,i}'(x_i)|^2}{2} - \frac{|\psi_{\delta,i}(x_i)|^4}{4}\, dx_i + \frac{\alpha}{4}|\psi_{\delta,i}(0)|^2\right).
\]
Explicit ground state solutions were provided in~\cite{GoHoWe04} in
the cubic case (see also~\cite{adami2012stationary,AdNoVi13} for the
general case). Define $a$ by
\begin{equation*}
  a=\frac{1}{\sqrt{\omega}} \text{arctanh} \left (
    \frac{|\alpha|}{2\sqrt{\omega}}\right),
\end{equation*}
and define the function
$\phi_\delta=(\phi_{\delta,1},\phi_{\delta,2})$ by
\begin{equation*}
  \phi_{\delta,i}(x_i)=\frac{\sqrt{2\omega}}{\displaystyle \text{cosh}\left(\sqrt{\omega}\left(x_i-\frac{\alpha}{|\alpha|}a\right)\right)}.
\end{equation*}
The mass of $\phi_{\delta}$ is explicitly given by
\[
m_{\delta}=M(\phi_{\delta})=4\sqrt{\omega}+2\alpha,
  \]
and the function $\phi_{\delta}$ has been constructed so that it is
the minimizer of $E_{\delta}(\psi_\delta)$ with constrained mass $m_{\delta}$.
The energy might be explicitly calculated :
\[
  E_\delta(\phi_\delta)=-\frac{2}{3}\omega^{\frac32}-\frac{\alpha^3}{12} =
-\frac{m_\delta^3}{96} + \frac{m_\delta^2\alpha}{16} - \frac{m_\delta\alpha^2}{8}.
\]
Like in the previous section, we apply the (BEFD) method to compute the
ground state. We take the same numerical parameters concerning the mesh size and
the approximation graph of Figure~\ref{fig:graph_2Edges} (left). The numerical solution
compared to the exact one with $\omega=1$, $\alpha=-1$ and therefore mass
$m_\delta=2$ is presented on Figure~\ref{fig:compar_delta_2Edges} (left). The initial data are Gaussian on both edges equal
to $\psi_0(x_i)=\rho e^{-10  x_i^2}$, $i=1,2$, with $\rho>0$ such that
$M(\psi_0)=m_\delta$. In
order to focus close to the vertex~$A$, we choose to plot these solutions on
$[-10,10]$. Once again, the numerical solution is very close to the exact
ground state.
\begin{figure}[htbp!]
  \centering
  \begin{tabular}{cc}
    \includegraphics[width=.38\textwidth]{comparaison_GS_delta_attractive_2Edges-crop.pdf} &
    \includegraphics[width=.38\textwidth]{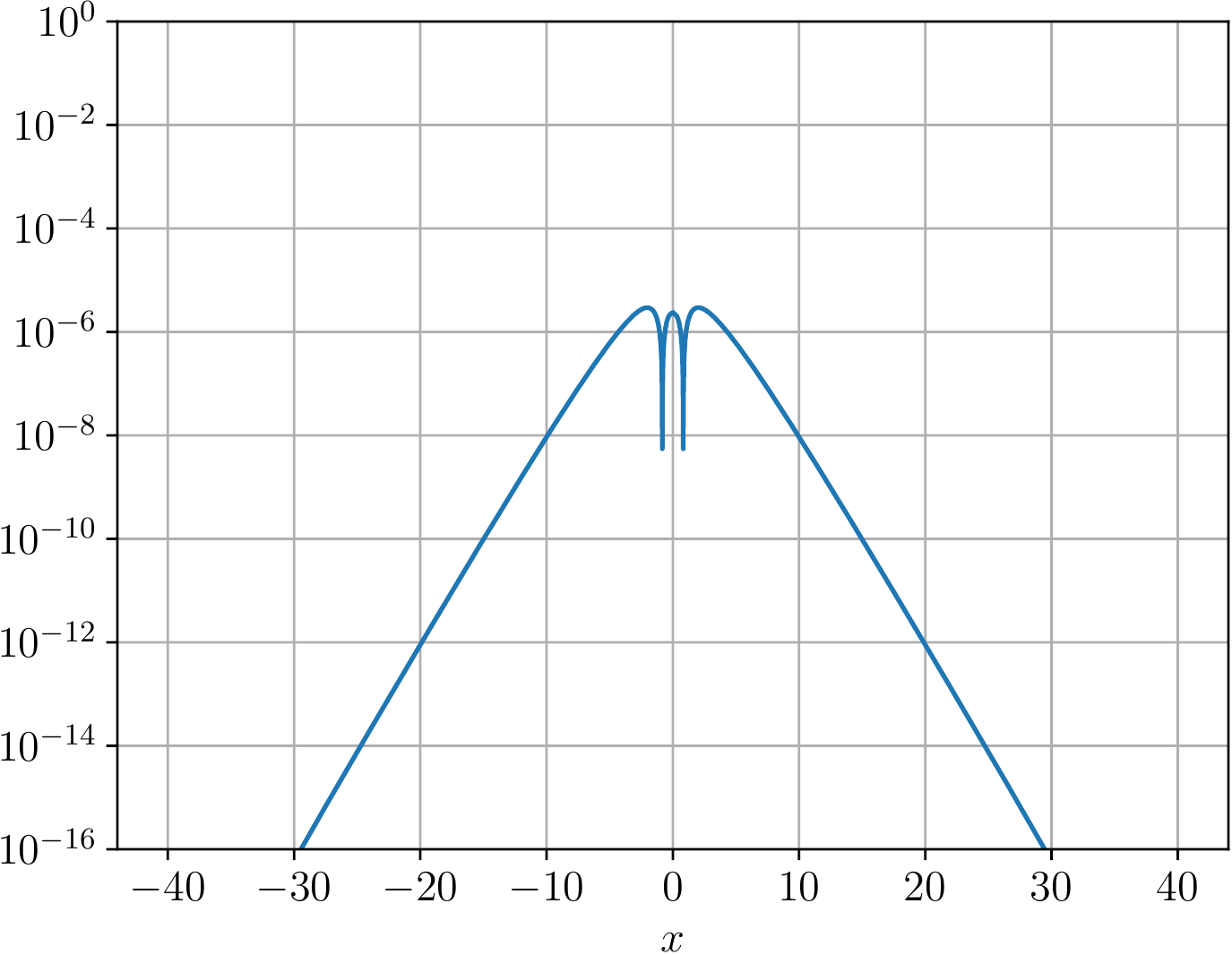} \\
    $\phi_{\delta}$ and $\phi_{\delta,\text{num}}$ & $|\phi_{\delta}-\phi_{\delta,\text{num}}|$\\
  \end{tabular}
  \caption{Comparison of $\phi_{\delta}$ and $\phi_{\delta,\text{num}}$ for $\delta$
    interaction, $\omega=1$ and $\alpha=-1$.}
  \label{fig:compar_delta_2Edges}
\end{figure}
A closer look on the error function
$|\phi_{\delta}-\phi_{\delta,\text{num}}|$ in
logarithmic scale (see Figure~\ref{fig:compar_delta_2Edges}, right) confirms the
accuracy of the numerical solution.
Finally, we plot the evolution of the energy on Figure~\ref{fig:compar_energy_delta_2Edges}. We restrict ourselves to $1000$ iterations on the horizontal
axis since the convergence is really fast. The agreement with the exact energy
is notable.
\begin{figure}[htbp!]
  \centering
  \includegraphics[width=.38\textwidth]{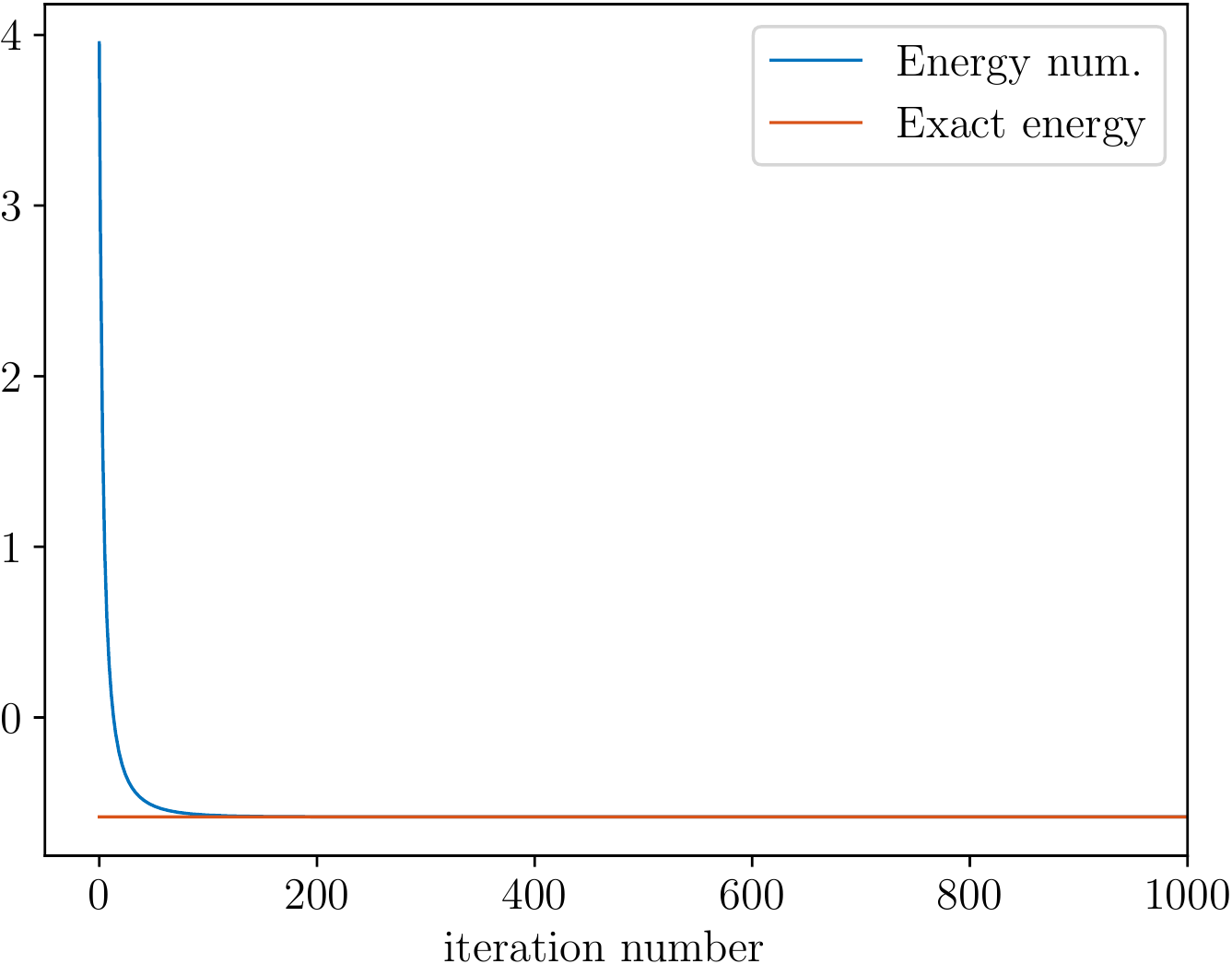}
  \caption{Evolution of the energy when computing the ground state of~\eqref{eq:nls2} with $\delta$ condition compared to $E_{\delta}$.}
  \label{fig:compar_energy_delta_2Edges}
\end{figure}
\FloatBarrier
Using the exact solutions when considering Kirchhoff and $\delta$ conditions, we
are able to evaluate the order of the numerical scheme with respect to the
spatial mesh size. As it was described in Section~\ref{sec:space_disc}, the
scheme should be of second order in space. To confirm this, we make various
simulations for different mesh sizes $\delta x$ and present the results in Figure~\ref{fig:conv_curves} for both Kirchhoff and $\delta$ conditions. In the two
cases, the order of convergence is $2$, as expected.
\begin{figure}[htbp!]
  \centering
  \begin{tabular}{cc}
    \includegraphics[width=.38\textwidth]{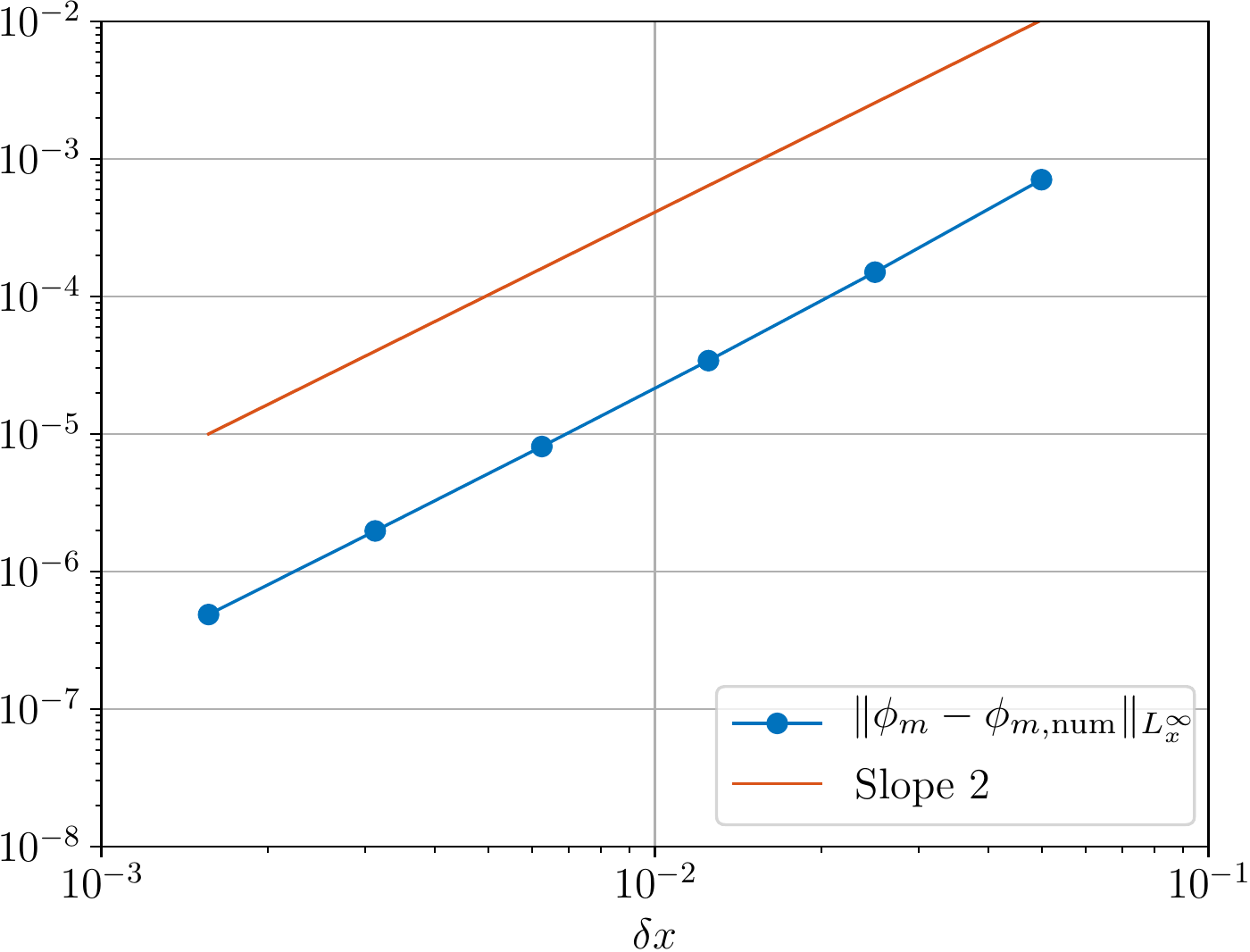} &
    \includegraphics[width=.38\textwidth]{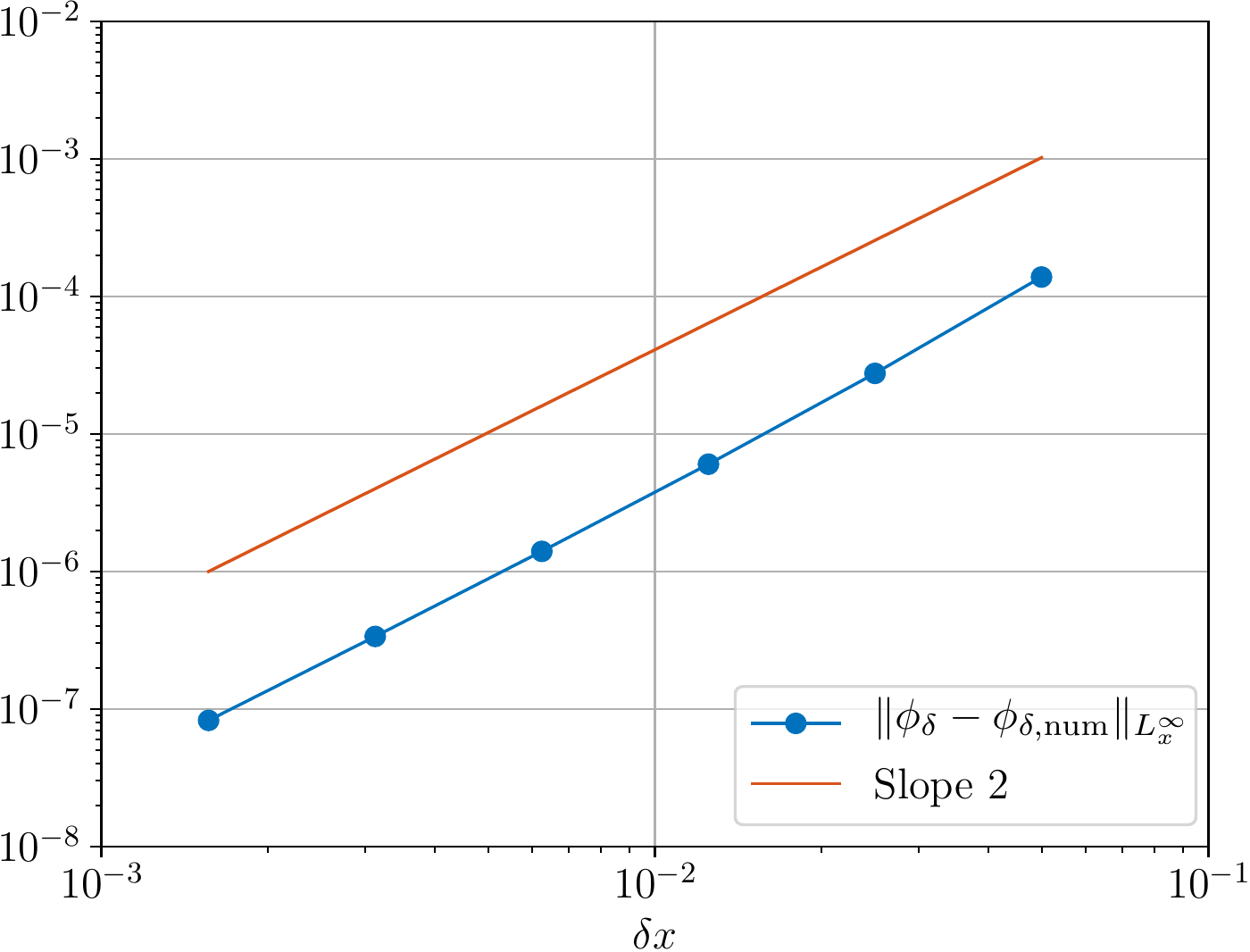}
  \end{tabular}
  \caption{Convergence curves for Kirchhoff (left) and $\delta$ (right) conditions.}
  \label{fig:conv_curves}
\end{figure}
\subsubsection{$\delta'$-condition}
The $\delta'$-condition on star graph is usually defined by
interchanging functions and their derivatives in the definition of the
$\delta$-condition (see e.g.~\cite{BeKu13}).
We prefer here to use the concept of $\delta'$ on the graph corresponding to the $\delta'$
interaction on the line, and give a precise definition in what follows.
As in the previous
section, the unknown 
$\psi_{\delta'}=(\psi_{\delta',1},\psi_{\delta',2})^T$ is the
collection of $\psi_{\delta',i}$ living on each edge $i=1,2$.
The boundary conditions at $A$ are
\begin{equation}
  \label{eq:cond_deltaprime}
    \psi_{\delta',1}(0)=\psi_{\delta',2}(0)+\beta \psi_{\delta',2}'(0) ,\qquad \psi_{\delta',1}'(0)+\psi_{\delta',2}'(0)=0,
\end{equation}
with $\beta>0$.
The mass and energy are
\[
  M(\psi_{\delta'})=\sum_{i=1}^2\int_{\mathbb{R}_{x_i}^+}
  |\psi_{\delta',i}(x_i)|^2\, dx_i
\]
and
\[
  E_{\delta'}(\psi_{\delta'})=\sum_{i=1}^2 \int_{\mathbb{R}_{x_i}^+}\left(
    \frac{|\psi_{\delta,i}'(x_i)|^2}{2} - \frac{|\psi_{\delta',i}(x_i)|^4}{4}\right) dx_i -\frac{1}{2\beta}
    |\psi_{\delta',2}(0)-\psi_{\delta',1}(0)|.
\]
Explicit ground state solutions are provided in~\cite{AdNoVi13}. Let us consider
the transcendental system
\begin{equation}
  \label{eq:trans_sys}
  \left \{
    \begin{array}{l}
      \displaystyle
      \frac{\tanh{(\sqrt{\omega}x_+)}}{\cosh{(\sqrt{\omega}x_+)}}+
      \frac{\tanh{(\sqrt{\omega}x_-)}}{\cosh{(\sqrt{\omega}x_-)}}=0,\\
      \displaystyle
      \frac{1}{\cosh{(\sqrt{\omega}x_+)}}+\frac{1}{\cosh{(\sqrt{\omega}x_-)}}=
      \beta\sqrt{\omega} \frac{\tanh{(\sqrt{\omega}x_+)}}{\cosh{(\sqrt{\omega}x_+)}}.
    \end{array}
  \right .
\end{equation}
We are looking for real solutions such that
\[
  x_-<0<x_+.
\]
When $4/\beta^2<\omega\leq 8/\beta^2$, there exists a unique  couple
$(-\bar x,\bar x)$ 
solution to~\eqref{eq:trans_sys}, where $\bar x$ is  given by
\[
  \bar{x}=\frac{1}{\sqrt{\omega}} \text{arctanh} \left ( \frac{2}{\beta\sqrt{\omega}}\right).
\]
When $8/\beta^2<\omega<+\infty$, 
in addition to the symmetric couple $(-\bar x,\bar
x)$ previously given, we have another, asymmetric,  not explicit,
unique, couple
$(\tilde x_-,\tilde x_+)\in \mathbb{R}^2$ such that
\[
  \tilde x_-<0<\tilde x_+ < |\tilde x_-|.
\]
For brevity in notation, we define
\[
(x_-,x_+)=
\begin{cases}
  (-\bar x,\bar x)&\text{ if }4/\beta^2<\omega\leq 8/\beta^2,\\
  (\tilde x_-,\tilde x_+)&\text{ if }8/\beta^2<\omega<+\infty.
\end{cases}
  \]
The ground state in
both cases is given (up to a phase factor) by
\begin{equation*}
  \phi_{\delta'}(x)=\left \{
    \begin{array}{ll}
      \displaystyle -{\sqrt{2\omega}}/{\cosh(\sqrt{\omega}(x_1+x_-))},& x_1\in
                                                                  [0,+\infty), \\[1mm]
      \displaystyle {\sqrt{2\omega}}/{\cosh(\sqrt{\omega}(x_2+x_+))},& x_2\in
                                                                  [0,+\infty).
    \end{array}
    \right.
\end{equation*}

When $4/\beta^2<\omega\leq 8/\beta^2$, the ground state of~\eqref{eq:nls2} with
boundary condition~\eqref{eq:cond_deltaprime} minimizing the energy
$E_{\delta'}$ with fixed mass $  M(\phi_{\delta'}) $ is an odd function. When
$8/\beta^2<\omega<+\infty$, the ground state is asymmetric.

When $4/\beta^2<\omega\leq 8/\beta^2$, since $\bar{x}=|x_\pm|$, the mass and energy are equal to
\[
  M(\phi_{\delta'})=4\sqrt{\omega}-\frac{8}{\beta},\qquad
  E_{\delta'}(\phi_{\delta'})=\frac23\left(\frac{8}{\beta^3}-\omega^{3/2} \right).
\]
If $8/\beta^2<\omega<+\infty$, the mass and energy are less explicit and are
equal to
\[
  M_{\delta'}(\phi_{\delta'})=
  2\sqrt{\omega}\big(
    2+\tanh(\sqrt{\omega}x_-)-\tanh(\sqrt{\omega}x_+)
  \big),
\]
and
\[
  \begin{array}{rl}
    E_{\delta'}(\phi_{\delta'})& = \displaystyle
    \frac{\omega^{3/2}}{3} \Big(
                                              -2-3(\tanh(\sqrt{\omega}x_-)-\tanh(\sqrt{\omega}x_+))\\
    & \qquad \qquad \qquad +2(\tanh^3(\sqrt{\omega}x_-)-\tanh^3(\sqrt{\omega}x_+))\Big)
    \\
    &\displaystyle -\frac{\omega}{\beta} \left(
      \frac{1}{\cosh(\sqrt{\omega}x_-)} + \frac{1}{\cosh(\sqrt{\omega}x_+)}\right)^2.
  \end{array}
\]
The parameters for the numerical simulations are $\beta=1$, $\delta t=10^{-2}$
and we keep $4000$ nodes per edges to discretize the two-edges graph (see
Figure~\ref{fig:graph_2Edges}, left). The initial data are Gaussian on both edges but
contrary to $\delta$-condition, we select a different sign for the two
edges (to increase the convergence speed).
Namely, $\psi_0(x_1)=-\rho e^{-10  x_1^2}$ and $\psi_0(x_2)=\rho e^{-10
  x_2^2}$, with $\rho >0$ such that
$M_\delta(\psi_0)=M_{\delta'}(\phi_{\delta'})$. In order to
simulate both odd and asymmetric ground states, we select respectively
$\omega=6$ and $\omega=16$. The exact and numerical solutions are plotted in
Figure~\ref{fig:compar_deltaprime}.
\begin{figure}[htbp!]
  \centering
  \begin{tabular}{cc}
    \includegraphics[width=.38\textwidth]{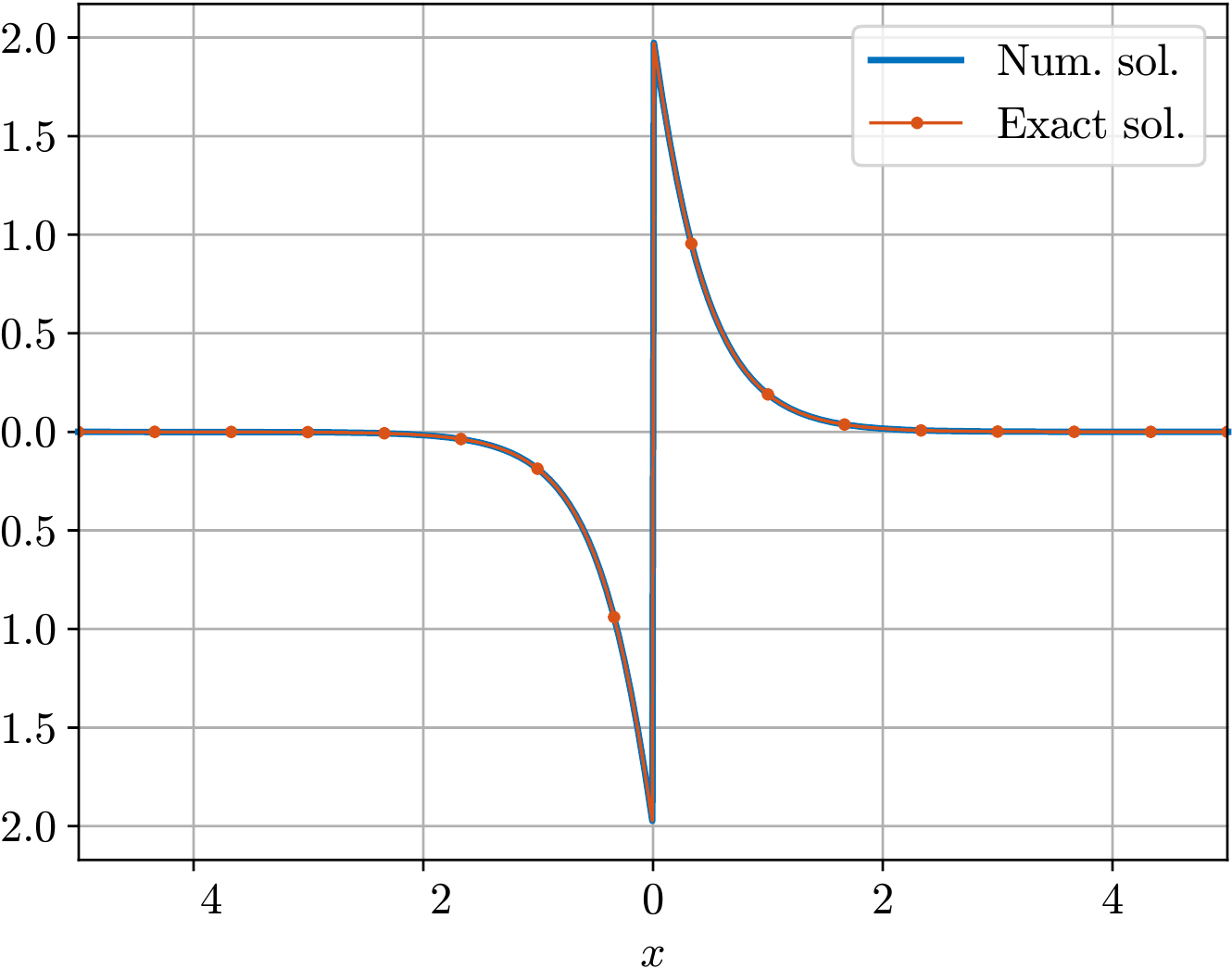} &
    \includegraphics[width=.38\textwidth]{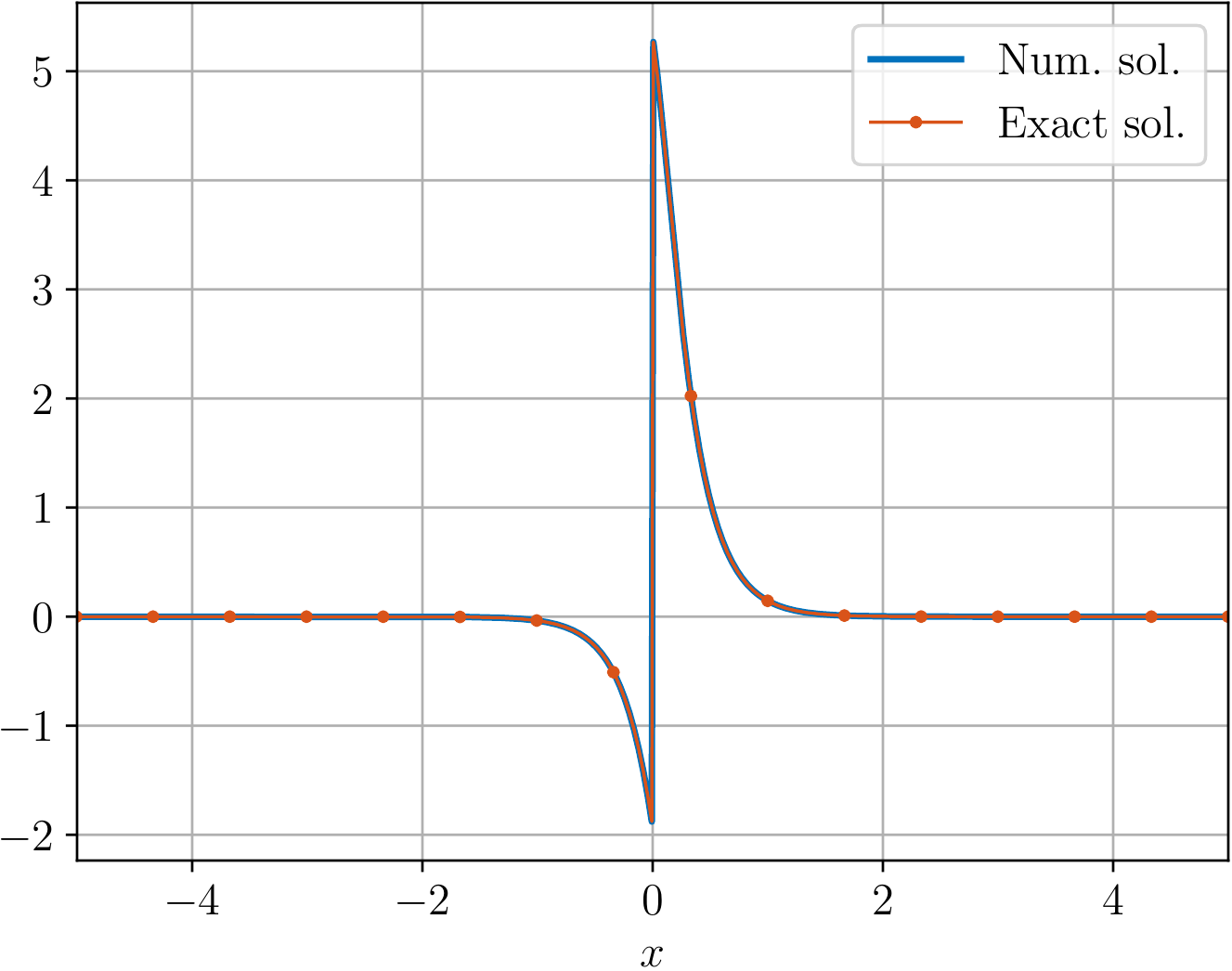} \\
    Odd solution, $\omega=6$ & Asymmetric solution, $\omega=16$
  \end{tabular}
  \caption{Comparison of numerical solutions to ground states for $\delta'$ interaction, $\beta=1$}
  \label{fig:compar_deltaprime}
\end{figure}
When looking to the comparison in logarithmic scale (see
Figure~\ref{fig:compar_log_deltaprime}), we see that we obtain a very good
agreement with the exact solutions.
\begin{figure}[htbp!]
  \centering
  \begin{tabular}{cc}
    \includegraphics[width=.38\textwidth]{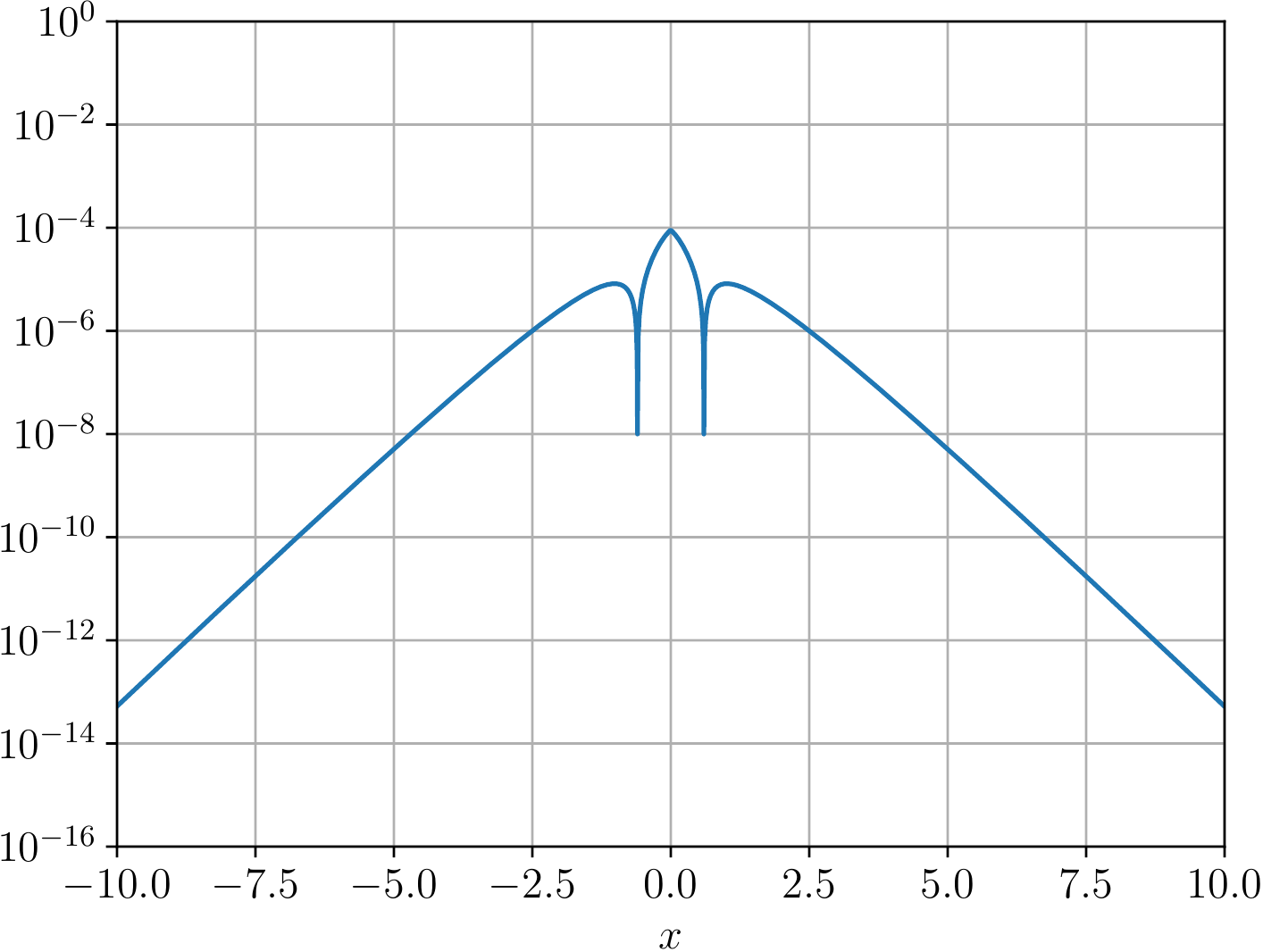} &
    \includegraphics[width=.38\textwidth]{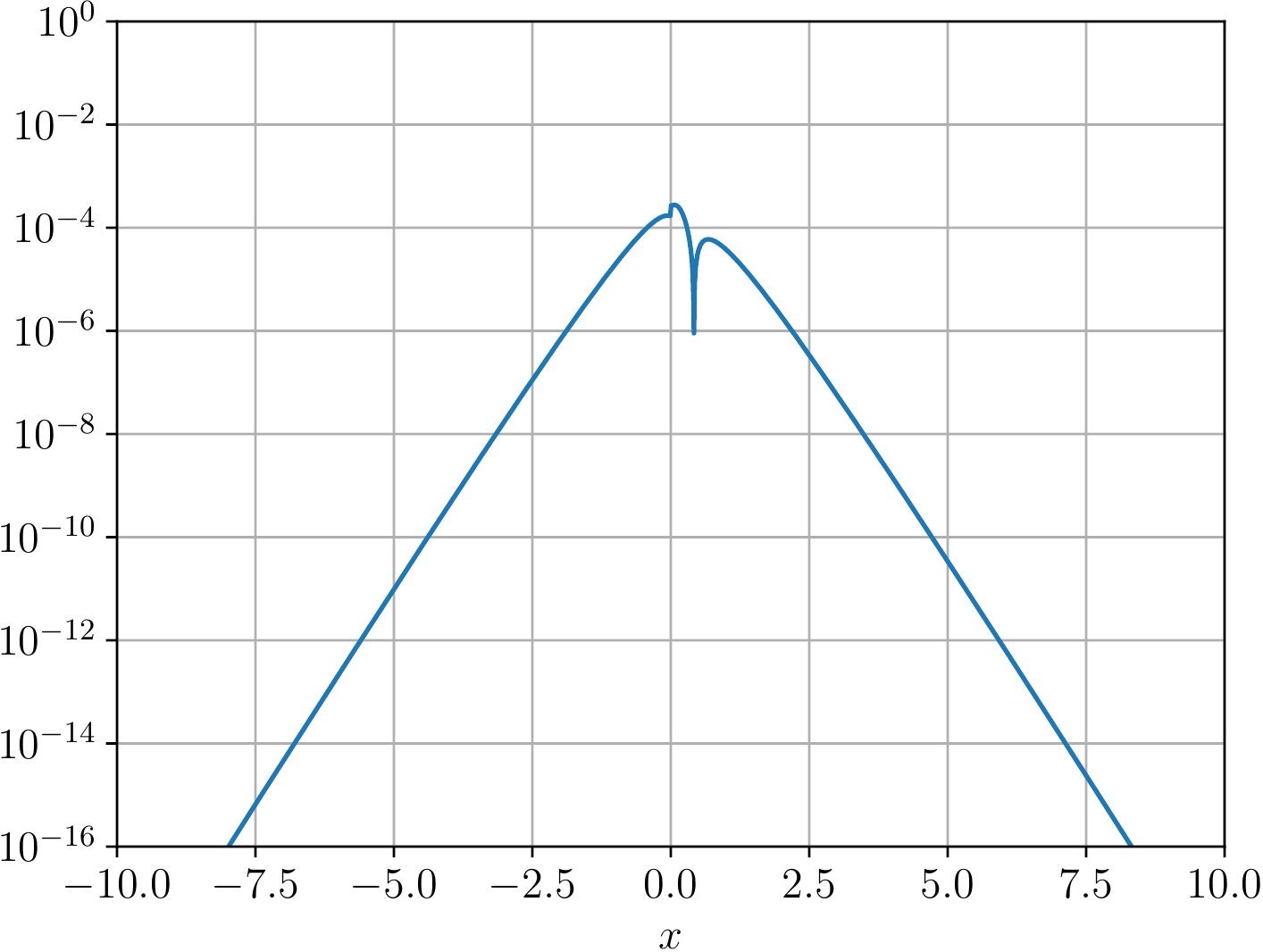} \\
    Odd solution, $\omega=6$ & Asymmetric solution, $\omega=16$
  \end{tabular}
  \caption{Comparison in logarithmic scale of numerical solutions to ground states for $\delta'$ interaction, $\beta=1$}
  \label{fig:compar_log_deltaprime}
\end{figure}
We note that the evolution of the energy (see
Figure~\ref{fig:compar_energy_deltaprime}) during the
minimization process for the asymmetric case is not strictly monotone. After a
first plateau, the algorithm allows to obtain the global minimum (second plateau).
\begin{figure}[htbp!]
  \centering
  \begin{tabular}{cc}
    \includegraphics[width=.38\textwidth]{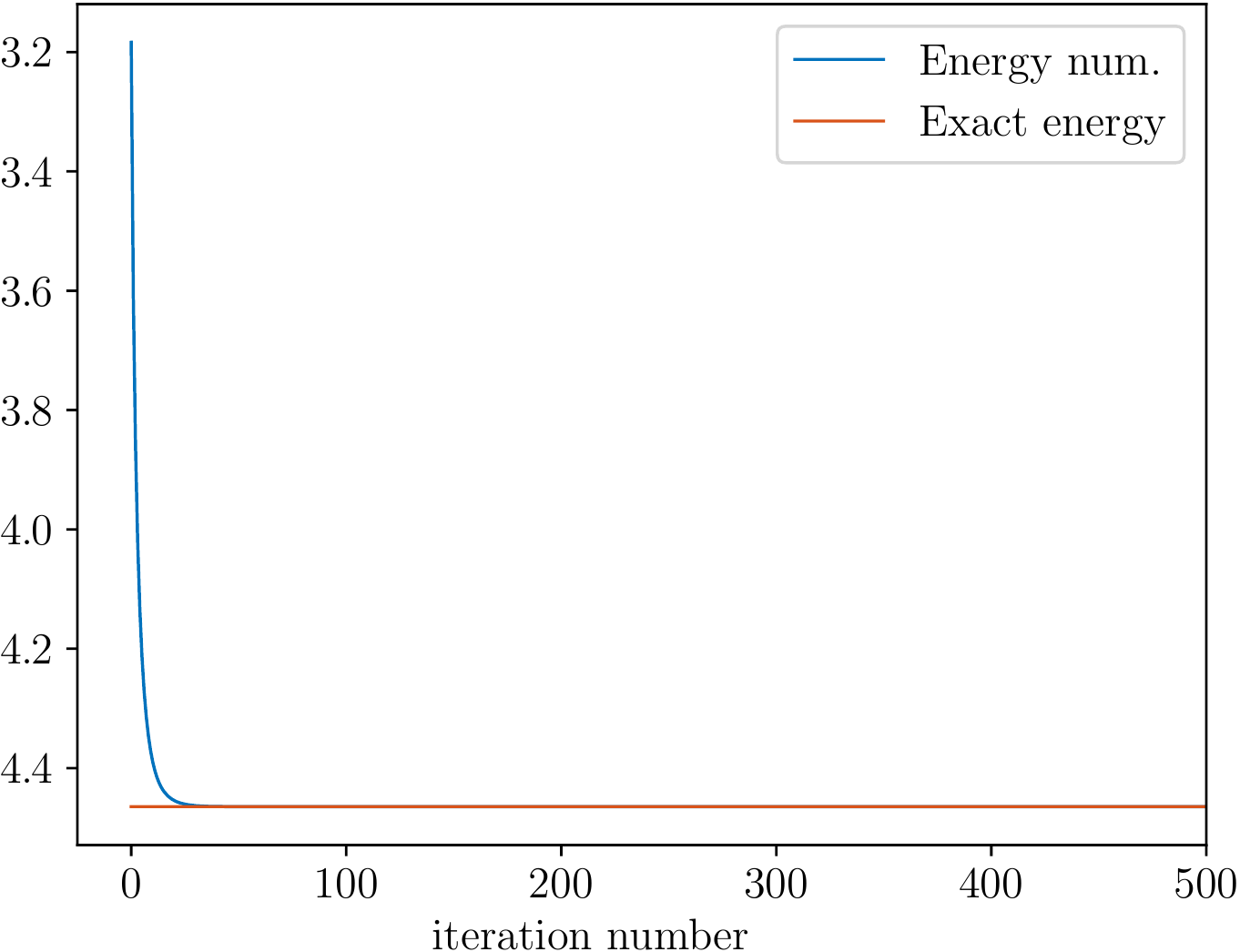} &
    \includegraphics[width=.38\textwidth]{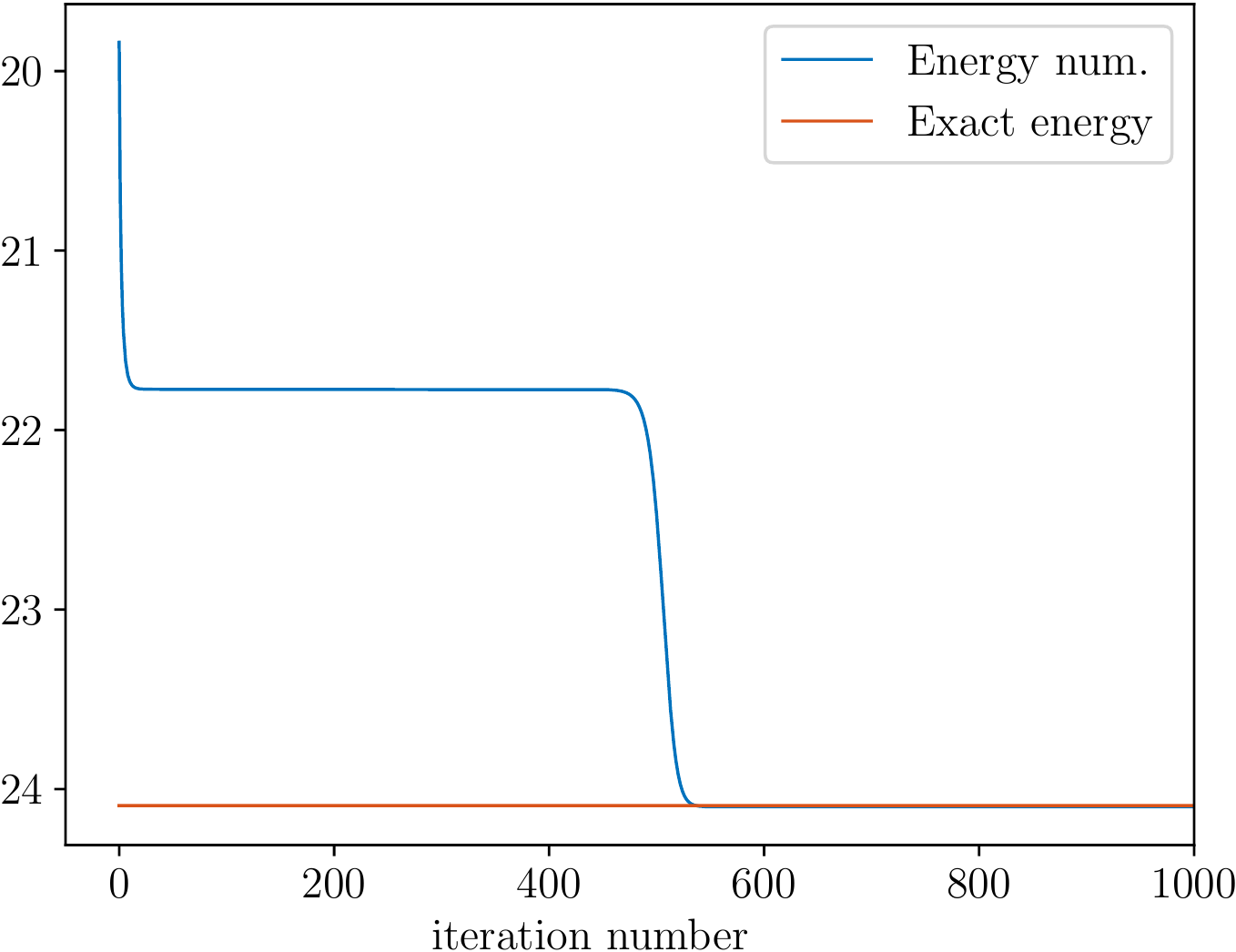} \\
    Odd solution, $\omega=6$ & Asymmetric solution, $\omega=16$
  \end{tabular}
  \caption{Evolution of the energy when computing the ground state of~\eqref{eq:nls2} with $\delta'$ condition compared to $E_{\delta'}(\phi_{\delta'})$.}
  \label{fig:compar_energy_deltaprime}
\end{figure}
\FloatBarrier

\subsection{General non-compact graphs with Kirchhoff condition}
 \label{sec:non-compact}
We consider the computation of ground states on non-compact graphs
not satisfying Assumption~\ref{ass:H} (H). We focus on the signpost and tower of
bubbles graphs. Beside the fact that they exist, little is known
about minimizers. Our numerical algorithm is an easy to use tool
to provide conjectures on the qualitative behavior of ground states on
metric graphs.

\subsubsection{Signpost graph} We now consider a signpost graph (see
Figure~\ref{fig:signpost_bubbles}). We wish to compute a stationary state of the
NLS equation~\eqref{eq:nls2}. The graph has the
following dimensions: the line segment is equal to $2$ and the perimeter of the
loop is $4$. The initial mass is taken to be $1$. Concerning the length of the main line (supposedly very large), it
is set to $100$. On the discretization side, we set the total number of grid
points to $5000$. Furthermore, the time step is fixed to $10^{-2}$ with a total
number of iteration equal to $5000$. 

The resulting stationary state is shown in
Figures~\ref{fig:sol_signpost} and~\ref{fig:sol_signpost_zoom}. We can see that it is localized in the
loop and the line segment and decreases slowly along the main
line. This is consistent with \cite{AdSeTi15b,AdSeTi16}.

\begin{figure}[htbp!]
  \centering
    \includegraphics[width=.7\textwidth]{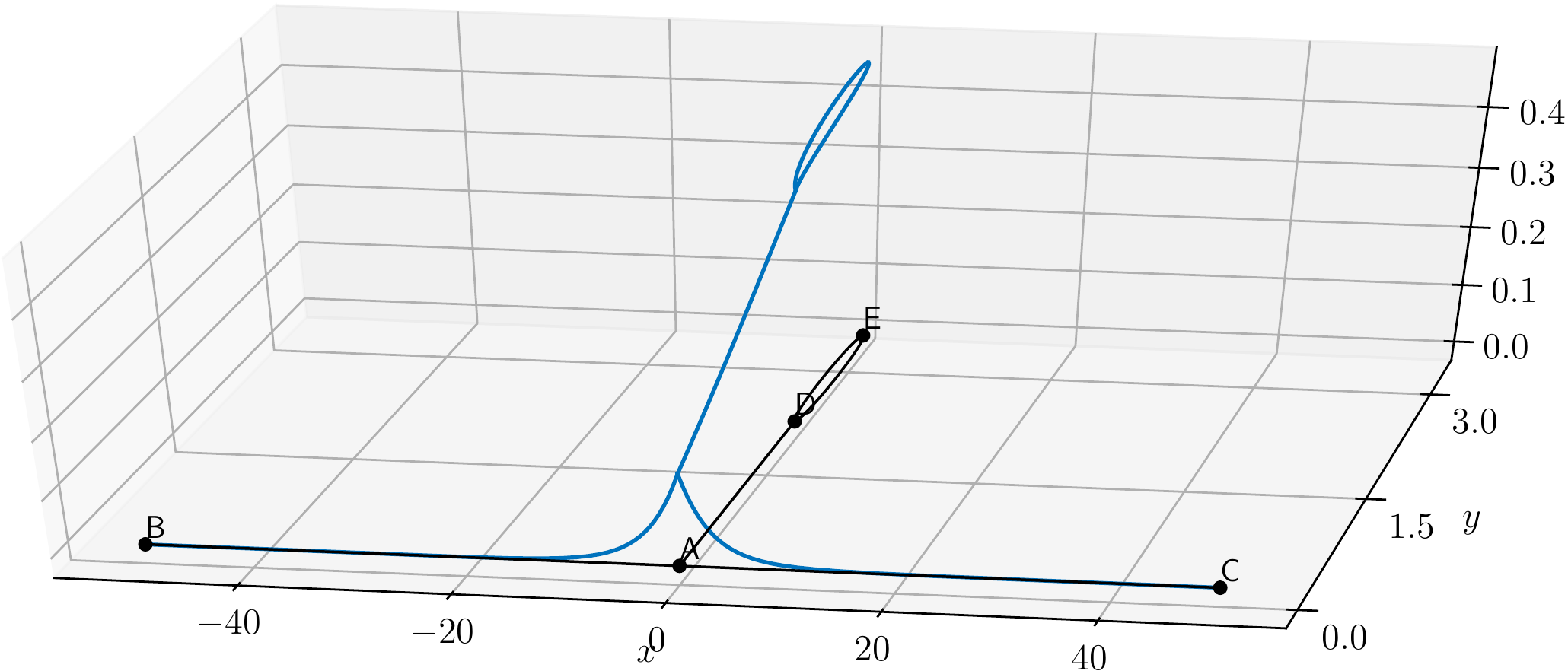}
  \caption{The numerical ground state on the signpost graph}
  \label{fig:sol_signpost}
\end{figure}

\begin{figure}[htbp!]
  \centering
    \includegraphics[width=.7\textwidth]{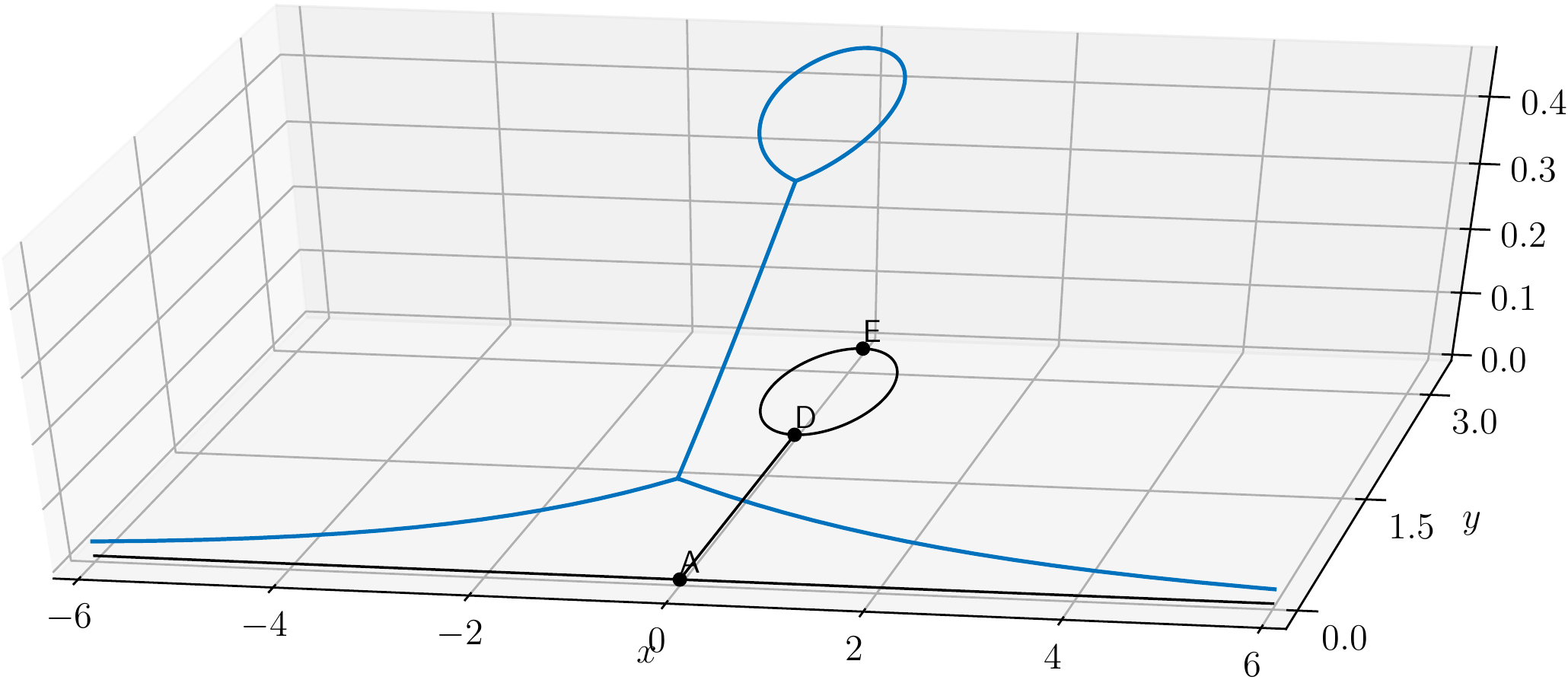}
  \caption{Zoom of the numerical ground state on the signpost graph}
  \label{fig:sol_signpost_zoom}
\end{figure}

\FloatBarrier
\subsubsection{Tower of bubbles graph} Finally, we have computed a stationary
state for the tower of bubbles graph, with $2$ bubbles. The graph is
characterized by the following dimensions: the top bubble has a perimeter of $8$
and the one of the bottom loop is set to $4$. The main line, which is suppose to
be very large, has a length of $100$. The initial mass is taken to be $1$. The space discretization is set by fixing
a total number of grid points to $10000$ and, for the time discretization, we
have the time step set to $10^{-2}$ for a total number of iterations of
$10000$. 

We obtain the stationary state depicted in
Figure~\ref{fig:sol_bubbles} and~\ref{fig:sol_bubbles_zoom}. It is
clear that, as for the signpost graph, the ground state is localized in the
two bubbles and decreases slowly along the main line. Again,  this is consistent with \cite{AdSeTi15b,AdSeTi16}.

\begin{figure}[htbp!]
  \centering
    \includegraphics[width=.7\textwidth]{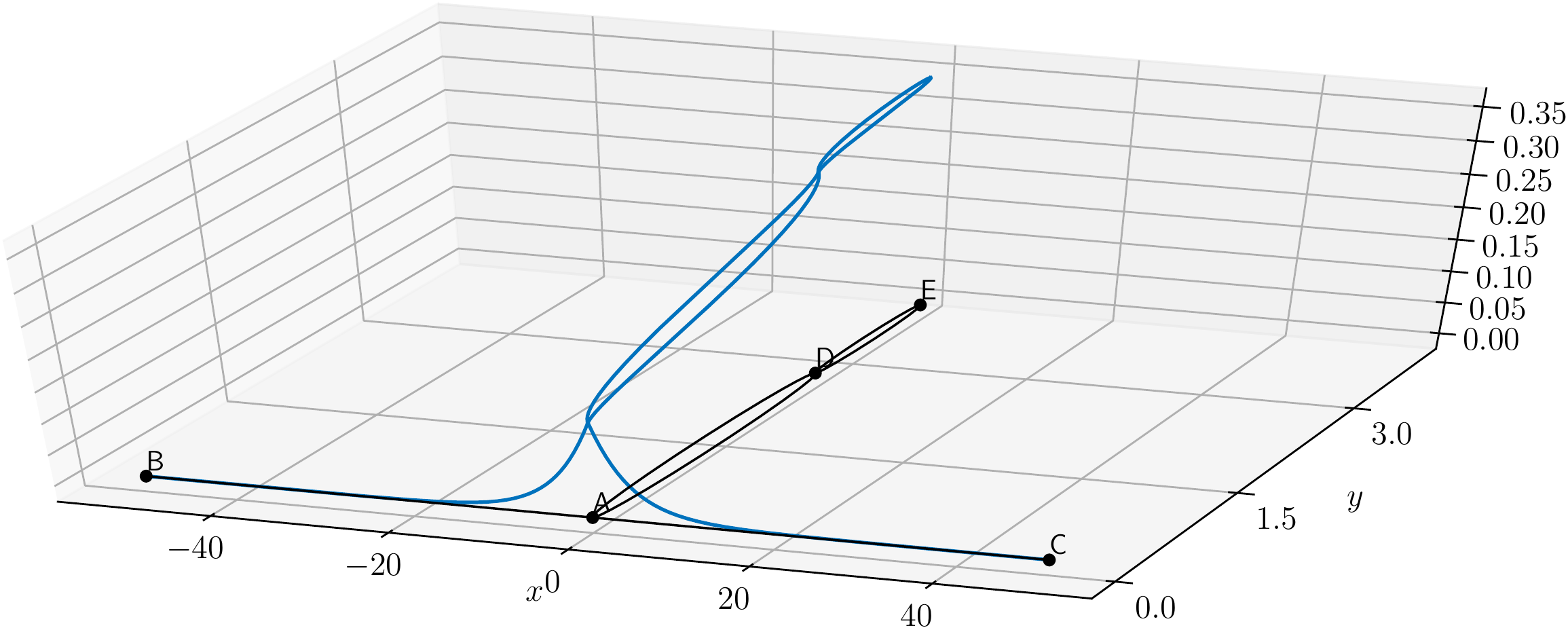}
  \caption{The numerical ground state on the tower of bubbles graph}
  \label{fig:sol_bubbles}
\end{figure}

\begin{figure}[htbp!]
  \centering
    \includegraphics[width=.7\textwidth]{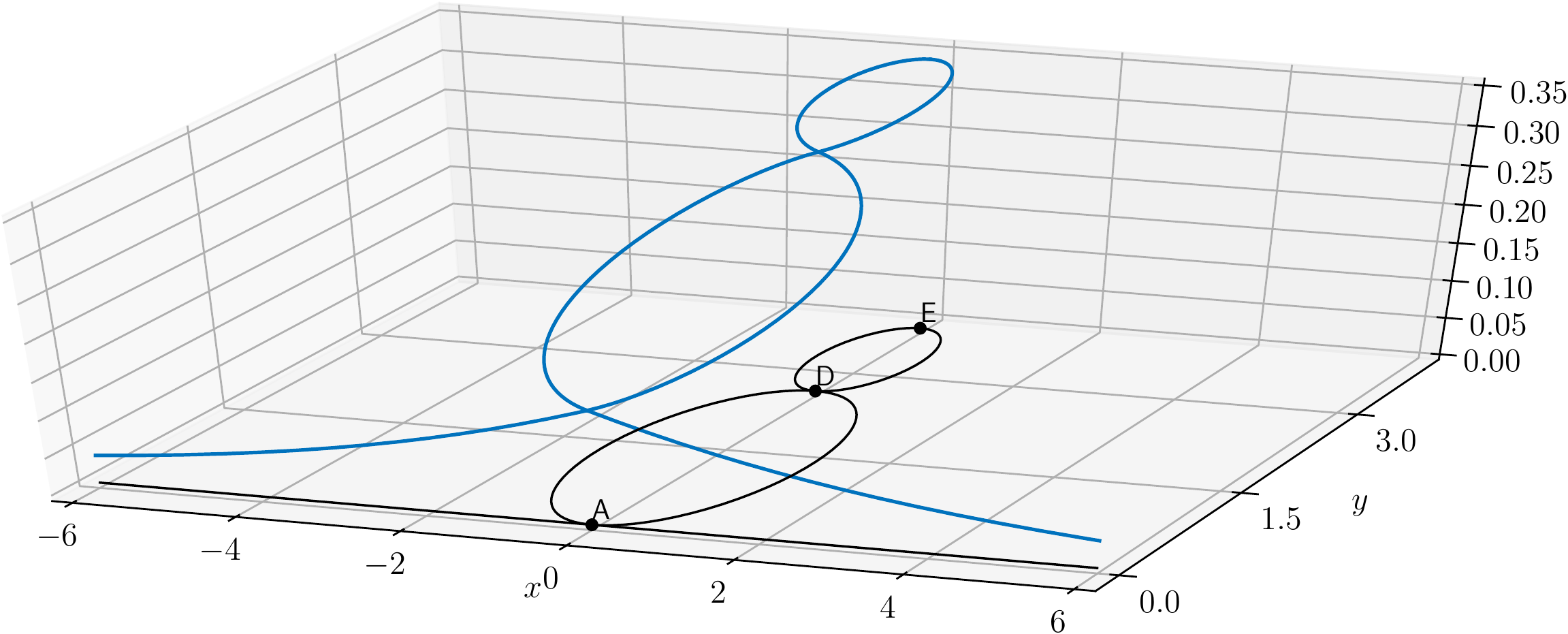}
  \caption{Zoom of the numerical ground state on the tower of bubbles graph}
  \label{fig:sol_bubbles_zoom}
\end{figure}

\bibliographystyle{abbrv}
\bibliography{master}

\def\cprime{$'$}
\begin{thebibliography}{10}

\bibitem{AdBoRu20}
R.~Adami, F.~Boni, and A.~Ruighi.
\newblock {Non-Kirchhoff Vertices and Nonlinear Schrödinger Ground States on
  Graphs}.
\newblock {\em Mathematics}, 8(4), 2020.

\bibitem{AdCaFiNo12}
R.~Adami, C.~Cacciapuoti, D.~Finco, and D.~Noja.
\newblock On the structure of critical energy levels for the cubic focusing
  {NLS} on star graphs.
\newblock {\em J. Phys. A}, 45(19):192001, 7, 2012.

\bibitem{adami2012stationary}
R.~Adami, C.~Cacciapuoti, D.~Finco, and D.~Noja.
\newblock {Stationary states of NLS on star graphs}.
\newblock {\em EPL (Europhysics Letters)}, 100(1):10003, 2012.

\bibitem{AdCaFiNo14}
R.~Adami, C.~Cacciapuoti, D.~Finco, and D.~Noja.
\newblock Variational properties and orbital stability of standing waves for
  {NLS} equation on a star graph.
\newblock {\em J. Differential Equations}, 257(10):3738--3777, 2014.

\bibitem{AdCaFiNo16}
R.~Adami, C.~Cacciapuoti, D.~Finco, and D.~Noja.
\newblock Stable standing waves for a {NLS} on star graphs as local minimizers
  of the constrained energy.
\newblock {\em J. Differential Equations}, 260(10):7397--7415, 2016.

\bibitem{AdNo09}
R.~Adami and D.~Noja.
\newblock Existence of dynamics for a 1{D} {NLS} equation perturbed with a
  generalized point defect.
\newblock {\em J. Phys. A}, 42(49):495302, 19, 2009.

\bibitem{AdNo13}
R.~Adami and D.~Noja.
\newblock Stability and symmetry-breaking bifurcation for the ground states of
  a {NLS} with a {$\delta'$} interaction.
\newblock {\em Comm. Math. Phys.}, 318(1):247--289, 2013.

\bibitem{AdNoVi13}
R.~Adami, D.~Noja, and N.~Visciglia.
\newblock Constrained energy minimization and ground states for {NLS} with
  point defects.
\newblock {\em Discrete Contin. Dyn. Syst. Ser. B}, 18(5):1155--1188, 2013.

\bibitem{AdSeTi15b}
R.~Adami, E.~Serra, and P.~Tilli.
\newblock N{LS} ground states on graphs.
\newblock {\em Calc. Var. Partial Differential Equations}, 54(1):743--761,
  2015.

\bibitem{AdSeTi16}
R.~Adami, E.~Serra, and P.~Tilli.
\newblock Threshold phenomena and existence results for {NLS} ground states on
  metric graphs.
\newblock {\em J. Funct. Anal.}, 271(1):201--223, 2016.

\bibitem{AdSeTi17a}
R.~Adami, E.~Serra, and P.~Tilli.
\newblock Negative energy ground states for the {$L^2$}-critical {NLSE} on
  metric graphs.
\newblock {\em Comm. Math. Phys.}, 352(1):387--406, 2017.

\bibitem{AdSeTi17b}
R.~Adami, E.~Serra, and P.~Tilli.
\newblock Nonlinear dynamics on branched structures and networks.
\newblock {\em Riv. Math. Univ. Parma (N.S.)}, 8(1):109--159, 2017.

\bibitem{AlGeHoHo88}
S.~Albeverio, F.~Gesztesy, R.~Hoegh-Krohn, and H.~Holden.
\newblock {\em Solvable models in quantum mechanics}.
\newblock Texts and Monographs in Physics. Springer-Verlag, New York, 1988.

\bibitem{bao2007}
W.~Bao.
\newblock {Ground states and dynamics of rotating Bose-Einstein condensates}.
\newblock In C.~Cercignani and E.~Gabetta, editors, {\em Transport Phenomena
  and Kinetic Theory. Modeling and Simulation in Science, Engineering and
  Technology.}, Modeling and Simulation in Science, Engineering and Technology,
  pages 215--255. Birkhäuser Boston, 2007.

\bibitem{BaDu04}
W.~Bao and Q.~Du.
\newblock Computing the ground state solution of {B}ose-{E}instein condensates
  by a normalized gradient flow.
\newblock {\em SIAM J. Sci. Comput.}, 25(5):1674--1697, 2004.

\bibitem{BeKu13}
G.~Berkolaiko and P.~Kuchment.
\newblock {\em Introduction to quantum graphs}, volume 186 of {\em Mathematical
  Surveys and Monographs}.
\newblock American Mathematical Society, Providence, RI, 2013.

\bibitem{Grafidi}
C.~Besse, R.~Duboscq, and S.~Le~Coz.
\newblock Grafidi.
\newblock {\em PLMlab repository},
  \url{https://plmlab.math.cnrs.fr/cbesse/grafidi}, 2021.

\bibitem{2021_BDL_Arxiv}
C.~Besse, R.~Duboscq, and S.~Le~Coz.
\newblock Numerical simulations on nonlinear quantum graphs with the grafidi
  library, 2021.

\bibitem{CaLi82}
T.~Cazenave and P.-L. Lions.
\newblock Orbital stability of standing waves for some nonlinear
  {S}chr\"odinger equations.
\newblock {\em Comm. Math. Phys.}, 85(4):549--561, 1982.

\bibitem{DaZu06}
R.~D\'{a}ger and E.~Zuazua.
\newblock {\em Wave propagation, observation and control in {$1\text{-}d$}
  flexible multi-structures}, volume~50 of {\em Math\'{e}matiques \&
  Applications (Berlin) [Mathematics \& Applications]}.
\newblock Springer-Verlag, Berlin, 2006.

\bibitem{DeGeRo15}
S.~De~Bi\`evre, F.~Genoud, and S.~Rota~Nodari.
\newblock Orbital stability: analysis meets geometry.
\newblock In {\em Nonlinear optical and atomic systems}, volume 2146 of {\em
  Lecture Notes in Math.}, pages 147--273. Springer, Cham, 2015.

\bibitem{DeRo19}
S.~De~Bi{\`e}vre and S.~Rota~Nodari.
\newblock Orbital stability via the energy--momentum method: The case of higher
  dimensional symmetry groups.
\newblock {\em Archive for Rational Mechanics and Analysis}, 231(1):233--284,
  2019.

\bibitem{DOVETTA2020107352}
S.~Dovetta, E.~Serra, and P.~Tilli.
\newblock Uniqueness and non–uniqueness of prescribed mass nls ground states
  on metric graphs.
\newblock {\em Advances in Mathematics}, 374:107352, 2020.

\bibitem{FaJe18}
E.~Faou and T.~J\'{e}z\'{e}quel.
\newblock Convergence of a normalized gradient algorithm for computing ground
  states.
\newblock {\em IMA J. Numer. Anal.}, 38(1):360--376, 2018.

\bibitem{FuJe08}
R.~Fukuizumi and L.~Jeanjean.
\newblock Stability of standing waves for a nonlinear {S}chr\"odinger equation
  with a repulsive {D}irac delta potential.
\newblock {\em Discrete Contin. Dyn. Syst.}, 21(1):121--136, 2008.

\bibitem{FuOhOz08}
R.~Fukuizumi, M.~Ohta, and T.~Ozawa.
\newblock Nonlinear {S}chr\"odinger equation with a point defect.
\newblock {\em Ann. Inst. H. Poincar\'e Anal. Non Lin\'eaire}, 25(5):837--845,
  2008.

\bibitem{Go19}
R.~H. Goodman.
\newblock N{LS} bifurcations on the bowtie combinatorial graph and the dumbbell
  metric graph.
\newblock {\em Discrete Contin. Dyn. Syst.}, 39(4):2203--2232, 2019.

\bibitem{GoodmanLib}
R.~H. Goodman.
\newblock Quantum graph package.
\newblock \url{https://github.com/manroygood/Quantum-Graphs}, 2020.

\bibitem{GoHoWe04}
R.~H. Goodman, P.~J. Holmes, and M.~I. Weinstein.
\newblock Strong {NLS} soliton-defect interactions.
\newblock {\em Phys. D}, 192(3-4):215--248, 2004.

\bibitem{GrShSt87}
M.~Grillakis, J.~Shatah, and W.~Strauss.
\newblock Stability theory of solitary waves in the presence of symmetry. {I}.
\newblock {\em J. Funct. Anal.}, 74(1):160--197, 1987.

\bibitem{GrShSt90}
M.~Grillakis, J.~Shatah, and W.~A. Strauss.
\newblock Stability theory of solitary waves in the presence of symmetry. {II}.
\newblock {\em J. Func. Anal.}, 94(2):308--348, 1990.

\bibitem{GuLeTs17}
S.~Gustafson, S.~Le~Coz, and T.-P. Tsai.
\newblock Stability of periodic waves of 1{D} cubic nonlinear {S}chr\"odinger
  equations.
\newblock {\em Appl. Math. Res. Express. AMRX}, 2:431--487, 2017.

\bibitem{Ho19}
M.~Hofmann.
\newblock An existence theory for nonlinear equations on metric graphs via
  energy methods, 2019.

\bibitem{IaLeRo17}
I.~Ianni, S.~Le~Coz, and J.~Royer.
\newblock On the {C}auchy problem and the black solitons of a singularly
  perturbed {G}ross-{P}itaevskii equation.
\newblock {\em SIAM J. Math. Anal.}, 49(2):1060--1099, 2017.

\bibitem{KaMaPeXi20}
A.~Kairzhan, R.~Marangell, D.~E. Pelinovsky, and K.~L. Xiao.
\newblock Standing waves on a flower graph, 2021.

\bibitem{KaPeGo19}
A.~Kairzhan, D.~E. Pelinovsky, and R.~H. Goodman.
\newblock Drift of spectrally stable shifted states on star graphs.
\newblock {\em SIAM J. Appl. Dyn. Syst.}, 18(4):1723--1755, 2019.

\bibitem{LeFuFiKsSi08}
S.~Le~Coz, R.~Fukuizumi, G.~Fibich, B.~Ksherim, and Y.~Sivan.
\newblock Instability of bound states of a nonlinear {S}chr\"odinger equation
  with a {D}irac potential.
\newblock {\em Phys. D}, 237(8):1103--1128, 2008.

\bibitem{Lu80}
G.~Lumer.
\newblock Connecting of local operators and evolution equations on networks.
\newblock In {\em Potential theory, {C}openhagen 1979 ({P}roc. {C}olloq.,
  {C}openhagen, 1979)}, volume 787 of {\em Lecture Notes in Math.}, pages
  219--234. Springer, Berlin, 1980.

\bibitem{Lu95}
A.~Lunardi.
\newblock {\em Analytic semigroups and optimal regularity in parabolic
  problems}.
\newblock Modern Birkh\"{a}user Classics. Birkh\"{a}user/Springer Basel AG,
  Basel, 1995.
\newblock [2013 reprint of the 1995 original] [MR1329547].

\bibitem{MaPe16}
J.~L. Marzuola and D.~E. Pelinovsky.
\newblock Ground {S}tate on the {D}umbbell {G}raph.
\newblock {\em Appl. Math. Res. Express. AMRX}, 2016(1):98--145, 2016.

\bibitem{NaSoMaSa11}
K.~Nakamura, Z.~A. Sobirov, D.~U. Matrasulov, and S.~Sawada.
\newblock Transport in simple networks described by an integrable discrete
  nonlinear schr\"odinger equation.
\newblock {\em Phys. Rev. E}, 84:026609, Aug 2011.

\bibitem{Ni85}
S.~Nicaise.
\newblock Some results on spectral theory over networks, applied to nerve
  impulse transmission.
\newblock In {\em Orthogonal polynomials and applications ({B}ar-le-{D}uc,
  1984)}, volume 1171 of {\em Lecture Notes in Math.}, pages 532--541.
  Springer, Berlin, 1985.

\bibitem{No14}
D.~Noja.
\newblock Nonlinear {S}chr\"odinger equation on graphs: recent results and open
  problems.
\newblock {\em Philos. Trans. R. Soc. Lond. Ser. A Math. Phys. Eng. Sci.},
  372(2007):20130002, 20, 2014.

\bibitem{noja2020standing}
D.~Noja and D.~E. Pelinovsky.
\newblock Standing waves of the quintic nls equation on the tadpole graph.
\newblock {\em Calculus of Variations and Partial Differential Equations},
  59(5):1--31, 2020.

\bibitem{PeSc17}
D.~Pelinovsky and G.~Schneider.
\newblock Bifurcations of standing localized waves on periodic graphs.
\newblock {\em Ann. Henri Poincar\'{e}}, 18(4):1185--1211, 2017.

\bibitem{PiSoVe20}
D.~Pierotti, N.~Soave, and G.~Verzini.
\newblock Local minimizers in absence of ground states for the critical nls
  energy on metric graphs.
\newblock {\em Proceedings of the Royal Society of Edinburgh: Section A
  Mathematics}, page 1–29, 2020.

\bibitem{QuSo19}
P.~Quittner and P.~Souplet.
\newblock {\em Superlinear parabolic problems}.
\newblock Birkh\"{a}user Advanced Texts: Basler Lehrb\"{u}cher. [Birkh\"{a}user
  Advanced Texts: Basel Textbooks]. Birkh\"{a}user/Springer, Cham, 2019.
\newblock Blow-up, global existence and steady states, Second edition of [
  MR2346798].

\bibitem{SaBaMaKe18}
K.~K. Sabirov, D.~B. Babajanov, D.~U. Matrasulov, and P.~G. Kevrekidis.
\newblock Dynamics of dirac solitons in networks.
\newblock {\em Journal of Physics A: Mathematical and Theoretical},
  51(43):435203, sep 2018.

\bibitem{SoBaMaNaUe16}
Z.~Sobirov, D.~Babajanov, D.~Matrasulov, K.~Nakamura, and H.~Uecker.
\newblock Sine-gordon solitons in networks: Scattering and transmission at
  vertices.
\newblock {\em {EPL} (Europhysics Letters)}, 115(5):50002, sep 2016.

\bibitem{SoMaSaSaNa10}
Z.~Sobirov, D.~Matrasulov, K.~Sabirov, S.~Sawada, and K.~Nakamura.
\newblock Integrable nonlinear schr\"odinger equation on simple networks:
  Connection formula at vertices.
\newblock {\em Phys. Rev. E}, 81:066602, Jun 2010.

\bibitem{We85}
M.~I. Weinstein.
\newblock Modulational stability of ground states of nonlinear {S}chr\"odinger
  equations.
\newblock {\em SIAM J. Math. Anal.}, 16:472--491, 1985.

\bibitem{YuSaEhMa19a}
J.~Yusupov, K.~Sabirov, M.~Ehrhardt, and D.~Matrasulov.
\newblock Transparent quantum graphs.
\newblock {\em Physics Letters A}, 383(20):2382--2388, 2019.

\bibitem{YuSaAsEhMa20}
J.~R. Yusupov, K.~K. Sabirov, Q.~U. Asadov, M.~Ehrhardt, and D.~U. Matrasulov.
\newblock Dirac particles in transparent quantum graphs: Tunable transport of
  relativistic quasiparticles in branched structures.
\newblock {\em Phys. Rev. E}, 101:062208, Jun 2020.

\bibitem{YuSaEhMa19b}
J.~R. Yusupov, K.~K. Sabirov, M.~Ehrhardt, and D.~U. Matrasulov.
\newblock Transparent nonlinear networks.
\newblock {\em Phys. Rev. E}, 100:032204, Sep 2019.

\end{thebibliography}

\end{document}